\documentclass[11pt,reqno]{amsart}
 \usepackage{amsmath,amsfonts,amssymb,amsthm,mathrsfs}
 \usepackage{hyperref}
 \usepackage{bm}
 \usepackage{color}
 \usepackage{enumerate}
 \usepackage{esint}
\usepackage{graphicx}
\usepackage{graphics}
\usepackage{geometry}
\usepackage[all]{xy}

\usepackage{tikz}

 \newtheorem{theorem}{Theorem}[section]
 \newtheorem{lemma}[theorem]{Lemma}
 \newtheorem{corollary}[theorem]{Corollary}
 \newtheorem{proposition}[theorem]{Proposition}

 \theoremstyle{definition}
 \newtheorem{definition}[theorem]{Definition}

 \newtheorem{example}[theorem]{Example}

 \newtheorem{remark}[theorem]{Remark}

\numberwithin{equation}{section}

\newcommand{\p}{\partial}
\newcommand{\bg}{\bar{g}}
\newcommand{\bx}{\bm{x}}
\newcommand{\bo}{\bm{0}}
\newcommand{\dC}{\mathbb{C}}

\newcommand{\dN}{\mathbb{N}}
\newcommand{\fa}{\mathfrak{a}}

\newcommand{\sA}{\mathscr{A}}
\newcommand{\sB}{\mathscr{B}}
\newcommand{\sU}{\mathscr{U}}

\newcommand{\cA}{\mathcal{A}}
\newcommand{\cB}{\mathcal{B}}
\newcommand{\cC}{\mathcal{C}}

\newcommand{\cE}{\mathcal{E}}
\newcommand{\cH}{\mathcal{H}}
\newcommand{\cI}{\mathcal{I}}
\newcommand{\cJ}{\mathcal{J}}

\newcommand{\cL}{\mathcal{L}}
\newcommand{\cN}{\mathcal{N}}
\newcommand{\cT}{\mathcal{T}}
\newcommand{\cO}{\mathcal{O}}
\newcommand{\cS}{\mathcal{S}}

\newcommand{\cZ}{\mathcal{Z}}
\newcommand{\dH}{\mathbb{H}}
\newcommand{\dQ}{\mathbb{Q}}
\newcommand{\dR}{\mathbb{R}}

\newcommand{\dZ}{\mathbb{Z}}

\newcommand{\fA}{\mathfrak{A}}

\newcommand{\fC}{\mathfrak{C}}
\newcommand{\fG}{\mathfrak{G}}
\newcommand{\fH}{\mathfrak{H}}

\newcommand{\fS}{\mathfrak{S}}
\newcommand{\fX}{\mathfrak{X}}
\newcommand{\fm}{\mathfrak{m}}

\newcommand{\dsg}{d\sigma_{\bg}}
\newcommand{\dsj}{d\sigma_{\bg_j}}

\newcommand{\dst}{d\sigma_t}
\newcommand{\uC}{\underline{\mathcal{C}}}
\newcommand{\uS}{\underline{\mathcal{S}}}

\newcommand{\ud}{\underline{\delta}}

\newcommand{\vr}{\varrho}
\newcommand{\wt}{\widetilde}
\newcommand{\tU}{\widetilde{U}}

\DeclareMathOperator{\Div}{div}

\DeclareMathOperator{\dvol}{dvol}

\DeclareMathOperator{\Inj}{Inj}

 \DeclareMathOperator{\loc}{loc}
  \DeclareMathOperator{\Min}{Min}
 
   \DeclareMathOperator{\PV}{P.V.}
  \DeclareMathOperator{\Rea}{Re}
 \DeclareMathOperator{\Res}{Res}
\DeclareMathOperator{\Ric}{Ric}

\DeclareMathOperator{\SO}{SO}
\DeclareMathOperator{\Span}{Span}

\DeclareMathOperator{\Tr}{Tr}
\DeclareMathOperator{\Vol}{Vol}

\begin{document}
 \title[Singular sets of degenerate and nonlocal elliptic equations]{The singular sets of degenerate and nonlocal elliptic equations on Poincar\'e-Einstein manifolds}

\author{Xumin Jiang}
\address{Department of Mathematics, Fordham University, Bronx, NY, USA, 10458}
\email{xjiang77@fordham.edu}

 \author{Yannick Sire}
  \address{Department of Mathematics, Johns Hopkins University,  Baltimore, MD, USA, 21218}
 \email{ysire1@jhu.edu}

 \author{Ruobing Zhang}
  \address{Department of Mathematics, Princeton University, Princeton, NJ, USA, 08544}
 \email{ruobingz@princeton.edu}

   \thanks{YS is partially supported by 
NSF Grant DMS-2154219; RZ is supported by NSF Grants DMS-1906265 and DMS-2304818.}

 \maketitle
  \begin{abstract}
   The main objects of this paper include some degenerate and nonlocal elliptic operators which naturally arise in the conformal invariant theory of Poincar\'e-Einstein manifolds. These operators generally reflect the correspondence between the Riemannian geometry of a complete Poincar\'e-Einstein manifold and the conformal geometry of its associated conformal infinity. 
   In this setting, we develop the quantitative differentiation theory that includes quantitative stratification for the singular set and Minkowski type estimates for the (quantitatively) stratified singular sets. All these, together with a new $\epsilon$-regularity result for degenerate/singular elliptic operators on Poincar\'e-Einstein manifolds, lead to uniform Hausdorff measure estimates for the singular sets. Furthermore, the main results in this paper provide a delicate synergy between the geometry of Poincar\'e-Einstein manifolds and the elliptic theory of associated degenerate elliptic operators.

  \end{abstract}

 \tableofcontents
 
 \section{Introduction}
 
 The main purpose of this paper is to establish a series of new dimension bounds and geometric measure estimates for the singular sets of a typical class of degenerate and nonlocal elliptic equations in the context of Poincar\'e-Einstein manifolds. These estimates are new even in the Euclidean model case. 
First, let us recall the backgrounds and major strategies of several aspects employed in our studies.

 \subsection{Background and setup}

 A fundamental theme in the analysis of elliptic PDEs is to understand the structure, particularly the geometric profiles, of their solutions.
 It is a crucial to study the structure of the singular set of these solutions.
 The implicit function theorem and Sard's theorem provide a qualitative characterization for   smooth functions: for a smooth function $f$, its level set $\Sigma$ is a smooth hypersurface except at the critical set of $f$, which has measure zero. In other words, the critical points contribute to the topological and geometric complexity of $\Sigma$, and consequently, of $f$ itself. To obtain more refined geometric information,  it is necessary to investigate the structure of the critical set or singular set.

  In the context of the solutions of elliptic PDEs, the very first step is to consider the model case, namely the Laplace equation and harmonic functions. It is well known that every harmonic on a Euclidean ball must be analytic. Moreover, if a Euclidean harmonic function is homogeneous, then it must be a {\it homogeneous harmonic polynomial}. From the geometric point of view, any homogeneous polynomial is invariant under rescaling at the origin which is also called a {\it Euclidean cone} in the literature.
 Since the critical set of a homogeneous harmonic polynomial is well understood, in the study of general elliptic PDEs, a common strategy  
 is to formulate {\it effective cone structure} of solutions on small scales. That is, one hopes to understand in what sense solutions can be approximated by Euclidean cones, i.e., homogeneous harmonic polynomials. This natural philosophy has endless descendants in different disciplines of  geometry and analysis, such as geometric measure theory, minimal submanifolds, harmonic mappings, mean curvature flow, the metric geometry of Ricci curvature, etc. 
In the specific context of elliptic PDEs, significant progress has been made in recent decades, with many fundamental works focusing on the measure-theoretic properties of the singular set; see  \cite{DF, Lin, Han, HHL, Hardt-N, Han-Lin-harmonic-maps, Cheeger-Naber-Valtorta, naber-valtorta}. 
 It is worth mentioning that groundbreaking results have been obtained recently in the studies of Yau's Conjecture and Nadirashvili's Conjecture, which bounds the Hausdorff measure of the nodal sets of eigenfunctions and harmonic functions on Riemannian manifolds; see \cite{LM, logu1,logu2}.

In this paper, we are mainly concerned with the singular sets of the solutions for a class of degenerate and nonlocal elliptic equations defined in the context of Poincar\'e-Einstein manifolds. Before presenting our main results, it is worth introducing the geometric background and motivations of this geometric setting.
Extensive studies of Poincar\'e-Einstein manifolds originally come from 
 the AdS/CFT correspondence in theoretic physics which relates the (Anti-de-Sitter) Riemannian geometry of a complete Einstein manifold $(X^{n + 1}, g_+)$ to the conformal (field) theory of the conformal infinity $(M^n , [h])$, where $M^n$ is the topological boundary of $X^{n + 1}$.
 This physical philosophy has continuous impacts on conformal geometry so that studying Poincar\'e-Einstein manifolds has become a very active research direction. 
  On the mathematical side, this topic dates back to the   ambient metric construction by Fefferman-Graham in 1985 (\cite{FG-conformal-invariants, Fefferman-Graham-ambient}), which successfully produced conformal invariants and conformally covariant elliptic operators. In this setting, a typical class of conformally covariant elliptic operators $P_{2\gamma}$ constructed on the conformal manifolds $(M^n, [h])$ are nowadays called {\it fractional GJMS operators} with principal symbol equal to the principal symbol of $(-\Delta)^{\gamma}$ on $\dR^n$; see \cite{CG, Case-Chang, Chang-Yang} for a synthetic background. The operators $P_{2\gamma}$ depend not only on the conformal geometry of $(M^n, [h])$ but also on the Riemannian geometry of  $(X^{n+1}, g_+)$: they exhibit nonlocal nature for generic values of $\gamma$.

 Regarding the construction and localization of nonlocal GJMS operators, there are two different approaches. First,
 in Graham-Zworski's work \cite{GZ}, the operators $P_{2\gamma}$ were discovered to have intimate connections with the geometric scattering theory on the asymptotically hyperbolic Poincar\'e-Einstein manifolds. Roughly speaking, the GJMS operators defined on the conformal infinity $(M^n,[h])$
 can be realized as certain ``Dirichlet-to-Neumann mapping", via a scattering operator coming from a non-degenerate Poisson equation defined on the Poincar\'e-Einstein filling $(X^{n + 1}, g_+)$.
  A powerful advantage of this framework is that conformally covariant operators and conformal invariants of different orders can be treated in a uniform way when $\gamma\in(0,\frac{n}{2})$. The second approach to constructing GJMS operators $P_{2\gamma}$ of low orders $\gamma \in (0,1)$ originated in Caffarelli-Silvestre's work \cite{CS}. The main construction of $P_{2\gamma}$ is realized as a ``Dirichlet-to-Neumann mapping" via a degenerate Laplace equation defined on a conformal compactification $(\overline{X^{n + 1}}, \bg)$ which degenerates precisely on the boundary $M^n$. This approach translates the difficulty in studying the behaviors of non-degenerate operators near 
the infinity of a complete manifold $X^{n + 1}$ to investigating degenerate operators defined on a compact manifold $\overline{X^{n + 1}}$ with boundary $M^n$. Recognizing the equivalence between these two constructions and formulating higher order generalizations represents a major progress in the geometric direction; see \cite{CG} and \cite{Case-Chang} for the details.

To introduce the main results of the paper, we now set up the geometric background. 
Given any integer $n\geq 2$, let $(X^{n + 1}, g_+)$ be a complete asymptotically hyperbolic Poincar\'e-Einstein manifold such that $\Ric_{g_+} \equiv - n g_+$. Let $(M^n, [h])$ be the conformal infinity of $(X^{n + 1}, g_+)$: $M^n$ is the topological boundary of $X^{n + 1}$ and $g_+$ admits a compactification  $\bg = \rho^2 g_+$ on the compact $\overline{X^{n + 1}}$ with boundary $M^n$ so that $\bg|_{M^n} = \bg|_{\rho = 0} = h$. We refer the readers to Section \ref{ss:PE-metrics} for detailed definitions of the Poincar\'e-Einstein manifold.
 As in the aforementioned works \cite{GZ, CG}, on the conformal infinity $M^n$, for any $\gamma\in(0,\frac{n}{2})$, there is a family of formally self-adjoint pseudo differential operators, $P_{2\gamma}$, called fractional GJMS operators.
 We restrict ourselves to the case $\gamma \in (0,1)$. The main reason is that we are using monotonicity methods which are less handy when considering $\gamma >1$.

A function $f$ is called $\gamma$-harmonic on $B_1(p)\subset M^n$ if $P_{2\gamma}(f) = 0$ on $B_1(p)$.
 We also denote 
 \begin{align}\label{e:zero-critical-singular}
\cZ(f)   \equiv \{x\in M^n: f(x) = 0\}, \quad 
\cC(f)   \equiv \{x\in M^n: |\nabla f|(x) = 0\}, \quad 
\cS(f)   \equiv \cZ(f) \cap \cC(f).
 \end{align}
Our primary interest is to estimate the size of the zero/critical/singular sets in \eqref{e:zero-critical-singular}.
Similar to the second-order elliptic case in \cite{Lin, HHL, Cheeger-Naber-Valtorta, naber-valtorta}, the size of those sets essentially rely on the ``degree" of the solutions, more precisely, on the frequency of the solutions.

  In our setting, to make sense of the frequency of a $\gamma$-harmonic function, one needs to first localize the operator $P_{2\gamma}$.
  As we mentioned, Chang and Gonz\'alez managed to localize the operators $P_{2\gamma}$ via the Caffarelli-Silvestre type extension; see \cite{CG}. For our purpose, 
we will choose a canonical conformal compactification, and we specify Fefferman-Graham's metric $\bg \equiv  \vr^2 g_+ = e^{2 w}g_+$ in this paper, where $w$ satisfies the equation $-\Delta_{g_+} w = n$ on $X^{n + 1}$. By \cite{CG} (see also Proposition \ref{p:extension-problem-in-FG-metric}), a function $f\in C^{\infty}(M^n)$ that is $\gamma$-harmonic on $B_1(p)\subset M^n$ is always associated to the boundary value problem, called {\it Caffarelli-Silvestre type extension},  
\begin{align}\label{e:harmonic-extension-PE}
 \begin{cases}
		-\Div_{\bg}(\vr^{\fa} \nabla_{\bg} U) + C_{n,\gamma}\vr^{\fa} R_{\bg} U = 0   & \text{in}\ \overline{X^{n+1}},
		\\
	 U = f  & \text{on}\ M^n = \{\vr = 0\},
	 \\
	 P_{2\gamma} f = 0 & \text{on}\ B_1(p) \subset M^n,
\end{cases}
\end{align}  
where $R_{\bg}$ is the scalar curvature of $\bg$, $\fa \equiv 1 - 2\gamma \in (-1,1)$, and $C_{n,\gamma} > 0$ is a constant depending only on $n$ and $\gamma$. We remark that 
the equation is {\it degenerate} when $\fa \in (0,1)$, and {\it singular} when $\fa \in (-1,0)$. 
For the simplicity of the notion, we just call this equation 
   a degenerate equation. 
In this case, we also have 
\begin{align}
 P_{2\gamma} f =  \frac{2^{2\gamma}\cdot\Gamma(\gamma)}{2\gamma\cdot\Gamma(-\gamma)} 
\lim\limits_{\vr \to 0}\vr^{\fa} \frac{\p U}{\p \vr}, \quad \gamma\in(0,1) ;
\end{align}
see Section \ref{ss:GJMS} for details. Moreover, the extension $U$ is unique in the space $H^{1,\fa}(\overline{X^{n + 1}})$. In particular, if the (partial) compactified metric $\bg$ is the Euclidean metric on $\dR_+^{n + 1}$, then \eqref{e:harmonic-extension-PE} is precisely reduced to Caffarelli-Silvestre's original harmonic extension. 
These operators were first studied in the pioneering works \cite{fks,fkj,fjk2}, where the basic theory related to De Giorgi-Nash-Moser estimates, classical functional weighted inequalities, and elliptic measure is considered.

In this paper, when we refer to a $\gamma$-harmonic function $f$, we always associate it with a pair of functions $(U, f)$ for some $U\in H^{1,\fa}(\overline{X^{n + 1}})$ that satisfies \eqref{e:harmonic-extension-PE}. Now that the zero and singular sets of $f$ have been defined in \eqref{e:zero-critical-singular}, for the extension $U$, we define 
\begin{align}\label{e:U-zero-critical-singular}
\cZ(U)   \equiv \{x_+ \in M^n: U(x_+) = 0\}, \quad 
\cC(U)   \equiv \{x_+ \in M^n: |\nabla U|(x_+) = 0\}, \quad 
\cS(U)   \equiv \cZ(U) \cap \cC(U),	
\end{align}
which are the {\bf  restrictions} of the zero/critical/singular sets of $U$ in $\overline{X^{n + 1}}$ on the boundary $M^n$.
To define the frequency, it is convenient to double the compact manifold $(\overline{X^{n + 1}}, \bg)$ along the boundary $(M^n, h)$, which gives a closed manifold $\fX^{n + 1} \equiv \overline{X^{n + 1}}\bigcup\limits_{M^n}\overline{X^{n + 1}}$. The doubled metric is still denoted as $\bg$; see Section \ref{s:preliminaries} for the regularity of the doubled metric. 
Now for any $x \in B_{1/2}(p) \subset M^n$ and for any $r\in (0,1/2)$, the {\it generalized Almgren's frequency} $\cN_f (x, r)$ is defined as: 
\begin{align}
	\cN_f (x, r) 
\equiv \frac{ \displaystyle{r\int_{\cB_r(x)} \vr^{\fa} \left(|\nabla_{\bg} U|^2 + C_{n,\gamma} R_{\bg} U^2\right) \dvol_{\bg} } }{\displaystyle{ \int_{\p\cB_r(x)} \vr^{\fa} U^2 \dsg} },
\end{align}
where $U$ is the extension of $f$ given in \eqref{e:harmonic-extension-PE} and $\cB_r(x)$ is a metric ball in the doubled space $\fX^{n + 1}$. We also define another quantity, called the {\it normalized Dirichlet energy} of $f$:
\begin{align}
\cE_f (x, r) 
\equiv \frac{ \displaystyle{r\int_{\cB_r(x)} \vr^{\fa} |\nabla_{\bg} U|^2 \dvol_{\bg} } }{\displaystyle{ \int_{\p\cB_r(x)} \vr^{\fa} U^2 \dsg} }.	
\end{align}
Similarly, given any $x_+ \in \cB_{1/2}(p_+)\cap M^n$ and $r\in (0,1/2)$, we define 
\begin{align}
\cE_U (x_+, r) 
\equiv \frac{ \displaystyle{r\int_{\cB_r(x_+)} \vr^{\fa} |\nabla_{\bg} U|^2 \dvol_{\bg} } }{\displaystyle{ \int_{\p\cB_r(x_+)} \vr^{\fa} U^2 \dsg} }.		
\end{align}
The relation between the above frequency and normalized Dirichlet energy will be discussed in Sections \ref{ss:almost-monotonicity} and \ref{ss:quantitative-symmetry-PE}. How their large-scale information passes to smaller scales will be also discussed there.

 \subsection{Main results} 
 \label{ss:main-results}
 
Our main results consist of a series of dimension bounds and geometric measure estimates for the nodal and singular sets of both $U$ and $f$.  
In the following, for any $p\in M^n$,  we denote  $p_+ \equiv \iota(p)$, where $\iota:M^n\hookrightarrow \fX^{n + 1}$ is the canonical inclusion.

The first result
 is a Hausdorff measure estimate on the zero set.  
\begin{theorem}[Nodal set estimate]\label{t: nodal_set_theo}
Let $f\in C^{\infty}(M^n)$ be $\gamma$-harmonic on $B_2(p) \subset M^n$ for some $p\in M^n$. 
For any $\Lambda > 0$, there exists a constant $C = C(\Lambda, n, \gamma, \bg) > 0$ such that 
if $\cE_f(p, 2) \leq \Lambda$, then 
\begin{align}
\begin{split}
\mathcal{H}^n (\mathcal{Z}(U) \cap B_1(p)) 
\leq &\ C, \\ \mathcal{H}^{n - 1}(\mathcal{Z}(f) \cap B_1(p)) \leq &\ C,	
\end{split}
 \end{align}
 where $B_1(p)\subset M^n$ is the unit ball.
 \end{theorem}

The next two theorems provide effective Minkowski-type estimates for the singular sets of a $\gamma$-harmonic function $f$ and their Caffarelli-Silvestre type extension $U$. Note that the two singular sets $\cS(U)$ and $\cS(f)$ exhibit very different behaviors. Such phenomena are particularly intriguing in the non-local setting which distinguishes itself from the second-order elliptic case.  
Therefore, we state these estimates separately. 
 \begin{theorem} 
 [Minkowski-type estimate for extensions] \label{t:volume-estimate-extension} 	
Fix a point $p\in M^n$. Let $f\in C^{\infty}(M^n)$ be $\gamma$-harmonic on $B_2(p) \subset M^n$ and let $U$ be its extension in $\overline{\fX^{n + 1}}$ defined by \eqref{e:harmonic-extension-PE}. 
For any $\Lambda > 0$ and $\tau \in (0, 1)$, there exists $C = C(\tau, \Lambda, n, \gamma, \bg) > 0$ such that if $\cE_f(p, 2) \leq \Lambda$, then
\begin{align}
\Vol_{\bg}(T_r(\mathcal{S}(U)) \cap B_1(p)) \leq C \cdot r^{2 - \tau},	
\end{align}
where $T_r(A)$ is the $r$-tubular neighborhood of $A$ in $\overline{X^{n + 1}}$.
As a consequence, the following Minkowski dimension bound holds,
\begin{align}
\dim_{\Min}(\mathcal{S}(U)\cap B_1(p)) \leq n - 1.	
\end{align}
 \end{theorem}

 Different from $\cS(U)$, the singular set $\cS(f)$ of a $\gamma$-harmonic function $f$ on $M^n$ has a more complicated structure. To understand the subtlety of the structure of $\cS(f)$, we define 
 \begin{align}
 	\uS(f) \equiv \left\{x\in \cS(f): \text{the tangent map} \ T_x(f)\ \text{of}\ f\ \text{at}\ x\ \text{is a harmonic homogeneous polynomial}\right\},
 \end{align}
 and we denote $\fS(f) \equiv \cS(f) \setminus \uS(f)$.
  
 \begin{theorem} 
 [Minkowski type estimate on the boundary] \label{t:volume-estimate-boundary} Fix a point $p\in M^n$ Let $f\in C^{\infty}(M^n)$ be $\gamma$-harmonic on $B_2(p) \subset M^n$. For any $\Lambda > 0$ and $\tau \in (0, 1)$, there exists $C = C(\tau, \Lambda, n, \gamma, \bg) > 0$ such that if $\cE_f(p, 2) \leq \Lambda$, then
\begin{align}
\Vol_{\bg}(T_r(\uS(f)) \cap B_1(p)) \leq C \cdot r^{2 - \tau} \quad \text{and} \quad \Vol_{\bg}(T_r(\fS(f)) \cap B_1(p)) \leq C \cdot r^{1 - \tau},
\end{align}
As a consequence, 
\begin{align}
\dim_{\Min}(\uS(f)\cap B_1(p)) \leq n - 2\quad \text{and} \quad \dim_{\Min}(\fS(f)\cap B_1(p)) \leq n - 1.	
\end{align}

  \end{theorem}

 We remark that the Minkowski-type estimates in the case of second-order uniformly elliptic have been established by Cheeger-Naber-Valtorta \cite{Cheeger-Naber-Valtorta} within the framework of quantitative stratification of the singular set. We will adopt this new technique in the setting of degenerate and nonlocal elliptic equations.

The next theorem gives the Hausdorff measure estimate for the singular set, which requires higher regularity of the background metric. In the theorem below, we will add an extra geometric assumption.

\begin{theorem}[Hausdorff measure estimates]  \label{t:Hausdorff-measure-estimate-(U,f)}
 Given $n\geq 2$, let $(X^{n + 1}, g_+)$ be a Poincar\'e-Einstein manifold with conformal infinity $(M^n, [h])$. Assume that $[h]$
 is obstruction flat when $n$ is even. Let $f$ be $\gamma$-harmonic on $B_2(p) \subset M^n$ and let $U$ be the Caffarelli-Silvestre type extension in \eqref{e:harmonic-extension-PE}.
 For any $\Lambda > 0$, there exists $C = C(\Lambda, n, \gamma, g_+, h) > 0$ such that if $\cE_f(p, 2) \leq \Lambda$, then 
 \begin{enumerate}
 \item the singular set of $U$ yields $\mathcal{H}^{n - 1}(\mathcal{S}(U) \cap B_1(p)) \leq C$.
 \item the singular set of $f$ yields	$\mathcal{H}^{n - 2}(\uS(f) \cap B_1(p)) \leq C$ and
$\mathcal{H}^{n - 1}(\fS(f) \cap B_1(p)) \leq C$.

 \end{enumerate}

 \end{theorem}

\begin{remark} The hyperbolic space $(X^{n + 1}, g_+) \equiv (\dH^{n + 1}, y^{-2}(dy^2 + g_{\dR^n}))$ has a natural (partial) conformal compactification $\overline{X^{n + 1}} \equiv \dR_+^{n + 1}$ equipped with the Euclidean metric. This serves as the blow-up model of the general Poincar\'e-Einstein setting. In this special case, the first named author has proven, in \cite{STT}, the Hausdorff measure estimate for the nodal set (special case of Theorem \ref{t: nodal_set_theo}) and a weaker dimension bound on the singular set, i.e., the Hausdorff dimension upper bound. Theorem \ref{t:volume-estimate-extension} and Theorem \ref{t:volume-estimate-boundary} are much stronger than the previous Hausdorff dimension estimates in \cite{STT} even in the Euclidean case. 
We also notice that all
 the estimates in Theorems \ref{t:volume-estimate-extension}-\ref{t:Hausdorff-measure-estimate-(U,f)} can be strengthened to the critical sets $\cC(U)$ and $\cC(f)$.
 \end{remark}

An important purpose of the present work is to understand the influence of the interior of the Poincar\'e-Einstein manifold on the geometry of $\gamma$-harmonic functions on the conformal infinity. As previously mentioned, this heavily relies on a suitable compactification. However, many of the techniques we use in the present work, and in particular the {\sl quantitative stratification} (see the next section for a brief introduction to it),  can also be used to give a fine description of the nodal set and its singular part away from the boundary at infinity (see [Section 4,\cite{STT})] for such a study in the Euclidean case). In this case, the operator is of course uniformly elliptic (with a constant degenerating with the distance to the infinity) and the previous results are already known in this case. It is however important to mention that in the same vein, one could investigate the following problem, which is {\sl not} the one we consider but much related. Consider the doubled manifold $\fX^{n + 1}$ and the following equation posed in $\cB_1(x_+) \subset \fX^{n + 1}$
 \begin{equation}\label{another_model}
 \Div_{\bg}(\vr^{\fa} \nabla_{\bg} U)=0,
\end{equation}
where $\vr$ has the obvious meaning on $\fX^{n + 1}$ (comparable to the signed distance close enough to the set $\vr=0$). Notice that our purpose goes somehow in the opposite direction: we consider first a $\gamma$-harmonic function, extend it inside $X^{n + 1}$ and then double the manifold and reflect accordingly (evenly along the normal direction) the previous equation. In the Euclidean model case,  a thorough analysis of the nodal sets on the characteristic manifold of solutions of \eqref{another_model} has been investigated in \cite{STT}, which relies heavily on the oddness/evenness with respect to the doubling $\fX^{n + 1}$ of a general solution of the PDE \eqref{another_model}. In line with our results, one can also stratify the singular set of the {\it  trace} of $U$ on $\vr = 0$ and refine the decomposition of the singular set by means of frequencies $\geq 2$ and the symmetry of the associated tangent maps. We would like to point out that the quantitative stratification we implement in the present work would also lead to improvements of the relevant results in \cite{STT}, leading to volume estimates. Of further interest is that, as noted in the previous discussion, the global (nonlocal) effects are also present in this latter investigation. They appear indirectly in the investigation of the singular set of traces of solutions on the {\sl characteristic} manifold $\vr=0$.

\subsection{Outline of the proof and organization of the paper}

In this subsection, we will explain the main technical ingredients and new analytic inputs in the proof of the main results. We will also outline the interplay of these techniques from different areas for proceeding the main arguments.

As well understood in the literature, size estimates for the nodal/singular sets of elliptic equations fundamentally rely on the {\it unique continuation property} of the solutions. This property is based on sufficient regularity in the ellipticity coefficients of the equations. Notice that Lipschitz continuity of the coefficients is weakest regularity assumption that one can impose, as shown in \cite{Plis}, where it is demonstrated that the unique continuation property fails under H\"older continuous coefficients. In our specific context of Poincar\'e-Einstein manifold $(X^{n + 1}, g_+)$, it becomes imperative to select an appropriate conformal compactification with sufficient regularity such that the solutions of our interested problem \eqref{e:harmonic-extension-PE} satisfy the unique continuation. Indeed, the adapted compactified metric constructed in \cite{CG, Case-Chang} has only H\"older regularity up to the boundary when the order index of $P_{2\gamma}$ satisfies $\gamma \in (0,1/2)$, which cannot be chosen for our purpose. In the present paper, we will always specify the Fefferman-Graham compacitification the regularity of which will be discussed in detail in Section \ref{ss:PE-metrics}.

The entire strategy of analyzing the singular set is built upon the {\it conic structure analysis} at small scales. The framework has been well understood and widely applied in numerous areas in geometric analysis. In our specific context, to estimate the size of the singular set $\cS(f)$ or $\cS(U)$ for the solutions of \eqref{e:harmonic-extension-PE}, one seeks for appropriate stratification of the singular set. {\it Classical stratification} relies on the symmetry of the tangent map of a solution at a point. For example, the tangent map of a Euclidean harmonic function $f$ can be understood as the leading term, which is a homogeneous polynomial $P$, in its Taylor expansion at a given point, while the {\it degree of symmetry}, denoted as $k$, of this polynomial refers to the maximal dimension of a hyperplane $\cL\subset \dR^n$ for which $P$ is invariant. To obtain the correct dimension bound $\dim(\cS(f)) \leq n - 2$, one needs to establish the following for $\cS^k(f)$:
\begin{enumerate}
\item $\cS (f) \subset \cS^{n - 2}(f)$, namely any $(n - 1)$-symmetric point is not a critical point of $f$;
\item $\dim(\cS^k(f)) \leq k$ for any $k \leq n - 2$.
\end{enumerate}
Item (1) follows from a simple fact that any single-variable harmonic function must be linear, while item (2) is the key ingredient and relies on certain iterative blow-up arguments in classical works; see \cite{Lin, Han, HHL}
for second order elliptic PDEs and see \cite{STT} for the Euclidean model setting in our specific context.
 
This methodology in the metric geometry of the Ricci curvature has led to a series of fundamentally important works in the field. Pioneering works in this area include \cite{ChC, ChC1, ChC2, ChC3, ChCT, Cheeger-elliptic}. All these works, along with the aforementioned works based on classical stratification contribute to estimating the Hausdorff dimension of the singular set. However, Hausdorff dimension bound is somewhat weak for analysis purposes since a set with {\it a low (even zero) Hausdorff dimension} can be arbitrarily dense. More effective results are provided by Minkowski-type estimates, which not only bound the interested set itself but also its tubular neighborhood of that set. For example, the set of rational numbers $\dQ^n$ as a dense subset of $\dR^n$ has a Hausdorff dimension of $0$, but its Minkowski dimension is equal to $n$.
The fundamental tools used to prove Minkowski-type estimates are based on the {\it theory of quantitative differentiation}. This powerful theory was suggested by Cheeger-Kleiner-Naor in \cite{CKN}, and its geometric formulation was further developed by Cheeger and Naber 
in the context of Ricci curvature; see \cite{ChNa-quantitative}. This development has resulted in a rather widely applicable machinery for establishing strong estimates on the singular set and curvatures on non-collapsing Einstein manifolds and their limits. The methodology of quantitative differentiation theory has also been successfully applied in other geometric analytic contexts and in the analysis of nonlinear PDEs. Furthermore, the fundamental tools in this theory have inspired resolutions of numerous long-standing conjectures in the field; see \cite{ChNa-codim-4, JN, CJN}.  Building on this new framework, the authors obtained a series of effective estimates in the case of second-order uniform elliptic PDEs in \cite{Cheeger-Naber-Valtorta, naber-valtorta}.

In this paper, our work aims to establish the quantitative differentiation theory for degenerate and nonlocal elliptic PDEs on Poincar\'e-Einstein manifolds, where the regularity analysis is much more intricate compared to the classical Euclidean case.
We will adopt Cheeger-Naber-Valtorta's general framework \cite{Cheeger-Naber-Valtorta} to formulate quantitative stratification for the singular set of the solutions of \eqref{e:harmonic-extension-PE}.
This fundamental framework will be combined with a careful analysis on the regularity of the underlying Poincar\'e-Einstein manifolds.
 The foundation of such analysis is the {\it almost monotonicity formula} for the generalized Almgren's frequency in the Poincar\'e-Einstein setting, as stated in  Theorem \ref{t:almost-monotonicity}. This (almost) monotonicity enables us to obtain the conic structure (homogeneous polynomials) on all but finitely many scales, which leads us to prove the Minkowski type estimates for the singular set.

In the next stage, in order to obtain a more refined estimate of the Hausdorff measure for the singular set, one approach is to prove a strong $\epsilon$-regularity result. This result was originally established by Han, Hardt, and Lin in \cite[theorem 4.1]{HHL} in the context of second-order uniformly elliptic PDEs. Essentially, this $\epsilon$-regularity theorem states that if a solution of an elliptic PDE is sufficiently close to a homogeneous harmonic polynomial, then their singular sets exhibit similar behaviors.

To prove such an $\epsilon$-regularity in our specific setting, we need to obtain a precise characterization of the tangent maps, which are {\it homogeneous hypergeometric polynomials} (see Section \ref{ss:epsilon-regularity}). Additionally, the $\epsilon$-regularity in \cite{HHL} requires very high regularity of the ellipticity coefficients, which leads us to carry out delicate analysis on the compactification of the Poincar\'e-Einstein metric. Specifically, when the metric under consideration does not have sufficiently high regularity, we are still able to appropriately adjust the original solution to a more regular one. Our new observation for this aspect is included in Proposition \ref{p:higher-order-approximation-PE}. Taking into account these two technical points, our $\epsilon$-regularity result (Theorem \ref{t:eps-reg-PE}) is essentially novel and allows for lower regularity requirements compared to those in \cite{HHL}, making it of independent  interest.

\vspace{0.3cm}

The remaining sections of the paper are organized as follows. Section 2 provides several preliminary materials, including the basics of Poincar\'e-Einstein manifolds and fractional GJMS operators. It also covers some fundamental regularity properties of the PDEs that are being studied.

Section 3 focuses on the quantitative stratification of the singular set and the main estimates in the model case. Specifically, it considers the scenario where the Poincar\'e-Einstein manifold $(X^{n + 1}, g_+)$ is hyperbolic and its (partial) compactification is Euclidean. In Section 4, the main results are proven in their full generality, significantly extending beyond the model case discussed in Section 3.

\vspace{0.3cm}
{\bf Notations and conventions.} For convenience, we list several notations which are frequently used in the later sections.

\begin{itemize}
\item For any $p\in M^n$, let us denote $p_+ \equiv \iota(p)$ for the  inclusion $\iota: M^n \to \overline{X^{n + 1}}$ or   $\iota: M^n \to \fX^{n+1}$. 

\item Given $n\geq 2$, we often denote $\fm \equiv n + 1$.

\item $\cB_r(x)$ denotes a ball in $\overline{X}^{n+1}$ or $\fX^{n + 1}$, while $B_r(x)$ denotes a ball in $M^n$.

\item Given $x\in M^n$, $\cB_r^+(x)\equiv\cB_r(x) \cap \overline{X^{n+1}}$ denotes the intersection between a ball and the compactified manifold $\overline{X^{n+1}}$.  

\item $\rho^2 g_+$ denotes a general conformal compactification, and $\vr^2 g_+$ denotes Fefferman-Graham's compactification.

\end{itemize}

\vspace{0.3cm}

\subsection{Acknowledgements}

The authors would like to express their gratitude to Qing Han and Fang-Hua Lin
for their interest in this work and for their valuable comments. The third named author would also like to thank Fang Wang for the insightful discussions on the regularity of the Poincar\'e-Einstein metrics, and he is indebted to Jeffrey Case and Alice Chang for valuable communications on the fractional GJMS operators.

  \section{Preliminaries}\label{s:preliminaries}

\subsection{Poincar\'e-Einstein metrics and some regularity results} \label{ss:PE-metrics}

\begin{definition}
A complete Riemannian manifold $(X^{n+1},g_+)$ is called asymptotically hyperbolic Poincar\'e-Einstein if 
$g_+$ satisfies $\Ric_{g_+} \equiv - n g_+$ and admits a conformal compactification in the following sense. 
\begin{enumerate}
	\item $X^{n+1}$ is diffeomorphic to the interior of a compact manifold $\overline{X^{n + 1}}$ with boundary $M^n$;
	\item there is a smooth defining function $\rho:\overline{X^{n + 1}} \to  [0,\infty)$ such that 
	\begin{align}
	\begin{cases}
			\rho > 0  & \text{in}\ X^{n+1},
		\\
		    \rho = 0 & \text{on}\ M^n,
		    \\
				|\nabla\rho| = 1 & \text{on}\ M^n.
	\end{cases}
	\end{align}
\end{enumerate}
In this case, the conformal manifold $(M^n, [h])$ is called the conformal infinity, where $h \equiv \rho^2 g_+|_{\rho = 0}$. 
\end{definition}

On a Poincar\'e-Einstein manifold, the existence of the geodesic defining function near the boundary is well known in the literature now; see, e.g., \cite{Lee-spectrum} or \cite{Gra2000}.
\begin{lemma}[Geodesic defining function]\label{l:geodesic-def-funct}
Let $(M^n,[h])$ be the conformal infinity, then for every $h_0\in[h]$ there exists a unique defining function $y$ in a neighborhood of $M^n$ such that 

$(1)$ $y^2g_+|_{y=0}=h_0$,

$(2)$ there exists $\epsilon>0$
such that $|\nabla y|_{y^2g_+}\equiv 1$ on $M^n\times [0,\epsilon)$. 
\end{lemma}

On $M^n\times (0, \epsilon),$ the metric splits as $g_+ = \frac{dy^2+ g_y}{y^2},$ where $g_y$ satisfies $g_y|_{y=0} = h$. Moreover, in \cite{Gra2000}, the compactified metric $g_y$ yields an expansion in terms of the geodesic defining function. 
If $n$ is odd,
\begin{align}\label{eq-gy1}
g_y = g^{(0)} + \frac{g^{(2)}}{2}y^2 +\ldots + (\text{even powers}) + g^{(n-1)}y^{n-1} + g^{(n)}y^{n}+\ldots,
\end{align}
where $g^{(2i)}$ depends only on $h$ for any $2i \leq  n - 1$, and {\it the global term} $g^{(n)}$ depends on both $g_+$ and $h$. If $n$ is even, 
$g_y$ has an expansion of the form
\begin{align}\label{eq-gy2}
g_y = g^{(0)} + \frac{g^{(2)}}{2}y^2 +\ldots + (\text{even powers}) + g^{(n)}y^{n} + \omega \cdot y^n \log y +  \ldots,
\end{align}
where $g^{(n)}$ is also a global term which depends on both $h$ and $g_+$, and $\omega$ depends only on $h$ and satisfies $\Tr_h(\omega)=0$.
For $n=2$, $\omega=0$, and $g_y$ is smooth on $M^n \times [0, \epsilon)$ in this case.
It was proved by Biquard that $h$ and $g^{(n)}$ completely determine the compactified metric and hence $g_+$; see \cite{Biquard}.

For our analysis, we need a compactified metric $\bg$ which has sufficiently high regularity and provides explicit geometric properties when interacting with the scattering operators introduced in Section \ref{ss:GJMS}.
For this purpose, we choose the compactification constructed by Fefferman-Graham in \cite{FG}. 
Let $(X^{n + 1}, g_+)$ be an asymptotically hyperbolic Poincar\'e-Einstein manifold and let $(M^n, h)$ be its conformal infinity. Fefferman-Graham proved 
\begin{lemma}[\cite{FG}] \label{l:Fefferman-Graham} For any representative $h$ on the conformal infinity $(M^n, [h])$, there exists a  conformal compactification 
  $\bg = \vr^2 g_+ \equiv  e^{2w} g_+$, called the {\it Fefferman-Graham compactification}, that satisfies \begin{align}
	-\Delta_{g_+} w = n \quad  \text{on}\ X^{n + 1},  
\end{align}
and $\bg|_{M^n} = h$. Near the boundary, $w = \log y +O(y^\epsilon)$  for some $\epsilon \in (0, 1)$.
\end{lemma}
 
Now we discuss the expansion of $w$ in the geodesic defining function $y$ near the conformal infinity $M^n$.
It was proved in \cite[theorem 3.1]{FG} that $w$ has a formal expansion \begin{align} \label{e:w-expansion}
w = \log y + \sA + \sB y^n\log y +O(y^n),
\end{align}
where $\sA, \sB \in C^\infty(\overline{X^{n + 1}})$ and
$\sA=O(y^2)$.
It is similar to the expansion of $g_y$ in \eqref{eq-gy1}-\eqref{eq-gy2}  that $\sA\mod O(y^n)$ is even in $y$. So it follows that $\vr$ yields 
 \begin{align}\label{e:vr-expansion}
\vr = e^w = y\exp\left(\sA + \sB y^n\log y +O(y^n)\right).
\end{align}
Note that $\sB\equiv 0$ when $n$ is odd, which implies that both $w$ and $\bg$ are $C^{\infty}$ up to the boundary.  
In general, $\vr \in C^{n - 1, \alpha}(\overline{X^{n + 1}})$ and hence
$\bar g$ 
is $C^{n-1,\alpha}$ up to boundary.  
 The appearance of $\sB$ obstructs $\bg$ to be $C^{\infty}$
up to the boundary.

In summary, we have the following  helpful result on the regularity of Fefferman-Graham's compactified metric is as follows. Note that a stronger result was proved in \cite{wang-zhou-2}.  

\begin{proposition}\label{p:general-regularity-FG} Given $n \geq 2$, let $(X^{n + 1}, g_+)$ be a Poincar\'e-Einstein manifold with $\Ric_{g_+} \equiv - n g_+$ such that it admits a conformal compactification which is $C^{k,\alpha}$ up to the boundary for some $k \geq 3$ and $\alpha\in(0,1)$.
Denote by $\bar{g}$ the Fefferman-Graham compactification. Then the following holds: 
\begin{enumerate}
	\item  if $n$ is odd, then  $\bar{g}$ is $C^{k , \beta}$ up to the boundary for any $\beta\in(0,\alpha)$;
	 	\item if $n$ is even and $n\geq 4$, then $\bar{g}$ is $C^{n - 1 , \beta}$ up to the boundary for any $\beta\in(0,1)$.
\end{enumerate}
\end{proposition}
 
\begin{remark}
In the simplest case $n = 2$, any Poincar\'e-Einstein manifold $(X^3, g_+)$ is hyperbolic and its conformal infinity $(M^2,[h])$ is locally conformally flat. This implies that $(X^3, g_+)$ admits a conformal compactification which is  $C^{\infty}$ up to the boundary. In this case, the Fefferman-Graham compactified metric is also $C^{\infty}$ up to the boundary.
\end{remark}

One can see from item (2) of Proposition \ref{p:general-regularity-FG} that the regularity of $\bar{g}$  does not exceed $C^n$ in general when $n$ is even. A fundamental reason is that the regularity 
of $\bar{g}$ is determined by some geometric invariant, called the {\it obstruction tensor}, of the conformal infinity; see \cite{Graham-Hirachi} for the definition of the obstruction tensor in this context.
If we assume the conformal infinity is {\it obstruction flat}, 
the following proposition follows from a regularity result due to Graham-Hirachi; see \cite[theorem 2.2]{Graham-Hirachi}.

\begin{proposition}\label{p:regularity-obstruction-free}
	Let $n\geq 4$ be even, and let $(X^{n + 1}, g_+)$ be a Poincar\'e-Einstein manifold with $\Ric_{g_+} \equiv - n g_+$ such that it admits a conformal compactification which is $C^{k,\alpha}$ up to the boundary for some $k \geq 3$ and $\alpha\in(0,1)$. Assume that the conformal infinity $(M^n, [h])$
is obstruction flat. Then 
the Fefferman-Graham compactified metric $\bar{g}$ is $C^{k,\beta}$ for any $\beta \in (0,\alpha)$.
\end{proposition}

\begin{remark}
In the case when $n \geq 4$ is even, the term $\sB$ in the expansion \eqref{e:w-expansion} is vanishing if and only if the obstruction tensor of $h$ is vanishing on $M^n$, and $\vr$ is smooth up to the boundary.
\end{remark}

 We finally recall the following well-known result. 

\begin{lemma}\label{geom-boundary}
Let $(X^{n + 1}, g_+)$ be a Poincar\'e-Einstein manifold with a   conformal infinity $(M^n, [h])$. Let $\bar{g}$ be the Fefferman-Graham compactification of $(X^{n + 1}, g_+)$. Then $(M^n, h)$ is totally geodesic in $(\overline{X^{n + 1}}, \bg)$. 
\end{lemma}

The previous lemma allows us to double the manifold $M^n$ inside $X^{n + 1}$ and consider a suitable {\it  even} extension of solutions of the PDEs under consideration. This is crucial for regularity purposes as discussed in Section \ref{ss:regularity-of-solutions}. The whole package will be used in Section \ref{s:results-on-PE}.

\subsection{Scattering and GJMS operators on the Poincar\'e-Einstein manifolds}
\label{ss:GJMS}
 
This subsection will summary background materials regarding the fractional GJMS operators defined on the conformal infinity $(M^n, h)$ of a Poincar\'e-Einstein manifold $(X^{n + 1}, g_+)$. Let us consider the Poisson equation \begin{align}
	(\Delta_{g_+} + s(n-s)) u = 0, \quad s\in\dC,\label{e:Poisson}
\end{align}
on $X^{n+1}$. The interaction between conformal geometry of $(M^n, h)$ and   scattering operators defined in this setting is systematically in \cite{GZ}. 
Let us first describe the structure of the solutions of \eqref{e:Poisson}. It is pointed out in \cite{GZ} that for any $s\in\dC$ that satisfies $s(n-s)\not\in\sigma_{pp}(-\Delta_{g_+})$, 
equation \eqref{e:Poisson} admits a generalized eigenfunction
\begin{align}\label{eq-u-asymp}
	u = F y^{n-s} + G y^s, \quad F,G\in C^{\infty}(X^{n+1}).
	\end{align}
Here when $n$ is odd or $n=2$, $F, G\in C^{\infty}(\overline{X^{n + 1}})$; when $n\geq 4$ is even, $F, G \in C^{n-1, \alpha}(\overline{X^{n + 1}})$ for $\alpha \in (0,1)$. In fact, when $n\geq 4$ even and $h$ is {\it not obtruction-free}, the terms involving $y^n\log y$ appear in the expansion of $F$ and $G$ with respect to $y$.

The scattering operator $S(s)$, by definition, is the following Dirichlet-to-Neumann 
operator: if $f\equiv F|_{y=0}$, then 
\begin{align}
S(s) f \equiv G|_{y=0}.	
\end{align}  
It was proved in \cite{GZ} that, as conformally covariant pseudo-differential operators,  $S(s)$
gives a meromorphic family with respect to $\{s\in\dC:\Rea(s)>\frac{n}{2}\}$ with simple poles in $\frac{n}{2}+\dZ_+\subset \dR$.
Now for any $\gamma\in(0,\frac{n}{2})$, the {\it fractional GJMS operator $P_{2\gamma}$} is defined by
\begin{align}
P_{2\gamma}(f) \equiv P_{2\gamma}[h,g_+](f)\equiv 2^{2\gamma}\frac{\Gamma(\gamma)}{\Gamma(-\gamma)}S\left(\frac{n}{2}+\gamma\right)(f).\label{regular-fractional-operator}
\end{align}
In this case, the nonlocal curvature $Q_{2\gamma}$ of order $2\gamma$ is defined as $Q_{2\gamma} \equiv (\frac{n - 2\gamma}{2})^{-1}P_{2\gamma}(1)$.
Observe that the function $\Gamma(\gamma)$
cancels the simple poles of the scattering operator $S(\frac{n}{2}+\gamma)$ so that we define 
\begin{align}
	P_n   \equiv 2^n \frac{\Gamma(\frac{n}{2})}{\Gamma(-\frac{n}{2})}\Res_{s=n}S(s),
\end{align}
and $Q_n 
\equiv \lim\limits_{\gamma\to n} Q_{2\gamma} \equiv (\frac{n - 2\gamma}{2})^{-1}P_{2\gamma}(1)$. Therefore, $P_{2\gamma}$
is continuous in the range $\gamma\in(0,\frac{n}{2}]$.
There are two important special cases: when $\gamma = 1$,
the operator
$P_2$ coincides with the conformal Laplacian \begin{align}\mathscr{L}_h \equiv - \Delta_h + \frac{n-2}{4(n-1)} R_h;\end{align} 
when $\gamma = 2$, the operator $P_4$ coincides with the Paneitz operator of order $4$.

The pseudo-differential operator $P_{2\gamma}$ has an important conformal covariance property in the following sense: if $\gamma\in (0,\frac{n}{2})$ and $\tilde{h} \equiv v^{\frac{4}{n - 2\gamma}}h$, then
\begin{align}
	P_{2\gamma}[\tilde{h}, g_+](u) = v^{-\frac{n + 2\gamma}{n - 2\gamma}} P_{2\gamma}[h,g_+](uv); \end{align}
if $\gamma = \frac{n}{2}$ and $\tilde{h} \equiv e^{2w} h$, then 
\begin{align}
	e^{nw} \widetilde{Q}_n = Q_n + P_n(w).
\end{align}
This gives a uniform way in studying the conformal geometry involving the operator $P_{2\gamma}$ and makes $P_{2\gamma}$ play a crucial role in the conformal invariance theory. How the operators $P_{2\gamma}$ reveal the geometry and topology of the underlying space still needs further exploration.  
The third named author initiated studies of the topological aspect of the operator $P_{2\gamma}$ and the associated nonlocal curvature $Q_{2\gamma}$; see \cite{Zhang-Kleinian} and \cite{Chen-Zhang-rigidity} for topological and isometric rigidity/classification results in this direction.

\begin{remark}\label{r:bottom-of-spectrum}
	We emphasize that, 
throughout this section, we will assume that   $\lambda_1(-\Delta_{g^+}) > \frac{n^2}{4} - \gamma^2$ which guarantees $s = \frac{n}{2} + \gamma$ satisfies $s(n-s)\not\in \sigma_{pp}(-\Delta_{g_+})$ .
A sufficient condition for this is that the non-negativity of the Yamabe constant of  $(M^n,h)$, i.e., $\mathcal{Y}(M^n, [h]) \geq 0$. In fact, Lee proved that in this case
$\lambda_1(-\Delta_{g_+}) = \frac{n^2}{4}$;
 see \cite{Lee-spectrum}.

\end{remark}

 Let us discuss a different formulation of the fractional GJMS operator. First let us recall an example in the model case. 
 \begin{example}
 [Fractional Laplacian in the hyperbolic case]	
Let $(\dH^{n + 1},g_{_1})$ be the hyperbolic space which admits a partial compactification $(\dR^{n + 1}, g_0)$, where $g_{-1} = y^{-2}(dy^2 + g_{\dR^n})$ is the hyperbolic metric and $g_0 = dy^2 + g_{\dR^n}$ is the Euclidean metric on the upper half space $\dR_+^{n + 1}$. One can show that $P_{2\gamma}[g_0, g_{-1}] = (-\Delta_{g_0})^{\gamma}$, where 
 \begin{align}
 	\left((-\Delta_{g_0})^{\gamma}f\right)(x)
 	\equiv C_{n,\gamma} \cdot \PV\int_{\dR^n}\frac{f(x) - f(\xi)}{|x - \xi|^{n + 2\gamma}}d\xi,
 \end{align}
 where $C_{n,\gamma}$ is some constant depending only on $n$ and $\gamma$.
 By solving an extension problem, Caffarelli and Silvestre have introduced in \cite{CS} an equivalent definition of $(-\Delta_{g_0})^{\gamma}$ when $\gamma\in(0,1)$. For a function $f: \dR^n \to \dR$, one 
 can construct the extension $U\in H^{1,\fa}(\dR_+^{n+1})$ as the unique solution of the following boundary value problem
 \begin{align}
 	\begin{cases}
 		\Div_{g_0}(y^{\fa} \nabla_{g_0} U) = 0 & \text{in} \ \dR_+^{n+1},\\
 		U(0,\cdot) = f & \text{on} \ \dR^n \equiv \dR_+^{n+1}\cap \{y = 0\},
 		 	\end{cases}
 \end{align}
 where $\fa \equiv 1 - 2\gamma \in (-1, 1)$ and  \begin{align*}H^{1,\fa}(\dR_+^{n + 1}) \equiv \overline{C^{\infty}(\dR_+^{n + 1})}^{\|\cdot\|_{H^{1,\fa}}},\quad \|U\|_{H^{1,\fa}(\dR_+^{n+1})}\equiv \left(\int_{\dR_+^{n+1}}y^{\fa} U^2 \dvol_{g_0}+\int_{\dR_+^{n+1}}y^{\fa} |\nabla_{g_0} U|^2 \dvol_{g_0} \right)^{\frac{1}{2}}.\end{align*}
 Caffarelli and Silvestre proved that 
 \begin{align}
 	(-\Delta_{g_0})^{\gamma} f =  \frac{2^{2\gamma}\cdot\Gamma(\gamma)}{2\gamma\cdot\Gamma(-\gamma)}\lim\limits_{y\to 0} y^{\fa}\frac{\p U}{\p y}.
 \end{align}  
  \end{example}

\vspace{0.5cm}

In the general case, Chang and Gonz\'alez obtained in \cite{CG} a similar formulation for the fractional GJMS operator $P_{2\gamma}$. 
 We fix a positive number $\gamma \in (0,1)$. 
 Let $(X^{n+1}, g_+)$ be a Poincar\'e-Einstein manifold with conformal infinity $(M^n, h)$. Let us take a conformal compactification $(\overline{X^{n + 1}}, \rho^2 g_+)$ of $(X^{n+1}, g_+)$ for some defining function $\rho$ such that $\rho^2 g_+|_{M^n} = h$.
 Chang-Gonz\'alez proved an equivalence relation: the equation 
 \begin{align}
 -\Delta_{g_+}u - s(n - s)u = 0\quad \text{in} \ (X^{n + 1}, g_+)	
 \end{align}
is equivalent to 
\begin{align}\label{e:equivalent-to-harmonic-extension}
\begin{split}
	     -\Div_{\bar{g}}(\rho^{\fa}\nabla_{\bg} U) + E(\rho) U = 0, \text{ in } X^{n + 1},
    \\
    \bg \equiv \rho^2 g_+, \quad U \equiv \rho^{s - n}u,
    \end{split}
    \end{align}
 where $\fa \equiv 1 - 2\gamma$, and 
 \begin{align}
 	E(\rho) \equiv  - \Delta_{\bar{g}}(\rho^{\frac{\fa}{2}})\rho^{\frac{\fa}{2}} + \left(\gamma^2 - \frac{1}{4}\right)\rho^{-2 + \fa} + \frac{n-1}{4n}R_{\bar{g}}\rho^{\fa}.	 \label{e:equation-of-E}
 \end{align} 
 Furthermore, in \cite{Case-Chang}, there is an improvement of \cite[theorems 4.3 and 4.7]{Case-Chang}. 
 Let $y$ be the geodesic defining function corresponding to $h$ and taking a defining function $\rho$ for $M^n$ that satisfies 
  \begin{align}
    \rho = y + \Phi y^{1+2\gamma} + o(y^{1+2\gamma}) \label{e:assumption-on-vr}
\end{align}
for some smooth function $\Phi \in C^\infty(M^n)$.
Given $f\in C^\infty(M^n)$, let $U$ be the solution of the boundary value problem
\begin{align}\label{e:U-extension-problem-PE}
\begin{cases}
    -\Div_{\bar{g}}(\rho^{\fa}\nabla U) + E(\rho) U = 0, &\text{ in } X^{n+1}, \\
    U = f &\text{ on } M^n.
\end{cases}
\end{align}
Then 
\begin{align}
    P_{2\gamma} f = \frac{d_\gamma}{2\gamma}\lim_{\rho\rightarrow 0} \rho^\fa \frac{\partial U}{\partial \rho} + \frac{n-2\gamma}{2}d_\gamma \Phi f \label{e:expression-of-GJMS-lower-regularity}
\end{align}
for $d_\gamma \equiv 2^{2\gamma} \frac{\Gamma(\gamma)}{\Gamma(-\gamma)}$.

 In the following lemma, 
we will compute the term \eqref{e:equation-of-E} when $\bg = \vr^2 g_+$ is the Fefferman-Graham compactified metric.

\begin{lemma}\label{l:E}
Let $w$ be the solution of $-\Delta_{g_+} w = n$ that satisfies the asymptotics $w = \log y + O(y^{\epsilon})$ for some $\epsilon > 0$. With respect to Fefferman-Graham's 
 compactification,
\begin{align}
\bar{g} \equiv \vr^2 g_+, \quad \vr \equiv e^w,
\end{align}
we have that 
\begin{align}
E(\vr) = C_{n,\gamma} \cdot \vr^{\fa} \cdot R_{\bar{g}}.	
\end{align}
If $n=2$ or $n$ is odd, then $R_{\bar g}\in C^\infty(\overline{X^{n + 1}})$ is smooth. If $n\geq 4$ is even, then $R_{\bar g}\in C^{n-3,\alpha}(\bar X^{n+1})$.
\end{lemma}

\begin{proof}
It is straightforward to check that 
\begin{align}
\begin{split}
-(d-1) g_+= \Ric_{g_+} = &\ \Ric_{\bar{g}} + (d - 2) \vr^{-1}\nabla_{\bar{g}}^2\vr + (\vr^{-1}\Delta_{\bar{g}}\vr - (d - 1) \vr^{-2}|\nabla_{\bar{g}}\vr|^2) \bar{g},
\\
-d(d-1)= R_{g_+} = &\ \vr^2 \left(R_{\bar{g}} + (2d-2)\vr^{-1}\Delta_{\bar{g}}\vr - d(d-1)\vr^{-2}|\nabla_{\bar{g}}\vr|^2\right),. 	
\end{split}
\end{align}
where $d = n + 1$.
Applying $-\Delta_{g_+} w = n$, we have that 
\begin{align*}
& R_{\bar{g}} =  (d-1)(d-2)\vr^{-2}(1 - |\nabla_{\bar{g}}\vr|_{\bar{\rho}}^2),
\\
& \Ric_{\bar{g}} =  -(d-2)\vr^{-1}\nabla_{\bar{g}}^2\vr.	
\end{align*}
It turns out that 
$R_{\bar{g}} = -(d - 2) \vr^{-1}\Delta_{\bar{g}}\vr$. 	
 
Now 
\begin{align}
\Delta_{\bar{g}}(\vr^{\frac{\fa}{2}})\vr^{\frac{\fa}{2}} = \frac{\fa}{2}\vr^{\fa - 1}\Delta_{\bar{g}} \vr + \frac{\fa}{2}\left(\frac{\fa}{2} -1 \right)\vr^{\fa - 2}|\nabla_{\bar{g}}\vr|^2,
 \end{align}
and hence 
\begin{align*}
E(\vr) = &\ - \frac{\fa}{2}\vr^{\fa - 1}\Delta_{\bar{g}}\vr + \left(\frac{\gamma^2}{4} - 1\right)\cdot\vr^{\fa - 2}(1 - |\nabla_{\bar{g}}\vr|^2) + \frac{n-1}{4n} R_{\bar{g}}\vr^{\fa}
\nonumber\\
= &\ \vr^{\fa}\left(\frac{\fa}{2}\cdot\frac{1}{d-2} + \left(\frac{\gamma^2}{4} - 1\right)\cdot\frac{1}{(d-1)(d-2)} + \frac{n - 1}{4n}\right) R_{\bar{g}}
\nonumber\\
= & \ C_{n,\gamma} \cdot \vr^{\fa} \cdot R_{\bar{g}},
\end{align*}
where
\begin{align}
C_{n,\gamma} = \frac{n^2 -3 - 4n\gamma + \gamma^2}{4n(n-1)}.	
\end{align}
The proof of the lemma is complete.
\end{proof}

We end this subsection by mentioning a lemma regarding the simplified expression of $P_{2\gamma}$ when $\bg = \vr^2 g_+$ is Fefferman-Graham's compactified metric.
\begin{lemma}\label{l:GJMS-operator-regular-metric}
Let $(X^{n + 1}, g_+)$ be a Poincar\'e-Einstein manifold and let $(M^n, h)$ be its conformal infinity. If $\bg = \vr^2 g_+$ is Fefferman-Graham's compactified metric, then   
\begin{align} 
    P_{2\gamma} f = \frac{2^{2\gamma}\cdot\Gamma(\gamma)}{2\gamma\cdot\Gamma(-\gamma)} \lim_{\vr\rightarrow 0} \vr^\fa \frac{\partial U}{\partial \vr}, \quad f\in C^{\infty}(M^n),
\end{align}
where $U$ is the unique solution in \eqref{e:U-extension-problem-PE}.
\end{lemma}

\begin{proof}
By \eqref{e:vr-expansion}, 
$\vr = y + O(y^3)$ which implies the term $\Phi$ in \eqref{e:assumption-on-vr} is vanishing. Using \eqref{e:expression-of-GJMS-lower-regularity}, 
we have that
\begin{align}\label{eqe-Pf}
P_{2\gamma} f =  \frac{2^{2\gamma}\cdot\Gamma(\gamma)}{2\gamma\cdot\Gamma(-\gamma)} \lim_{\vr\rightarrow 0} \vr^\fa \frac{\partial U}{\partial \vr},\end{align}
 which completes the proof.
\end{proof}

Summarizing the above discussion, we have the following proposition which is a combination of Lemma \ref{l:E} and Lemma \ref{l:GJMS-operator-regular-metric}. 
\begin{proposition}\label{p:extension-problem-in-FG-metric}
Let $(X^{n + 1}, g_+)$ be a Poincar\'e-Einstein manifold and let $(M^n, h)$ be its conformal infinity. Let $\bg = \vr^2 g_+$ be Fefferman-Graham's compactified metric associated to $h$. Given a smooth function $f \in C^{\infty}(M^n)$, let $U\in H^{1,\fa}(X^{n + 1})$ be the unique solution of 
\begin{align}
\begin{cases}
		-\Div_{\bg}(\vr^{\fa} \nabla_{\bg} U) + \vr^{\fa} \cJ_{\bg} U = 0,  & \text{in}\ X^{n+1}
		\\
	 U = f, & \text{on}\ M^n,
\end{cases}
\end{align}  
where $\cJ_{\bg} \equiv C_{n,\gamma} R_{\bg}$, $C_{n,\gamma}$ is the constant in Lemma \ref{l:E}, and
\begin{align*}
H^{1,\fa}(X^{n + 1}) \equiv \overline{C^{\infty}(X^{n + 1})}^{\|\cdot\|_{H^{1,\fa}}},\quad \|U\|_{H^{1,\fa}(X^{n+1})}\equiv \left(\int_{X^{n+1}}\vr^{\fa} U^2 \dvol_{\bg}+\int_{X^{n + 1}}\vr^{\fa} |\nabla_{\bg} U|^2 \dvol_{\bg} \right)^{\frac{1}{2}}.\end{align*}
Then 
\begin{align}
P_{2\gamma} f =  \frac{2^{2\gamma}\cdot\Gamma(\gamma)}{2\gamma\cdot\Gamma(-\gamma)} \lim_{\vr\rightarrow 0} \vr^\fa \frac{\partial U}{\partial \vr}.\end{align}
\end{proposition}

\subsection{Degenerate/singular equations and regularity}
\label{ss:regularity-of-solutions}

In this section, we clarify the regularity properties for the (weak) solutions $f\in C^{\infty}(M^n)$ of $P_{2\gamma} f = 0$ and the solutions  $U\in H^{1,\fa}(X^{n + 1})$ of corresponding 
extension problem \eqref{e:U-extension-problem-PE} in $(X^{n + 1}, \bg)$. Since we will always specify Fefferman-Graham's metric as our conformal compactification throughout this paper, the simplified equation in the extension problem and the expression of the fractional GJMS operator are given by Proposition \ref{p:extension-problem-in-FG-metric}.

 Notice first that we consider the range of parameters $\gamma \in (0,1)$, which amounts $\fa=1-2\gamma \in (-1,1)$. The aim of this section is to recall the regularity results needed for our purposes. In the model case when $(X^{n + 1}, g_+)$ is the hyperbolic space,  a general regularity theory on the partial compactification $(\dR_+^{n + 1}, g_0)$ of $(X^{n + 1}, g_+)$ has been developed in \cite{STV1, STV2} for a larger class of parameters $\fa \in (-1, \infty)$. In our particular case, we are only considering the range $\fa \in (-1,1)$. Several of the arguments can then be simplified because of the  special properties of the weight described below. 
 
 By the very construction of the defining function $\vr$ in Lemma \ref{l:Fefferman-Graham}, the following property holds: there is a sufficiently small constant $r_0 = r_0(n, \bg) >0 $ such that
\begin{equation} \label{mu.def}
\vr (x)\sim
\left\{
\begin{array}{cl}
 \text{dist}_{\bar g}(x, M^n) & \quad \text{when} \quad \text{dist}_{\bar g}(x, M^n) < r_0,\\
 1 & \quad \text{when} \quad \text{dist}_{\bar g}(x, M^n) > 2r_0.
 \end{array}
\right.
\end{equation}
 It is then easy to check that the function $\vr^\fa$ is an $A_2$-Muckenhoupt weight (see e.g. \cite{dyda}). We recall that a function $w \in L^1_{loc} (\mathbb R^{n+1})$ is called an $A_2$ weight  if the following holds: \begin{align}
\sup_{B \subset \mathbb R^{n+1}}  \left( \fint_B w \dvol_{g_0} \right) \left( \fint_B w^{-1} \dvol_{g_0}\right) <\infty,
\end{align}
whenever $B$ is a ball. In a series of important papers \cite{fks,fkj,fjk2}, Fabes, Kenig, Jerison and Serapioni developed the basic theory of equations in divergence form whose main coefficients are $A_2$ functions. Let $\fm \equiv n + 1$ For our purpose, we need to double the compactified manifold $(\overline{X^{\fm}}, \bg)$ along the totally geodesic boundary $(M^n, h)$, which gives a closed manifold $\fX^{\fm} \equiv \overline{X^{\fm}}\bigcup\limits_{M^n}\overline{X^{\fm}}$ equipped with a $C^{n-1, \alpha}$-Riemannian metric (still denoted as $\bg$). Since we are interested in local results, we first choose a domain $\Omega$ in $\fX^{\fm}$ and denote now $\vr$ the previous special defining function where $ \text{dist}_{\bar g}$ denotes the {\it  signed} distance to $M^n$. For any $\fa\in (-1,1)$, we define the weighted Sobolev space $H^{1,\fa}(\Omega)$ as the closure of $C^\infty(\overline\Omega)$ with respect to the norm
\begin{align}\|u\|_{H^{1,\fa}(\Omega)} \equiv \left(\int_{\Omega}\vr^\fa u^2\, \dvol_{\bar g}+\int_{\Omega}\vr^\fa |\nabla_{\bg} u|^2\, \dvol_{\bar g}\right)^{\frac{1}{2}}.\end{align}

By the results in \cite{fks}, it is known the space $\|u\|_{H^{1,\fa}(\Omega)}$ is a Hilbert space (for its usual norm).  We record here the following Sobolev inequality (which follows from the proof of a more general one in \cite{fks}, see also \cite{STV1}). 

\begin{lemma}[Sobolev inequality]
    Let $\fa \in (-1,1)$ and $u\in C^1_c(\Omega)$.  Then there exists a constant $c(\fm, \Omega)$ such that
\begin{equation}\label{sobo>-1}
\left(\int_{\Omega}\vr^\fa |u|^{2^*(\fa)} \dvol_{\bar g}\right)^{2/2^*(a)}\leq c(\fm, \Omega)\int_{\Omega}\vr^\fa |\nabla_{\bg} u|^2\dvol_{\bar g},
\end{equation}
where the optimal embedding exponent is
\begin{equation}\label{2*a}
2^*(\fa)\equiv \frac{2(n+1+\fa^+)}{n+\fa^+-1}.  
\end{equation}
When $n=1$ and $\fa^+>0$ the same inequality holds. When $n=1$ and $\fa^+=0$ then the embedding holds in any weighted $L^p(\Omega,\vr^\fa \mathrm{d}z)$ for $p>1$.
\end{lemma}

We now describe the regularity properties of the solutions under consideration. The next lemmas are straightforward adaptations of the results in \cite{fks,STV1}. We will always consider an energy solution of 
   \begin{align}\label{eq-U-e}
        -\Div_{\bg}(\vr^{\fa} \nabla_{\bg} U) + \vr^{\fa} \cJ_{\bg} U = 0 \quad \text{ in some geodesic ball}\ \cB_r(x_+)  \subset \fX^{\fm},
    \end{align}
    where $\cJ_{\bg} \equiv C_{n,\gamma} R_{\bg}$.
The following lemma follows from classical integration by parts. 

\begin{lemma}[Caccioppoli type inequality]
     Assume that $U$ solves \eqref{eq-U-e}. 
         Let $r \in (0, 1]$, $\beta>1$ and $\eta \in C_0^\infty(\cB_r(x_+))$. Then there exists a constant $C = C(r, \fa, \bg) > 0$  such that
    \begin{align}
\int_{\cB_r(x_+)}\vr^{\fa} |\nabla_{\bg} (\eta U^\frac{\beta}{2})|^2 \dvol_{\bg} \leq C\int_{\cB_r(x_+)} \vr^{\fa}  (|\nabla_{\bg} \eta|^2 +|\cJ_{\bg}| \eta^2)U^2 \dvol_{\bg}.
    \end{align}
\end{lemma}
Iterating the previous lemma, one gets classically the following upper bound by Moser's argument using the previous Sobolev inequality .  
\begin{lemma}[$L^\infty$ bounds]
 Assume that $U$ solves \eqref{eq-U-e}. Then for $r\in (0,1)$, there exists a constant $C = C(r, \fa, n, \bg) > 0$ such that
    \begin{align*}
\|U\|_{L^\infty(\cB_r(x_+))} \leq C\|U\|_{L^{2,\fa}(\cB_1(x_+))}.
    \end{align*}
\end{lemma}

The previous lemma allows to consider the term $ \vr^{\fa} \cJ_{\bg} U$ as the r.h.s. of the equation \eqref{eq-U-e} and apply directly the arguments in \cite{STV1} to deduce that 

\begin{lemma}\label{lem-eps-C1a} Let $x_+\in M^n$.
 Assume that $U$ is an even solution  to \eqref{eq-U-e} under the doubling $\fX^{\fm}$ such that $P_{2\gamma} f =0$ where $f$ is the trace of $U$ on $M^n$. Then for $r\in (0,1)$, there exists a constant $C = C(r, \fa, n, \bg) > 0$ such that
    \begin{align*}
\|U\|_{C^{1,\alpha}(\cB_r(x_+))} \leq C\|U\|_{L^{2,\fa}(\cB_1(x_+))}.
    \end{align*}
\end{lemma}
\begin{proof}
We follow the argument in [\cite{STV1}, Theorem 1.2] for instance. By the previous lemma, we know that $U$ solves locally 
\begin{align}\label{eq-U-temp}
        -\Div_{\bar g} (\vr^\fa  \nabla_{\bar g} U) = \vr^\fa  \zeta \text{ in } \cB_1(x_+),
    \end{align}
    where $\zeta =- \cJ_{\bg} U \in L^p(\cB_1(x_+))$ for any $p>1$.  Now since $U$ is assumed to be even w.r.t. the doubling $\fX^{\fm}$, we are in the {\it  local } situation of \cite{STV1} and the result follows from applying the techniques to prove Theorem 1.2 there. Notice that our metric enjoys some generic $C^{\fm-2}$ regularity, which is of course enough to derive the desired $C^{1,\alpha}$ bound.  
\end{proof}

\begin{remark}
In the previous lemma, the symmetry assumption is necessary as pointed out in \cite{STV1}. In our special case since the right hand side involves $U$ and the {\it  doubled metric} $\bar g$, which is also symmetric.   
\end{remark}

The following lemma is a summary of the previous discussion and provides additional regularity depending on the smoothness of the compactification. 

\begin{lemma}\label{eq-C1a-U}
    Let $U$ be an even solution to \eqref{eq-U-e} under the doubling $\fX^{\fm}$ such that $P_{2\gamma} f =0$ where $f$ is the trace of $U$ on $M^n$. Then,
\begin{align}\label{reg1}
    \|U\|_{C^{n - 1,\alpha}(\cB_{2/3}(x_+))}\leq C(n, \fa, \bg)\|U\|_{L^{2,\fa}(\cB_1(x_+))}.
\end{align}
Furthermore, 
\begin{enumerate}
\item 	
if $n = 2$, or $n \geq 4$ is even and the conformal infinity $(M^n,[h])$ is obstruction flat, then for any $k\in \dN_0$ and $\alpha \in (0,1)$,  
\begin{align}\label{reg2}
    \|U\|_{C^{k,\alpha}(\cB_{1/2}(x_+))}\leq C(k, \alpha, n, \fa, \bg) \|U\|_{L^{2,\fa}(\cB_1(x_+))};
\end{align}
\item 
if $n$ is odd, then for any $k\in\dN_0$ and $\alpha \in (0,1)$,
\begin{align}\label{reg3}
    \|U\|_{C^{k,\alpha}(\cB_{1/2}(x_+))}\leq C(k, \alpha, n, \fa, \bg)\|U\|_{L^{2,\fa}(\cB_1(x_+))}.
\end{align}
\end{enumerate}
 \end{lemma}
 \begin{proof}
Notice first that since the equation under consideration is posed in a ball of the doubled manifold $\fX^{\fm}$, all the estimates are {\sl interior } estimates and in particular hold across the characteristic manifold. By Lemma \ref{lem-eps-C1a}, we already know that $U$ is $C^{1,\alpha}$ for some $\alpha \in (0,1)$. More concretely, invoking \eqref{eq-gy1} and Proposition \ref{p:general-regularity-FG}, it can be observed that the even extension of the metric $\bar g$ onto $\cB_1(x_+)$ is locally $C^{n - 1,\beta}$ but is not $C^n$ in general and the scalar curvature term $\cJ_{\bg}$ is then locally $C^{n-3,\beta}$. To get the desired higher estimates under our standing geometric assumptions, we use the bootstrap argument in [\cite{STV1}, Section 6] noticing that in our case of $\fa \in (-1,1)$, the "dual" operator, which is associated to the parameter  $-\fa$ is also of the form $L_{-\fa}$ with $-\fa \in (-1,1)$. Therefore since $U$ satisfies  locally
\begin{align}-\Div_{\bg}(\vr^{\fa} \nabla_{\bg} U) =- \vr^{\fa} \cJ_{\bg} U, 
\end{align}
taking derivatives in the tangential and normal variables respectively, and using the Schauder estimates in Theorem 7.9 in \cite{STV1} (the lemmas before this theorem allow to deal with the normal derivatives solely) leads to the desired estimate \eqref{reg1}. The other two estimates \eqref{reg2}-\eqref{reg3} are obtained in exactly the same way taking higher derivatives.  
 \end{proof}
 
 \begin{remark}
 It is important to notice for later purposes that Lemma \ref{eq-C1a-U} provides generically {\sl finite} smoothness of the solution $U$. In view of the $\epsilon-$regularity we will need later, this technical aspect requires an adaptation of stability results. 
 \end{remark}

\section{Singular set in the model case}\label{model_section}

In this section, we focus on the model case. Let $(X^{n+1}, g_+)\equiv (\dH^{n+1}, g_{\dH^{n+1}})$ be the hyperbolic space, where 
\begin{align}g_+ \equiv \frac{dx^2 + dy^2}{y^2},\quad y>0, \quad x\equiv (x_1,\ldots, x_n)\in \dR^n\end{align}  is the hyperbolic metric on $\dH^{n+1}$ with $\sec_{g_+} \equiv -1$. Now $(X^{n+1},g_+)$ can be {\it partially compactified} to the upper half space $\dR_+^{n+1}\equiv \{(x,y)| y \geq 0,\ x\in\dR^n \}$  
equipped with the Euclidean metric $g_{\dR_+^{n+1}} \equiv dx^2 + dy^2$.

Throughout this section, we will consider the fractional Laplacian $(-\Delta)^{\gamma}$ defined on the conformal infinity $\dR^n$, where
 $\gamma\in(0,1)$. 
 Now let $f:\dR^n \to \dR$ be a smooth function in the Schwarz space $\mathcal{S}(\dR^n)$. Then there exists a unique Caffarelli-Silvestre type extension $U\in H^{1,\fa}(\dR_+^{n+1})$ of $f$ that satisfies
 \begin{align}\label{e:U-extension}
 \begin{cases}
 	\Div(y^{\fa}\nabla U) = 0 & \text{in}\ \dR_+^{n+1},
 	\\
 	U(x,0) = f(x) & \text{on} \ \dR^n,
 \end{cases}	
 \end{align}
where $\fa\equiv 1 - 2\gamma\in(-1,1)$. Then 
\begin{align}\label{eq-Delta-f-Uy}
(-\Delta)^{\gamma} f \equiv \frac{2^{2\gamma}}{2\gamma}\cdot \frac{\Gamma(\gamma)}{\Gamma(-\gamma)}\cdot \lim\limits_{y\to 0}y^{\fa}\frac{\p U}{\p y}.	
\end{align}
We are interested in a $\gamma$-harmonic function $f\in \mathcal{S}(\dR^n)$ that satisfies $(-\Delta)^{\gamma} f = 0$ on $B_1(\bo)\subset \dR^n$. 
\subsection{The quantitative symmetry and quantitative stratification}

This subsection summarizes some definitions and examples. To begin with, let us recall some basic definitions given in \cite{Cheeger-Naber-Valtorta}.

\begin{definition}Let $B_1(\bo)\subset \dR$. 
A smooth function $u: B_1(\bo)\to \dR$ is called {\it nondegenerate} if for every $x$ some derivative of some order is nonzero.
\end{definition}

\begin{example}
	It is immediate to see that any non-constant analytic function is nondegenerate. 	
\end{example}
 
For nondegenerate functions, it is essential to define the tangent map to study delicate properties on small scales.  

\begin{definition}
[Tangent map on $\dR^n$] Let $B_1(\bo)\subset\dR^n$. Let $u:B_1(\bo)\to \dR$ be a smooth nondegenerate function and $r>0$. Then for every $x\in B_{1-r}(\bo)$, we define 
\begin{align}
T_{x,r} u(\xi) \equiv \frac{u(x + r \xi) - u(x)}{ {\displaystyle \left(\fint_{\p B_1(\bo)}(u(x + r \zeta) - u(x))^2d\sigma(\zeta)\right)^{\frac{1}{2}} } }, \quad \xi \in B_1(\bo)\subset \dR^n.	
\end{align}
	If the denominator vanishes, we set $T_{x,r}u=\infty$. We also define 
	\begin{align}
	T_{x,0} u(\xi) = T_x u(\xi) = \lim\limits_{r\to 0} T_{x,r} u(\xi).	
	\end{align}
 
\end{definition}

We remark that if $u$ is smooth and nondegenerate at $x$, then a unique 
tangent map $T_x u$ exists. Moreover, up to rescaling, $T_x u $ is the leading order polynomial of the Taylor expansion of $u-u(x)$ at $x$. To define the symmetry of a function, we need the following notion. 
\begin{definition}
A polynomial $P$ is said to be {\it  homogeneous at $x\in \dR^n$} if for some $d\in\dN_0$,
\begin{align}
P(y) = \sum\limits_{|\beta| = d}A_{\beta}(y - x)^{\beta},	
\end{align}
where $\beta$ is a multi-index and $|\beta|= \sum\limits_{i}\beta_i=d$.	
\end{definition}

\begin{definition}
[$k$-symmetry] Let $u:\dR^n \to \dR$ be a continuous function. 
\begin{itemize}
\item $u$ is called $0$-symmetric at the origin $\bo\in\dR^n$ if $u$ is a homogeneous polynomial at the origin $\bo \in \dR^n$. 
\item $u$ is called $k$-symmetric at the origin $\bo\in\dR^n$  if  $u$ is called $0$-symmetric at the origin and there exists a $k$-dimensional subspace $V \subset \dR^n$ such that 
\begin{align}
u(x + y) = u(x), \quad \forall x\in \dR^n,\  y\in V.	
\end{align}
 	
\end{itemize}

\end{definition}

\begin{definition}[First-order stratification]
Given a smooth nondegenerate function $u:B_1(\bo)\to \dR$ we define the $k^{th}$-singular stratum of $u$ by 
\begin{align}
\mathcal{S}^k(u) \equiv \{x\in B_1(\bo): T_xu\ \text{is not}\ (k+1)\text{-symmetric}\}.	
\end{align}

\end{definition}

\begin{example}Here are some basic examples. 
\begin{itemize}
	\item 	Any homogeneous polynomial of degree $1$, i.e., a linear function, on $\dR^n$ is $(n-1)$-symmetric.
\item Let $u$ be a nondegenerate function on $\dR^n$. Then for almost every point $x$, the tangent map $T_x u$ is linear, and hence $\dim_{\mathcal{H}}(\mathcal{S}^{n-1}) = n$ and $\mathcal{S} = \mathcal{S}^{n-2}$.  
\end{itemize}
\end{example}

The {\it quantitative differentiation theory} in our context 
relies on the quantitative stratification of the singular set which provides a precise characterization of the function on definite scales. 

\begin{definition}
[Quantitative symmetry on $\dR^n$] Let $B_1(\bo)$ be the unit ball in $\dR^n$, and let $u: B_1(\bo)\to\dR$ be a continuous function. We say $u$ is $(k,\eta,s)$-symmetric at $x$ if there exists a $k$-symmetric polynomial $P$ such that $\fint_{\p B_1(\bo)}|P|^2 d\sigma_{\p B_1(\bo)}= 1$ and 
\begin{align}
\fint_{B_1(\bo)}|T_{x,s} u - P|^2 \dvol_{\dR^n} < \eta.	
\end{align}
\end{definition}

In our interested context of $\dR^{n + 1}\equiv\{(x,y): x\in\dR^n, \ y\in \dR\}$, we will also consider how a function $U: \cB_1(\bo_+)\to \dR$ behaves near the hypersurface $\dR^n \equiv\{y = 0\} \subset \dR^{n+1}$, where $\cB_1(\bo_+)\subset \dR^{n+1}$ is the unit ball centered at the origin $\bo_+ \in\dR^{n+1}$. For our purpose,  the weighted measure $|y|^{\fa}\dvol_{\dR^{n + 1}}$ on $\dR^{n+1}$ will be frequently used. Now we can define the tangent map and the quantitative symmetry similarly.

\begin{definition}
[Tangent map on $\dR^{n+1}$]  Let $\cB_1(\bo_+)$ be the unit ball in $\dR^{n+1}$. Let $U:\cB_1(\bo_+)\to \dR$ be a smooth nondegenerate function and $r>0$. Then for every $x_+ \equiv (x, 0) \in \cB_{1-r}(\bo_+)$, we define 
\begin{align*}
\cT_{x_+,r} U(\xi) \equiv \frac{U(x_+ + r \xi) - U(x_+)}{ {\displaystyle \left(\frac{1}{r^{\fa}}\fint_{\p \cB_1(\bo_+)}   |y(x_+ + r\zeta)|^{\fa} \cdot \left(U(x_+ + r \zeta) - U(x_+)\right)^2 d\sigma_{\p \cB_1(\bo_+)}(\zeta)\right)^{\frac{1}{2}} } }, \quad \xi \in \cB_1(\bo_+).	
\end{align*}
	If the denominator vanishes, we set $\cT_{x_+,r}U =\infty$. We also define 
	\begin{align}
	\cT_{x_+,0} U(\xi) = \cT_{x_+} U(\xi) = \lim\limits_{r\to 0} \cT_{x_+,r} U(\xi).
	\end{align}
 
\end{definition}

\begin{definition}
[Quantitative symmetry on $\dR^{n+1}$] A continuous function $U: \cB_1(\bo_+)\to\dR$  is said to be $(k,\eta,s ,\fa)$-symmetric at $x_+\equiv(x,0)\in\dR^{n+1}$ if there exists a $k$-symmetric polynomial $P$ such that $\frac{1}{s^{\fa}}\fint_{\p \cB_1(\bo_+)}|y|^{\fa} \cdot |P|^2 d\sigma_{\p \cB_1(\bo_+)} = 1$ and 
\begin{align}
\frac{1}{s^{\fa}}\fint_{\cB_1(\bo_+)}|y|^{\fa} \cdot |\cT_{x_+,s} U - P|^2 \dvol_{\dR^n} < \eta.	
\end{align}

\end{definition}

Now let us formulate the quantitative stratification of a continuous function $U:\cB_1(\bo_+)\to \dR$. 
\begin{definition}
[Quantitative singular strata]	Let us fix $\fa \in (-1,1)$, and let $U:\cB_1(\bo_+)\to \dR$ and $f: B_1(\bo)\to \dR$ be continuous functions on $\dR^{n + 1}$ and $\dR^n$,  respectively. Then we define 
the $(k,\eta,r)$-quantitative singular stratum of $U$ by 
\begin{align}
\mathcal{S}_{\eta,r}^k(U) \equiv &\ \left\{x\in  B_1(\bo) \subset \dR^n: U\ \text{is not} \ (k+1,\eta,s, \fa)\text{-symmetric at}\ x\ \text{for all}\ s \geq r \right\},
\\
\mathcal{S}_{\eta,r}^k(f) \equiv &\  \left\{x\in B_1(\bo) \subset \dR^n: f\ \text{is not} \ (k+1,\eta,s)\text{-symmetric at}\ x\ \text{for all}\ s \geq r \right\}.
\end{align}

\end{definition}

\begin{remark}\label{r:inclusion}
	Immediately, we have the following monotonicity,
	\begin{align}
	\mathcal{S}_{\eta, r}^k(U)	\subset 	\mathcal{S}_{\eta', r'}^{k'}(U) \quad \text{if}\ k\leq k',\ \eta'\leq \eta, \ r\leq r'. 
	\end{align}
Moreover, we can recover the standard stratification by 
\begin{align}
\mathcal{S}^k(U) = \bigcup\limits_{\eta > 0}\bigcap\limits_{r > 0} 	\mathcal{S}_{\eta, r}^k(U).	
\end{align}
Here, similar to \eqref{e:U-zero-critical-singular}, each classical stratum $\cS^k(U)$ is defined as the restriction 
\begin{align}
	\mathcal{S}^k(U) \equiv \{x\in B_1(\bo) \cap \dR^n : T_x U\ \text{is not}\ (k+1)\text{-symmetric}\}.	
\end{align}
Similar relations hold for $\cS_{\eta, r}^k(f)$ as well. 
\end{remark}

\subsection{Monotonicity and quantitative rigidity}
\label{ss:quantitative-rigidity-Euclidean}

For our purpose, we will consider the extension of a $\gamma$-harmonic function in the {\it entire} Euclidean space $\dR^{n+1}$.
For any smooth function $f\in \mathcal{S}(\dR^n)$ that is harmonic on $B_1(\bo)$, there exists a unique symmetric extension of $f$ (denoted as $U$) in $\dR^{n+1}$ satisfying
\begin{align}\label{e:symmetric-extension}
\begin{cases}
 	\Div(y^{\fa}\nabla U) = 0 & \text{in}\ \dR^{n+1},
  	\\ U(x, -y) = U(x,y) & \text{in}  \ \dR^{n+1},
\\
 	U(x,0) = f(x) & \text{on} \ \dR^n,
 	\\
 	(-\Delta)^{\gamma} f = 0 & \text{on} \ B_1(\bo).
 \end{cases}	
	 \end{align}

\begin{example}\label{ex-special-soln-U}
Let $x_1$ be the first component of $x$. A special solution to \eqref{e:symmetric-extension} when $f(x)=x_1^2$ is
\begin{align}\label{eq-U-special}
U(x, y) =  x_1^2 -\frac{1}{1+\fa}y^2.
\end{align}
While $y^{1-\fa}$	is a homogeneous solution to $\Div(y^{\fa}\nabla U) = 0$, by \eqref{eq-Delta-f-Uy}, $(-\Delta)^{\gamma} f =0$ implies that the term $y^{1-\fa}$ does not appear in \eqref{eq-U-special}. Notice that $U$ and $f$ here do not satisfy the decaying conditions: $f\in \cS(\dR^n)$ and $U\in H^{1,\fa}(\dR^{n + 1})$. But they can be served as local blow-up models.
\end{example}

One can define Almgren's frequency function for the symmetric extension $U\in H^{1,\fa}(\dR^{n+1})$. 
 Let us fix some notations.
 For any $x\in\dR^n$, let us denote $x_+\equiv (x, 0)\in\dR^{n+1}$ and 
\begin{align}\cB_r(x_+)\equiv \{w_+\in \dR^{n+1}: |w_+ - x_+|\leq r\},\quad \p\cB_r(x_+)\equiv \{w_+\in \dR^{n+1}: |w_+ - x_+| = r\}. \end{align}
Immediately, $\cB_r(x_+)\cap \dR^n = B_r(x)$ and $\p\cB_r(x_+)\cap \dR^n = \p B_r(x)$.
Then we define 
\begin{align}
\begin{split}
	E_U(x_+, r) 
	& \equiv \frac{1}{r^{n+\fa - 1}}\int_{\cB_r(x_+)}|y|^{\fa}\cdot |\nabla U|^2 \dvol_{\dR^{n+1}},
	\\
		H_U(x_+, r) 
	& \equiv \frac{1}{r^{n+\fa}}\int_{\p\cB_r(x_+)}|y|^{\fa}\cdot U^2 d\sigma.
\end{split}	
\end{align}
Then the Almgren's frequency function $N_U(x_+, r)$ is defined to be 
\begin{align}
N_U(x_+, r) \equiv \frac{E_U(x_+, r)}{H_U(x_+, r)} = \frac{\displaystyle {r \int_{\cB_r(x_+)}|y|^{\fa}\cdot |\nabla U|^2 \dvol_{\dR^{n+1}}}}{{\displaystyle \int_{\p\cB_r(x_+)}|y|^{\fa}\cdot U^2 d\sigma}}.
\end{align}
The following monotonicity result is due to Caffarelli-Silvestre \cite[theorem 6.1]{CS}.	

\begin{lemma}\label{l:Euclidean-monotonicity}
Let $U$ be the symmetric extension of $f$ that solves \eqref{e:symmetric-extension} in $\cB_1(\bo_+)\subset \dR^{n+1}$ such that $(-\Delta)^{\gamma} f = 0$ on $B_1(\bo)\subset \dR^n$. Then for any $x_+\in B_1(\bo)$, the frequency function $N_U(x_+, r)$ is monotone nondecreasing in $r\in (0,1 - |x_+|)$. Moreover, if 
\begin{align}N_U(x_+,r_1) = N_U(x_+, r_2)\quad \text{for some}\ 0\leq r_1 < r_2 < 1-|x_+|, \end{align} then $U$ is a homogeneous polynomial of degree $d = N_U(x,r)$.  
\end{lemma}

Since we are interested in the structure of the critical set of a function (not only the singular set), we will normalize Almgren's frequency as follows. 
If $U$ is not constant, let us define 
\begin{align}
\cN_U(x_+,r) \equiv   \frac{\displaystyle {r\cdot\int_{\cB_r(x_+)}|y|^{\fa}\cdot |\nabla U|^2 \dvol_{\dR^{n+1}}}}{{\displaystyle \int_{\p\cB_r(x_+)}|y|^{\fa}\cdot (U-U(x_+))^2 d\sigma}}.
\end{align}
When there is no ambiguity, for a $\gamma$-harmonic function $f$, we will denote by 
\begin{align}\cN(x, r) = \cN_f(x, r)  = \cN_U(x_+,r)\end{align} the Almgren's frequency of $f$ at $x\in \dR^n$.

Similarly, we have the following monotonicity. 
\begin{lemma}[Monotonicity and rigidity] \label{l:monotonicity}
Let $U$ be unique symmetric extension of a non-constant $\gamma$-harmonic function $f$ that solves \eqref{e:symmetric-extension} in $\cB_1(\bo_+)\subset \dR^{n+1}$ such that $(-\Delta)^{\gamma} f = 0$ on $B_1(\bo)\subset \dR^n$. Then the following properties hold. 
\begin{enumerate}
\item For any $x\in B_1(\bo)$, the frequency function $\mathcal{N}_f(x, r)$ of $f$ at $x_+$ is monotone nondecreasing with respect to $r\in (0,1 - |x|)$.
\item If for some $0\leq r_1 < r_2$, $\cN_f (x, r_1) = \cN_f (x, r_2)$, then $U - U(x_+)$ is a homogeneous polynomial of degree $d = \cN_f (x,r)$ centered at $x_+ = (x, 0)\in \dR^{n + 1}$.  

\end{enumerate}
\end{lemma}

Next, we will prove a uniform boundedness result that controls the frequency from a large scale to any smaller scales. 

\begin{lemma}[Uniform bound on frequency] \label{l:frequency-control-large-to-small}
Let $f \in \cS(\dR^n)$ be a non-constant function that is $\gamma$-harmonic on $B_1(\bo)\subset \dR^n$. Let $U$ be its unique symmetric extension such that
 $\cN_f(\bo,1)\leq \Lambda$. Then for each $\tau\in(0,1)$, there exists $C=C(n,\Lambda, \tau, \gamma)>0$ such that for each $x\in B_{\tau}(\bo)$ and $r\leq \frac{2}{3}(1-\tau)$, we have   
\begin{align}
\cN_f(x,r) \leq C.	
\end{align}

\end{lemma}

\begin{proof}
There is no harm to only prove the lemma for $\tau = \frac{1}{4}$ and $s = \frac{1}{2}$.
We also assume that $u(\bo) = 0$ such that $N(0,r) = \cN(0, r)$ for all $r\leq 1$. Let $x_+ \equiv (x, 0)\in \dR^{n + 1}$. By the definition of the normalized frequency,
\begin{align}
\cN_f(x, 1/2) = \cN_U(x_+,1/2) \equiv   \frac{\displaystyle {\int_{\cB_{1/2}(x_+)}|y|^{\fa}\cdot |\nabla U|^2 \dvol_{\dR^{n+1}}}}{{\displaystyle 2\int_{\p \cB_{1/2}(x_+)}|y|^{\fa}\cdot (U-U(x_+))^2 d\sigma}} \label{e:1/2-frequency}.
\end{align}
It follows from standard computations that 
\begin{align}
U(x_+)\cdot \int_{\p \cB_{1/2}(x_+)}|y|^{\fa} d\sigma = 	\int_{\p \cB_{1/2}(x_+)}|y|^{\fa} \cdot U  d \sigma.
\end{align}
Then for any $0<r\leq \frac{1}{2}$,
\begin{align}
\begin{split}
	& \int_{\p \cB_r(x_+)}|y|^{\fa}\cdot (U-U(x_+))^2 d\sigma
	\nonumber\\
	 = \ & \int_{\p \cB_r(x_+)}|y|^{\fa} U^2 d\sigma
 - \left(\int_{\p \cB_r(x_+)}|y|^{\fa} d\sigma\right)|U(x_+)|^2 
 \nonumber\\
 = \ & r^{n+\fa}\cdot \left(H(x_+ , r) - \left(\int_{\p \cB_1(\bo)}|\Theta|^{\fa}d\sigma \right)	|U(x_+)|^2  \right)\geq 0,
	\end{split}
\end{align}
where  $y \equiv r \cdot \Theta$
 and $\Theta$ is a spherical harmonic on the unit sphere.
 Immediately we have that  
\begin{align}
\left(\int_{\p \cB_1(\bo)}|\Theta|^{\fa}d\sigma \right)	|U(x_+)|^2 \leq H(x_+, r),\label{e:H-bound-scale-r}
\end{align}
and the integral in the denominator of \eqref{e:1/2-frequency} becomes 
\begin{align}
\int_{\p \cB_{1/2}(x_+)}|y|^{\fa}\cdot (U-U(x_+))^2 d\sigma = &\ \left(\frac{1}{2}\right)^{n+\fa}\left(H(x_+, 1/2) - \left(\int_{\p \cB_1(\bo)}|\Theta|^{\fa}d\sigma \right)	|U(x_+)|^2  \right).
\end{align}
Now we choose $r=1/3$ in \eqref{e:H-bound-scale-r} so that 
\begin{align} 
\int_{\p \cB_{1/2}(x_+)}|y|^{\fa}\cdot (U-U(x_+))^2 d\sigma \geq 	
    \left(\frac{1}{2}\right)^{n+\fa}\Big((H(x_+, 1/2) - H(x_+, 1/3)\Big).\label{e:denominator}	
\end{align}
To estimate the right-hand side of \eqref{e:denominator}, let us apply the formula
\begin{align}
\frac{d}{dr}\log H(x_+, r) = \frac{2N(x_+ , r)}{r}, 
\end{align}
which implies that 
\begin{align*}
 \frac{H(x_+, 1/3)}{H(x_+, 1/2)}  =  \exp\left(-2\int_{1/3}^{1/2}\frac{N(x_+, r)}{r}dr\right)  \leq    \exp\left(2N(x_+,1/3)\log(2/3)\right) = \left(2/3\right)^{2N(x_+, 1/3)}. 
\end{align*}
Plugging the above into \eqref{e:denominator}, 
\begin{align}
\int_{\p \cB_{1/2}(x_+)}|y|^{\fa}\cdot (U-U(x_+))^2 d\sigma \geq 	
    \left(\frac{1}{2}\right)^{n+\fa}\left(1-\left(2/3\right)^{2N(x_+, 1/3)}\right) H(x_+, 1/2).	
\end{align}
Therefore,
\begin{align}
\cN_f (x, 1/2) = \cN_U (x_+, 1/2) 
\leq N(x_+, 1/2) \cdot \left(1-\left(2/3\right)^{2N(x_+, 1/3)}\right)^{-1}\leq C(n,\Lambda),
\end{align}
 where we used the fact $N(x_+, 1/3) \geq 1$. The proof of the lemma is done.
\end{proof}

The rigidity part of Lemma \ref{l:monotonicity} has a quantitative version.  
\begin{proposition}
[Quantitative symmetry] \label{p:quantitative-symmetry}
	For any $\epsilon>0$, $n\geq 2$, $\gamma\in(0,1)$, $\Lambda>0$, there exists a uniform constant $\delta=\delta(\epsilon, n, \gamma, \Lambda)>0$ such that the following holds. Let $f$ be $\gamma$-harmonic on $B_s(\bo)\subset \dR^n$ with $\cN_f(0,s) = \cN_U(\bo_+,s)\leq \Lambda$, where $U$ is the unique symmetric extension of $f$ that solves \eqref{e:symmetric-extension} in $\dR^{n+1}$.
	If 
	\begin{align}
		\cN_U(\bo_+, s) - \cN_U(\bo_+, s\gamma) < \delta, 
	\end{align}
 then $U$ is $(0,\epsilon, s, \fa)$-symmetric at the origin $\bo_+\subset \dR^{n+1}$.
\end{proposition}

\begin{proof}
We will prove it by contradiction. Suppose there exist positive constants $\epsilon_0>0$, $\gamma > 0$, and a sequence of $U_j$ such that 
$\cN_{U_j}(\bo_+,1) \leq \Lambda$,  $\cN_{U_j}(\bo_+, 1) - \cN_{U_j}(\bo_+, \gamma) < \delta_j \to 0$, but $U_j$ is not $(0,\epsilon_0,1, \fa)$-symmetric at $\bo_+\in\dR^{n+1}$.
Without loss of generality, one can assume \begin{align}U_j(\bo_+) = \bo_+\quad  \text{and} \quad \fint_{\p \cB_1(\bo_+)}|y|^{\fa} \cdot  U_j^2 d\sigma = 1. \end{align}
	That is, $\cT_{\bo_+,1} U_j = U_j$. 
Then $U_j$ converges in $H^{1,\fa}(\cB_1(\bo_+))\cap C_{\loc}^1(\cB_1(\bo_+))$ to some function $U_{\infty}$. Moreover, by 
the weighted Sobolev trace inequality, we have that 
\begin{align}
\fint_{\p\cB_1(\bo_+)} |y|^{\fa} \cdot U_{\infty}^2 d\sigma	= 1. 
\end{align}
 Since  $\cN_{U_{\infty}}(\bo_+, 1) - \cN_{U_{\infty}}(\bo_+, \gamma) = 0$, 
by Lemma \ref{l:monotonicity}, $U_{\infty}$ is a homogeneous polynomial. The convergence 
$\| U_j - U_{\infty}\|_{H^{1,\fa}(\cB_1(\bo_+))} \to 0$ particularly contradicts the assumption that $U_j$ is not $(0,\epsilon_0,1,\fa)$-symmetric at the origin $\bo_+ \in \dR^{n + 1}$. 
\end{proof}

\begin{definition}[Good and bad scales]
Let us fix $\epsilon>0$, $n\geq 2$, $\gamma\in(0,1)$ and let us take $r_j \equiv \gamma^j$ for any $j\in\dZ_+$. Given a smooth function $U$ defined on $\cB_1(\bo_+)$ with $x_+\equiv(x,0)\in \cB_1(\bo_+)$, a scale $r_j$ is called a {\it good scale} at $x$ if $U$ is $(0,\epsilon, r_j, \fa)$-symmetric at $x$, and a scale $r_j$ is said to be a {\it bad scale} at $x_+$ if $U$ is not $(0,\epsilon, r_j, \fa)$-symmetric at $x_+$,
\end{definition}

\begin{corollary}
[Uniform control on the number of bad scales] \label{c:controlling-bad-scales}
For any $\epsilon>0$, $n\geq 2$, $\gamma\in(0,1)$, $\Lambda>0$, there exists a uniform constant $Q_0=Q_0(\epsilon, n, \gamma, \Lambda, \fa)>0$ such that the following holds. Let $U$ be a solution of \eqref{e:symmetric-extension} with $\cN_U(\bo_+, 1)\leq \Lambda$. Then for any $\epsilon > 0$ and  $x_+\equiv(x,0)\in \cB_{1/2}(\bo_+)$, there are at most $Q_0$ bad scales at $x_+$.

\end{corollary}

\begin{proof}
By Lemma \ref{l:frequency-control-large-to-small}, for any $x_+ = (x,0)\in B_{1/2}(\bo_+)$, we have 
 $\cN_U (x_+,1/2) \leq C(n,\Lambda, \tau)$.	
 Then one can obtain  
\begin{align}
C(n,\Lambda, \tau) \geq 
\cN_U (x_+, 1/2) = \sum\limits_{j = 1}^{\infty} (\cN_U (x_+, r_j) - \cN_U (x_+, r_{j + 1})).	
\end{align}
For any $\epsilon > 0$, let $\delta = \delta (\epsilon, n , \gamma, \Lambda) > 0$ be the constant in Proposition \ref{p:quantitative-symmetry}. Applying Lemma \ref{l:monotonicity}, there are finitely many scales $r_j$ of number bounded by $Q_0\equiv C(n,\Lambda, \tau) / \delta$, for which 
\begin{align}
	\mathcal{N}_U (x_+, r_j) - \mathcal{N}_U (x_+, r_{j + 1}) > \delta. 
\end{align}
Applying Proposition \ref{p:quantitative-symmetry}, one can see that all other scales are good scales.
\end{proof}

\subsection{Quantitative cone-splitting}
\label{ss:quantitative-cone-splitting}

We start with a general {\it cone splitting principle}. The computations are standard now; see also \cite[proposition 2.11]{Cheeger-Naber-Valtorta}.
\begin{lemma}
[Cone-splitting principle] \label{l:cone-splitting-principle}
Let $P:\dR^n\to \dR$ be a harmonic polynomial 
of degree $d$, homogeneous with respect to the origin. Assume that $P$ is symmetric with respect to a $k$-dimensional vector space $V\subset \dR^n$. Then 
\begin{enumerate}
\item $P$ is a linear function if and only if $P$ is $(n-1)$-symmetric.
\item If $P$ is not $(n-1)$-symmetric, and $P$ is also $0$-symmetric with respect to $z\not\in V$, then $P$ is $(k+1)$-symmetric with respect to the vector space $\Span(V,z)$. 	
\end{enumerate}
 	
\end{lemma}

\begin{remark}
The main property of cone splitting given in item (2) holds for general homogeneous polynomials.	
\end{remark}

\begin{proof}
Item (1) is obvious, so we omit the proof.

Let us prove item (2). Since $P$ is homogeneous with respect to  the origin,  
\begin{align}
	\langle \nabla P (x), x \rangle  = P_r (x)\langle\nabla r, x\rangle = d\cdot P(x).   
\end{align}
Also, $P$ is symmetric with respect to some $z\not\in V$, which implies 
\begin{align}
	\langle \nabla P , x - z\rangle  = d \cdot P(x).
\end{align}
Therefore, for any $x \in \dR^n$, 
we have	$\langle \nabla P (x), z \rangle  = 0$,	which implies that $P$ is constant along the line connecting $\bo$ and $z$. 
The proof is done. 
 \end{proof}

The following theorem gives a quantitative version of the cone splitting. 

\begin{theorem}
[Quantitative cone splitting] \label{t:quantitative-cone-splitting}
For any fixed $n\geq 2$, $\epsilon>0$, $\rho>0$, $r\in (0,1)$, $k\in\{0,1,\ldots, n-2\}$, $\Lambda>0$,  there exists 
a positive constant $\ud =\ud (n,\epsilon, \rho, \Lambda) > 0$ such that the following holds. Let $U$ be the unique solution of \eqref{e:U-extension} with $\cN_U(\bo_+,1)\leq \Lambda$.
If 
\begin{enumerate}
\item $U$ is $(k, \ud , r , \fa)$-symmetric at $\bo_+$ with respect to a $k$-dimensional vector space $V$,
\item $U$ is $(0, \ud , r, \fa)$-symmetric for some $z\in \cB_r(\bo_+)\setminus \cB_{\rho}(V)$, 
\end{enumerate}
then $U$ is also $(k+1, \epsilon , 1 , \fa)$-symmetric at $\bo_+$.

\end{theorem}

\begin{remark}
Assuming a weaker symmetry on a smaller ball and a cone symmetry with respect to a distinct point of definite distance, this theorem gives a stronger symmetry on a larger ball.
\end{remark}

\begin{proof}
We will prove it by contradiction. 
Suppose there exist positive constants $\epsilon_0>0$, $\rho_0>0$, $r_0>0$, $\Lambda_0$, and a positive integer $0\leq k \leq n-2$ such that 
	for a sequence $\ud_j\to 0$ and a sequence of solution $U_j$, the following properties hold.	
	\begin{enumerate}
	\item[(1)'] $U_j$	is $(k_0, \ud_j, r_0, \fa)$-symmetric at $\bo_+$ with respect to a $k_0$-dimensional vector space $V_j$.
	\item[(2)'] $U_j$ is $(0, \ud_j , r_0, \fa)$-symmetric for some point $z_j\in B_{r_0}(\bo_+)\setminus B_{\rho_0}(V_j)$. 
	\item[(3)'] $U_j$ is not $(k_0+1, \epsilon_0, 1, \fa)$-symmetric at $\bo_+$.
	\end{enumerate}
Without loss of generality, we can assume that $U_j(\bo_+) = \bo_+$ and 
$\fint_{\p\cB_1(\bo_+)}|y|^{\fa}\cdot|U_j|^2d\sigma =1$. 
Since $\cN_{U_j}(\bo_+, 1) \leq \Lambda_0$, we have that \begin{align}
 \int_{\cB_1(\bo_+)}|y|^{\fa}\cdot|\nabla U_j|^2\dvol_{\dR^{n+1}} \leq \Lambda_0.	
 \end{align}
Therefore, $U_j$ converges in $H^{1,\fa}(\cB_1(\bo_+))\cap C^1(\cB_1(\bo_+))$ to $U_{\infty}$. We also have the convergence $V_j \to V_{\infty}$ and $x_j \to x_{\infty} \not\in V_{\infty}$ for some vector space $V_{\infty}$ with $\dim V_{\infty} = k$.

By item (1)', there exists a sequence of $k_0$-symmetric homogeneous polynomial $P_j$ such that  
\begin{align}
\fint_{\cB_1(\bo_+)}|y|^{\fa}\cdot |\cT_{\bo_+,r_0}U_j - P_j|^2 \dvol_{\dR^{n+1}} \to 0  	
\end{align}
Notice that by the uniform doubling property,
\begin{align}
r_0^{-\fa} \cdot \fint_{\p\cB_1(\bo_+)} |y(\zeta)|^{\fa} \cdot U_{\infty}(r_0\cdot \zeta)^2  d\sigma(\zeta) = \fint_{\p\cB_{r_0}(\bo_+)}|y|^{\fa} \cdot U_{\infty}^2 d\sigma \geq r_0^{2\Lambda} > 0.
\end{align}
Therefore, $\cT_{\bo_+,r_0}U_j$ converges to some $k_0$-symmetric homogeneous polynomial $P_{\infty}$ with $\deg(P_{\infty}) = d$. By item (2)', $P_{\infty}$ is symmetric for some point $z_{\infty}\in \cB_{r_0}(\bo_+)\setminus \cB_{\rho_0}(V_{\infty})$. Applying the cone-splitting principle in Lemma \ref{l:cone-splitting-principle}, $P_{\infty}$ is $(k_0+1)$-symmetric at $\bo_+$.

To finish the proof, we only need to prove the claim that $P_{\infty} = U_{\infty}$, which gives the desired contradiction. Indeed, if the claim is true, by the normalization condition, we have $U_j = \cT_{\bo_+, 1} U_j$. Then
\begin{align}
\fint_{\cB_1(\bo_+)} |y|^{\fa} \cdot  |\cT_{\bo_+, 1} U_j - P_{\infty}|^2 \dvol_{\dR^{n+1}} \to 0 	
\end{align}
so that the desired contradiction arises. 

Now let us prove the claim. It follows from the convergence that
\begin{align} P_{\infty}(\xi) = \cT_{\bo_+,r_0}U_{\infty}(\xi) = \frac{ \displaystyle{U_{\infty}(r_0\cdot \xi)} }{ \displaystyle{ r_0^{\fa} \cdot \fint_{\p\cB_1(\bo_+)} |y(\zeta)|^{\fa} \cdot U_{\infty}(r_0\cdot \zeta)^2  d\sigma(\zeta)} },\end{align}
which implies that $U_{\infty}$ is homogeneous. Then we have that 
\begin{align}
 r_0^{-\fa} \cdot \fint_{\p\cB_1(\bo_+)} |y(\zeta)|^{\fa} \cdot U_{\infty}(r_0\cdot \zeta)^2  d\sigma(\zeta) = \fint_{\p\cB_{r_0}(\bo_+)}|y|^{\fa} \cdot U_{\infty}^2 d\sigma = r_0^d,
\end{align}
where the last equality follows from the homogeneity of $U_{\infty}$ and the normalization \begin{align}\fint_{\p\cB_1(\bo_+)}|y|^{\fa}\cdot|U_{\infty}|^2d\sigma =1.\end{align} 
Finally, we obtain that 
\begin{align}
P_{\infty}(\xi) = \frac{U_{\infty}(r_0\cdot \xi)}{r_0^d}
= 	\frac{U_{\infty}(r_0\cdot \xi)}{r_0^d} = U_{\infty}(\xi),\quad \xi \in \cB_1(\bo_+),
\end{align}
which completes the proof.
\end{proof}

Using induction, one can obtain the following corollary, which plays
a crucial role in our covering arguments. 
\begin{corollary}\label{c:inductive-splitting}
	For any fixed $n\geq 2$, $\epsilon>0$, $\rho>0$, $r\in (0,1)$, $k\in\{0,1,\ldots, n-2\}$, $\Lambda>0$,  there exists 
a positive constant $\delta=\delta(n,\epsilon, \rho, \Lambda) > 0$ such that the following holds. Let $U$ be the unique solution of \eqref{e:symmetric-extension} with $\cN_U(\bo_+, 1)\leq \Lambda$.
If 
\begin{enumerate}
\item $U$ is $(0,\delta,r,\fa)$-symmetric at $\bo_+$;
\item for any vector space $V$ of dimension $\leq k$, there exists some point $z\in \cB_r(\bo_+)\setminus B_{\rho}(V)$ such that $U$ is  $(0,\delta,r,\fa)$-symmetric at $z$, 
\end{enumerate}
then $U$ is also $(k+1,\epsilon,1,\fa)$-symmetric at $\bo_+$.

\end{corollary}

\begin{proof}
Let $\ud\equiv\ud(n,\epsilon,\rho,\Lambda)>0$ be the constant in Theorem \ref{t:quantitative-cone-splitting}. For any $1 \leq i \leq n-1$, let us define the composite function 
\begin{align}
	\ud^{(n-1-i)} \equiv \underbrace{\ud \circ 
\ud \circ \ldots \circ \ud(n,\epsilon, \rho,\Lambda)}_{i\ \text{factors}}.
\end{align}
 For example, $\ud^{(n-2)}= \ud(n,\ud, \rho, \Lambda)$. It is obvious that the following monotonicity holds, \begin{align}\ud^{(0)} < \ud^{(1)} < \ldots < \ud^{(n-2)} <  \ud^{(n-1)} \equiv \epsilon.\end{align} 
 Let us choose $\delta\equiv \ud^{(0)}$ (with $n$ factors in the composition). 
 
 Let $U$ be a function that satisfies properties (1) and (2) for some fixed integer $0\leq k \leq  n-2$. Since $U$ is $(0,\delta,r,\fa)$-symmetric at $\bo_+$, by applying property (2), there exists a largest integer $1\leq k_0 \leq n-2$ such that $U$ is $(\ell, \ud^{(\ell)}, 1,\fa)$-symmetric for any $1\leq \ell \leq k_0$. In fact, the existence is given by applying Theorem \ref{t:quantitative-cone-splitting} to the case $k=0$.

We observe that $k_0 \geq k$. Indeed, suppose $k_0 < k$ so that $U$ is $(k_0, \ud^{(k_0)},1,\fa)$-symmetric at $\bo_+$. In other words,  $U$ is $(k_0, \ud^{(k_0)}, 1 , \fa)$-symmetric at $\bo_+$ with respect to a $k_0$-dimensional vector space $V$ with $k_0 < k$. Applying property (2) and Theorem \ref{t:quantitative-cone-splitting}, 
 $U$ is $(k_0 + 1, \ud^{(k_0+1)}, 1, \fa)$-symmetric at $\bo_+$, which contradicts the maximality of $k_0$.   
Since $k_0 \geq k$, we particularly have that $U$ is $(k, \ud^{(k)}, 1, \fa)$-symmetric at $\bo_+$. Applying property (2) and Theorem \ref{t:quantitative-cone-splitting} again, $U$ is in fact $(k+1, \ud^{(k+1)}, 1, \fa)$-symmetric at $\bo_+$, where $\ud^{(k+1)} < \ud^{(n-1)} = \epsilon$. 
The proof is complete.
\end{proof}

\subsection{Volume estimate and frequency decomposition}

\label{ss:volume-estimate}

Let $f$ be $\gamma$-harmonic and let $U$ be the unique symmetric extension of $f$ given by \eqref{e:symmetric-extension}. 
In this subsection, for both $U$ and $f$, we will prove some uniform volume estimates for small tubular neighborhoods of the quantitative strata. 
We adopt the general formulation of the estimate and strategy of proof from \cite{Cheeger-Naber-Valtorta}.

The main result of this section is as follows.
\begin{theorem}\label{t:volume-estimate-model}
For every $j\in\dN$, $\epsilon>0$, $k\leq \fm-2$, there exist some positive constants $0<\mu(\fm,\epsilon,\Lambda)<1$ and $C(\fm,\Lambda, \epsilon)>0$ such that the following property holds. Let $U: \cB_1(\bo_+)\to \dR$ be the unique symmetric extension with 
$\cN_U(\bo_+, 1)\leq \Lambda$. Then 
\begin{align}
\Vol(B_{\mu^j}(\cS_{\epsilon,\mu^j}^k(U))\cap B_{1/2}(\bo)) \leq C \cdot (\mu^j)^{\fm - k -\epsilon}.	
\end{align}

\end{theorem}
Since we have mentioned in Remark \ref{r:inclusion} that 
\begin{align}
	\mathcal{S}^k(U) = \bigcup\limits_{\epsilon > 0}\bigcap\limits_{r > 0} 	\mathcal{S}_{\epsilon, r}^k(U),	
\end{align}
we also notice that $\mathcal{S}(U) = \mathcal{S}^{\fm - 2}(U)$, where $\mathcal{S}(U)$ is restriction of the singular set of $U$ on $\dR^n$. 
To obtain the dimension estimates for the singular set, we need the following lemma. 
\begin{lemma}\label{l:almost-(m-1)-symmetric}
	Let $f: B_1(\bo)\to \dR$ be $\gamma$-harmonic and let $U$ be the unique solution of \eqref{e:symmetric-extension} such that $\mathcal{N}_U(x_+,r) \leq \Lambda$ with $x_+ \equiv (x, 0)\in \dR^{n + 1}$. Then for every $\epsilon > 0$, $k\in\dN$, and $\alpha\in(0,1)$, there exists 
	$\eta = \eta (\epsilon, n, k, \Lambda, \fa) > 0$ such that if 
	$U$ is $(\fm - 1, \eta, r, \fa)$-symmetric at $x_+$, then 
	\begin{align}
		\|\cT_{x_+,r} U - L \|_{C^{1,\alpha}(\cB_{1/2}(\bo_+))} < \epsilon,
	\end{align}
	where $L$ is a linear polynomial with $\fint_{\p\cB_1(\bo_+)}|L|^2d\sigma = 1$.
	In particular, choosing any $\epsilon < \epsilon_0 \equiv |\nabla L|/3$, we have $r_{x_+} \geq r$ if $U$ is $(\fm - 1, \epsilon, r_{x_+} , \fa)$ at $x_+$. 
\end{lemma}
	
\begin{proof}
	The proof is standard and follows from a usual contradiction compactness argument, so we omit the details. 
\end{proof}

Then Theorem \ref{t:volume-estimate-model} immediately implies dimension estimates on the quantitative singular strata.
\begin{corollary}Let $U$ be the unique solution of \eqref{e:symmetric-extension}. Under the same assumptions as Theorem \ref{t:volume-estimate-model}, we have    
\begin{align}
 	\dim_{\mathcal{H}}(\cC(U) \cap B_{1/2}(\bo) ) \leq \dim_{\Min}(\cC(U) \cap B_{1/2}(\bo) ) \leq \fm-2.
	\label{e:singular-set-dimension-of-U} 
	\end{align}

\end{corollary}

\begin{proof}
 Applying Lemma \ref{l:almost-(m-1)-symmetric}, we have that for sufficiently small $\epsilon > 0$, 
\begin{align}
\cC(U) \subset \mathcal{S}_{\epsilon, r}^{\fm - 2}(U).
\end{align}
Then applying Theorem \ref{t:volume-estimate-model} and the conclusion follows.
\end{proof}

 Now we proceed to prove Theorem \ref{t:volume-estimate-model}.  
The proof follows from the following effective covering result, Proposition \ref{p:effective-covering}. Cheeger-Naber-Valtorta proved it in the second-order elliptic case; see \cite{Cheeger-Naber-Valtorta}.

\begin{proposition}
[Effective covering] \label{p:effective-covering} We fix a constant $\mu\in(0,1)$. For every $\epsilon>0$, $n \geq 2$ and $\Lambda>0$, there exist uniform constants $C_0=C_0(n)$, $C_1=C_1(n)$ and $D=D(n,\epsilon,\Lambda)$
such that the following properties hold for every $j\in\dN$.
\begin{enumerate}
\item $\cS_{\epsilon,\mu^j}^k(U) \cap \cB_{1/2}(\bo_+)$ is contained in the union of at most $j^D$ nonempty open sets $\fA_{\epsilon,\mu^j}^k$.

\item Each set $\fA_{\epsilon,\mu^j}^k$ is the union of at most $(C_1\mu^{-n-1})^D\cdot (C_0\mu^{-k})^{j-D}$ balls of radius $\mu^j$.	
\end{enumerate}

\end{proposition}

Let us make the following notations. 
\begin{align}
\cE(U, x, r) = \inf\{\eta\geq 0| \ U \ \text{is}\ (0,\eta, r)-\text{symmetric at}\ x\}.	
\end{align}
The points in $\cB_1(\bo_+) \cap \dR^n$ are divided in {\it good points} and {\it bad points} in terms of the quantitative symmetry:
\begin{align}
\fG_{r,\epsilon}(U)\equiv\{x\in \cB_{1/2}(\bo_+) \cap \dR^n | \cE(U, x, r) < \epsilon\},
\\
\fH_{r,\epsilon}(U)\equiv\{x\in \cB_{1/2}(\bo_+) \cap \dR^n | \cE(U, x, r) \geq \epsilon\}.
\end{align}
Let us denote by $\overline{T}^j\equiv(\overline{T}_1^j, \ldots, \overline{T}_{\ell}^j, \ldots ,\overline{T}_j^j)$ a $j$-tuple, where $\overline{T}_{\ell}^j\in\{0,1\}$ for $1\leq \ell \leq j$. The norm $|\overline{T}^j|$ of a $j$-tuple $\overline{T}^j$ is defined to be $\sum\limits_{\ell=1}^j \overline{T}_{\ell}^j$. For fixed $\mu>0$, $\epsilon>0$ and $j\in\dZ_+$,
we use the map $T^j$ to assign each point $ x\in \cB_1(\bo_+) \cap \dR^n$ to a $j$-tuple $(T_1^j(x), \ldots, T_j^j(x))$ such that for each $1\leq \ell \leq j$,
\begin{align}
T_{\ell}^j(x)\equiv \begin{cases}
	1, & x\in  \fH_{\mu^{\ell},\epsilon}(U),
	\\
	0, & x\in \fG_{\mu^{\ell},\epsilon}(U).
\end{cases}	
\end{align}
Then we label all the points in $\cB_{1/2}(\bo_+) \cap \dR^n$ by $j$-tuples: for each $j$-tuple $\overline{T}^j$ as described above, we define 
\begin{align}
E(\overline{T}^j) \equiv \{x\in \cB_{1/2}(\bo_+) \cap \dR^n | T^j(x) = \overline{T}^j\}.
\end{align}

The proof of Proposition \ref{p:effective-covering} follows from the key lemmas below. 
\begin{lemma}
	There exists $D=D(\epsilon, \mu, \Lambda, n, \fa) > 0$ such that $E(\overline{T}^j) = \emptyset$ if $|\overline{T}^j| \geq D$. 
\end{lemma}

\begin{proof}
This lemma follows from the uniform estimate on the number of bad scales, which is given by  Corollary \ref{c:controlling-bad-scales}.
\end{proof}

\begin{lemma}
Given a positive integer $j\in\dZ_+$ and a $j$-tuple $\overline{T}^j = (\overline{T}_1^j,\ldots, \overline{T}_j^j)$ with $\overline{T}_{\ell}^j\in\{0,1\}$.
For every  $1 \leq \ell \leq j$, the set 
 \begin{align}\cA_{\ell}\equiv \cS_{\epsilon, \mu^j}^k(U) \cap \cB_{\mu^{\ell-1}}(x) \cap E(\overline{T}^j)\end{align} admits the following effective coverings. 
 \begin{enumerate}
  
 \item 	If $\overline{T}_{\ell}^j = 1$, then $\cA_{\ell}$ can be covered by $C_1(n)\cdot\mu^{-n-1}$
 balls centered in $\cA_{\ell}$ of radius $\mu^{\ell}$.

  \item If $\overline{T}_{\ell}^j = 0$, then $\cA_{\ell}$ can be covered by $C_0(n)\cdot\mu^{-k}$
 balls centered in $\cA_{\ell}$ of radius $\mu^{\ell}$.

 \end{enumerate}
  
\end{lemma}

\begin{proof}
The proof of item (1) is trivial. 
 We will prove item (2). 

We claim that $\cA_{\ell} \subset B_{\frac{\mu^{\ell}}{10}}(V^k) \cap \cB_{\mu^{\ell - 1}}(x)$ for some vector space $V^k$ of dimension $k$.  If the claim is true, 
then the conclusion immediately follows since $B_{\frac{\mu^{\ell}}{10}}(V^k) \cap \cB_{\mu^{\ell - 1}}(x)$ can be covered by $C_0(n)\cdot \mu^{-k}$ balls of radius $\mu^{\ell}$.
Now let us prove the claim. Suppose the inclusion is not true, that is,  there exists a point $z\in \cA_{\ell}$ that satisfies $z \in \cB_{\mu^{\ell - 1}}(x)\setminus B_{\frac{\mu^{\ell}}{10}}(V^k)$. It follows from Corollary \ref{c:inductive-splitting} that $U$ is $(k+1, \epsilon, \mu^{\ell - 1})$-symmetric at $z$ which contradicts the fact $z\in \cA_{\ell} \subset \cS_{\epsilon, \mu^j}^k(U)$. This completes the proof of the claim. 
\end{proof}

\subsection{A new $\epsilon$-regularity result}
 \label{ss:epsilon-regularity}

The main result of this subsection is an $\epsilon$-regularity result (Theorem \ref{t:eps-regularity-smooth-approximation}). This result is a key ingredient to prove the Hausdorff measure estimate in Section \ref{ss:Hausdorff-measure}.

 We will start with a technical lemma in the two-dimensional case, which shows that the origin   is an isolated critical 
point of our interested model polynomial.

\begin{lemma}\label{l:hypergeometric-polynomial-critical-point}
Let $P$ be a non-constant homogeneous polynomial that solves \eqref{e:symmetric-extension} when $n=1$. Let us write $\nabla P = (P_x, P_y)$. Then $\bo_+\equiv(0,0)\in\dR^2$ is an isolated solution of $\nabla P = \bo_+$. 	
\end{lemma}

\begin{proof}
Let $k=\deg(P)$. There are two cases to analyze. 
\begin{flushleft}
{\bf Case (1):} $k$ is even. \end{flushleft}

By \cite[proposition 4.13]{STT}, the polynomial
$P(x,y)$ yields the form 
\begin{align}
P(x,y) = \Upsilon_e(k,\fa) \cdot \Phi\left(-\frac{k}{2}, \frac{1-k-\fa}{2} , \frac{1}{2} , - \frac{x^2}{y^2} \right) y^k,
\end{align}
where $\Phi(\alpha,\beta,\gamma, z)$ is the hypergeometric function and the normalization constant $\Upsilon_e(k,\fa)>0$ is chosen such that $\fint_{\p\cB_1(\bo_+)}y^{\fa} \cdot |P|^2  d\sigma = 1$. Computing the gradient of $P$, we have 
\begin{align}
P_x &\ = -2 \Upsilon_e \cdot  xy^{k-2}\cdot \frac{d\Phi}{dz},
\\
P_y &\ = 2 \Upsilon_e \cdot  x^2 y^{k-3} \cdot \frac{d\Phi}{dz} + k \Upsilon_e \cdot y^{k-1} \cdot \Phi. 
\end{align}
Let denote by $Z(P_x)$ and $Z(P_y)$ the zero sets of $P_x$
and $P_y$, respectively. It is evident that both $Z(P_x)$ and $Z(P_y)$ are unions of lines passing through the origin of $\dR^2$. 
We will prove that $\bo_+\in\dR^2$ is the only zero
of $\nabla P$. 

First, it is not hard to see that 
\begin{align}\{(x,y)\in\dR^2 | x y = 0\} \cap \{(x,y) \in \dR^2| \nabla P(x,y) = 0\} = (0,0).\end{align}
Indeed, since $P$ is homogeneous and it is even about $\{y=0\}$, $P$ can be  written as 
\begin{align}P(x,y) = a_0 x^k + \sum\limits_{i=1}^{k-1} a_i x^{k - 2i} y^{2i} + a_k y^k.\end{align}
Then $x=0$ and $P_y(x,y) = 0$
force $y=0$. Similarly, $y=0$ and $P_x(x,y) = 0$ force $x=0$.

Next, we consider the case  $xy \neq 0$. Let $|\nabla P| = 0$ along a line $\{y = \tau x\}$ for some $\tau\neq 0$. Then we have 
\begin{align}
\Phi(-\tau^2)  = \frac{d\Phi}{dz}\Big|_{z=-\tau^2} = 0.
\end{align}
Since $\Phi$ satisfies 
\begin{align}
z(1-z)\frac{d^2 \Phi}{dz^2}
+(\gamma - (\alpha + \beta + 1)z) \frac{d\Phi}{dz} - \alpha\beta \Phi = 0,	
\end{align}
 we have $\frac{d^2\Phi}{dz^2}\big|_{z = -\tau^2} = 0$. By simple induction argument, we have that $\frac{d^m \Phi}{dz^m}\big|_{z = -\tau^2} = 0$ for any $m\in\dZ_+$. This implies that $|\nabla^m P| = 0$ along $\{y = \tau x\}$ for any $m\in\dN_0$, and thus $P\equiv 0$ on $\dR^2$. Contradiction.

\begin{flushleft}
{\bf Case (2):} $k$ is odd. \end{flushleft}
 The polynomial $P$ in this case satisfies
\begin{align}
P(x,y) = \Upsilon_o(k,\fa) \cdot \Phi\left(\frac{1-k}{2}, \frac{2-k-\fa}{2} , \frac{3}{2} , - \frac{x^2}{y^2} \right) x y^{k-1},
\end{align}
the normalization constant $\Upsilon_e(k,\fa)>0$ is chosen such that $\fint_{\p\cB_1(\bo_+)}|P|^2  d\sigma = 1$.
 Computing the gradient of $P$, we have 
\begin{align}
P_x &\ = -2 \Upsilon_o \cdot  x^2 y^{k-3}\cdot \frac{d\Phi}{dz} + \Upsilon_o \cdot y^{k-1} \cdot \Phi ,
\\
P_y &\ = 2 \Upsilon_e \cdot  x^3 y^{k-4} \cdot \frac{d\Phi}{dz} + (k - 1) \Upsilon_e \cdot xy^{k-2} \cdot \Phi. 
\end{align}
Using the similar arguments 
as in Case (1), we have that $(0,0)$ is the only zero of $\nabla P$.
\end{proof}

\begin{theorem}\label{t:eps-regularity-smooth-approximation}
Let $P$ be a non-constant $(\fm-2)$-symmetric  homogeneous polynomial of degree $d$ that solves \eqref{e:symmetric-extension}. Then there exist positive constants $\delta> 0$ and $r > 0$ such that for any $u \in C^{2d^2}(\cB_1(\bo_+))$ if 
\begin{align}
|u - P|_{C^{2d^2}(\cB_1(\bo_+))} < \delta,	
\end{align}
then 
$\mathcal{H}^{\fm - 2}(\cC(u) \cap \cB_r(\bo_+)) \leq C(\fm) (d - 1)^2 r^{\fm -2}$.	
 
\end{theorem}

The proof of the proposition requires an important stability result in \cite[theorem 4.1]{HHL}.  
\begin{theorem}[\cite{HHL}]
\label{t:HHL-stability}
Let $P:\dR^n \to \dR^n$ be a   mapping with $f(\bo) = \bo$ so that each component is a homogeneous polynomial of degree $d$, where $\bo\in\dR^n$. Assume that $P$ can be extended to a holomorphic map from $\dC^n$ to $\dC^n$ and that the origin of $\dC^n$ is an isolated zero. Then there exist positive constants $\delta_0>0$, $k\in\dZ_+$, $r>0$ depending on $P$, such that if $U \in C^k(B_1(\bo), \dR^n)$ 
satisfies 
$\| U - P \|_{C^{2d^n}(B_1(\bo))} < \delta_0$,	
then $\#\{\cC(U) \cap B_r(\bo)\} \leq d^n$.	
\end{theorem}

\begin{proof}[Proof of Theorem \ref{t:eps-regularity-smooth-approximation}]	Let $P = P(x, y)$ and let us denote $x_1 = x$, $x_2 = y$, and $\bm{z} = (x_1 , x_2)$.
Since the partial derivatives $P_{x_1}$ and $P_{x_2}$ are homogeneous polynomials,
	one has the decompositions 
	\begin{align}
	P_{x_1}(\bm{z}) = & \ \left(\prod\limits_{k = 1}^{d_1} \langle \nu_k^1 , \bm{z}\rangle\right)\cdot\left(\prod\limits_{\mu = 1}^{q_1} Q_{\mu}^1(\bm{z})\right),
	\\
	P_{x_2}(\bm{z}) = & \ \left(\prod\limits_{k = 1}^{d_2} \langle \nu_k^2 , \bm{z}\rangle\right)\cdot\left(\prod\limits_{\mu = 1}^{q_2} Q_{\mu}^2(\bm{z})\right).	
	\end{align}
 Here each $\nu_k^{\alpha}$ is a vector in $\dR_{x_1,x_2}^2$ and each $Q_{\ell}^j$ is a quadratic form with the only zero $\bm{z} = 0$.
Now restricted onto the plane $\dR_{x_1,x_2}^2$, applying Lemma \ref{l:hypergeometric-polynomial-critical-point}, we find that $\nabla P = (P_{x_1}, P_{x_2})$ is vanishing only at $\bm{z} = 0$. Therefore, 
the vectors $\nu_k^{\alpha}$ and $\nu_{\ell}^{\beta}$
	are linearly independent unless $(k, \alpha) = (\ell, \beta)$.

	Given the expression of $\nabla P$ as above, we will prove the following claim.

	\vspace{0.3cm}
	
{\bf Claim.} There exists an open dense subset $\mathfrak{T}\subset SO(\fm)$ such that the following holds: for any rotation $\mathscr{T}= (A_{ij})_{1\leq i,j\leq \fm} \in \mathfrak{T}$
with  $\bm{x} = \mathscr{T}\bm{w}\in\dR^{\fm}$, in terms of the new coordinate system $\bm{w}=(w_1,\ldots, w_{\fm})$, 
for every coordinate plane
\begin{align}\mathcal{L}_{ij}\equiv \{\bm{w}=(w_1,\ldots, w_{\fm})\in\dR^{\fm}: w_{\ell} = 0\ \text{if} \ \ell\neq i \ \text{and}\ \ell \neq j\}, \end{align} the only zero of the restriction $(\nabla P)|_{\mathcal{L}_{ij}}$ 
on the plane $\mathcal{L}_{ij}$ is the origin.  
	
	\vspace{0.3cm}

To prove the claim, let us take a rotation $\mathscr{T}= (A_{ij})_{1\leq i,j\leq \fm} \in SO(\fm)$
with  $\bm{x} = \mathscr{T}\bm{w}\in\dR^{\fm}$. Then the partial derivative $P_{w_i}$ with respect to the new coordinates $\bm{w}=(w_1,\ldots, w_{\fm})$ is given by,
\begin{align}
 P_{w_i}(w) = &\ (D_{x_1}P) \cdot \frac{\p x_1}{\p w_i} + (D_{x_2}P) \cdot \frac{\p x_2}{\p w_i}\nonumber\\
 = & \  A_{1i}\left(\left(\prod\limits_{k = 1}^{d_1}\sum\limits_{\ell = 1}^{\fm}\langle\nu_k^1, \zeta_{\ell} \rangle w_{\ell}\right)\cdot\prod\limits_{\mu = 1}^{q_1}(\mathscr{R}\bm{w})^T\mathscr{Q}_{\mu}^1 (\mathscr{R}\bm{w})\right) \\
 & \  + A_{2i}\left(\left(\prod\limits_{k = 1}^{d_2}\sum\limits_{\ell = 1}^{\fm}\langle\nu_k^2, \zeta_{\ell} \rangle w_{\ell}\right)\cdot\prod\limits_{\mu = 1}^{q_2}(\mathscr{R}\bm{w})^T\mathscr{Q}_{\mu}^2 (\mathscr{R}\bm{w})\right),
\end{align}
where $\mathscr{R} \equiv 
\begin{bmatrix}
	A_{11} & \ldots & A_{1{\fm}}
	\\
		A_{21} & \ldots & A_{2{\fm}}
\end{bmatrix}$ is the first two rows of $(A_{ij})$, $\zeta_{\ell}$ is its $\ell^{th}$ column vector $\begin{bmatrix}
A_{1\ell}
\\
A_{2\ell}	
\end{bmatrix}
$, $\mathscr{Q}_j^{\alpha}$ $(\alpha=1,2)$ is 
the matrix of the quadratic form $Q_j^{\alpha}$. Restricted on the coordinate plane $\mathcal{L}_{ij}$, we have
\begin{align}
\nabla P|_{\mathcal{L}_{ij}}
= \mathscr{R}_{ij}
\begin{bmatrix}
	\left(\prod\limits_{k = 1}^{d_1} \left(\langle\nu_k^1, \zeta_i \rangle w_i + \langle\nu_k^1, \zeta_j \rangle w_j\right)\right)\cdot\prod\limits_{j = 1}^{q_1}(\mathscr{R}_{ij}\bm{w}_{ij})^T\mathscr{Q}_j^1 (\mathscr{R}_{ij}\bm{w}_{ij})
	\\
\left(\prod\limits_{k = 1}^{d_2}\left(\langle\nu_k^2, \zeta_i \rangle w_i + \langle\nu_k^2, \zeta_j \rangle w_j\right)\right)\cdot\prod\limits_{j = 1}^{q_2}(\mathscr{R}_{ij}\bm{w}_{ij})^T\mathscr{Q}_j^2 (\mathscr{R}_{ij}\bm{w}_{ij})	
\end{bmatrix},
\end{align}
	where $\bm{w}_{ij} = 
(0,\ldots,
 w_i,
\ldots, w_j, \ldots, 0
)$ and 
\begin{align}
\mathscr{R}_{ij} \equiv 
	\begin{bmatrix}
A_{1i} & A_{2i}
\\
A_{1j} & A_{2j}
\end{bmatrix}.\label{e:submatrix-R_{ij}}
\end{align}

Now the subset $\mathfrak{T}\subset \SO(\fm)$ is defined as: a  rotation matrix $\mathscr{T} = (A_{ij})_{1\leq i,j\leq \fm}$  is contained in $ \mathfrak{T}$ if each sub-matrix $\mathscr{R}_{ij}$, as in \eqref{e:submatrix-R_{ij}}, is nonsingular.  
Then $\mathfrak{T}\subset \SO(\fm)$
	is an open subset. Moreover, 
$\dim(\SO(\fm)\setminus \mathfrak{T}) = \frac{\fm(\fm - 1)}{2} - 1$. Therefore, $\mathfrak{T}\subset \SO(\fm)$ is an open dense subset.

Next, we will show that, for every $\mathscr{T}\in \mathfrak{T}$ with $\bm{x} = \mathscr{T}\bm{w}$ and for every $\bm{w}$-coordinate plane $\mathcal{L}_{ij}$, the only zero of $\nabla P|_{\mathcal{L}_{ij}}$ is the origin.
Indeed, by the definition of $\mathfrak{T}$, the matrix $\mathscr{R}_{ij}$ is always nonsingular, so
 $\nabla P|_{\mathcal{L}_{ij}}$
if and only if 
\begin{align}
\begin{split}
	\left(\prod\limits_{k = 1}^{d_1} \left(\langle\nu_k^1, \zeta_i \rangle w_i + \langle\nu_k^1, \zeta_j \rangle w_j\right)\right)\cdot\prod\limits_{j = 1}^{q_1}(\mathscr{R}_{ij}\bm{w}_{ij})^T\mathscr{Q}_j^1 (\mathscr{R}_{ij}\bm{w}_{ij})
= & \ 	0,
\\
\left(\prod\limits_{k = 1}^{d_2}\left(\langle\nu_k^2, \zeta_i \rangle w_i + \langle\nu_k^2, \zeta_j \rangle w_j\right)\right)\cdot\prod\limits_{j = 1}^{q_2}(\mathscr{R}_{ij}\bm{w}_{ij})^T\mathscr{Q}_j^2 (\mathscr{R}_{ij}\bm{w}_{ij})	 = & \ 0.
\end{split}\label{e:product-vanishing}
\end{align}
Since $\mathscr{Q}_j^{\alpha}$ is nonsingular,  for any $\alpha=1,2$ and $1\leq j\leq q_{\alpha}$,
$(\mathscr{R}_{ij}\bm{w}_{ij})^T\mathscr{Q}_j^{\alpha} (\mathscr{R}_{ij}\bm{w}_{ij}) = 0$ if and only if $\bm{w}_{ij} = 0$. It follows that \eqref{e:product-vanishing} if and only if 
\begin{align}
\prod\limits_{k = 1}^{d_1} \left(\langle\nu_k^1, \zeta_i \rangle w_i + \langle\nu_k^1, \zeta_j \rangle w_j\right) = \prod\limits_{k = 1}^{d_2}\left(\langle\nu_k^2, \zeta_i \rangle w_i + \langle\nu_k^2, \zeta_j \rangle w_j \right)= 0.	
\end{align}
Since $(\nu_k^1, \nu_{\ell}^2)$
is pair of linearly independent vectors for any $k\neq \ell$, and $(\zeta_i, \zeta_j)$ is also a pair of linearly independent vectors for any $i\neq j$, simple calculations in linear algebra imply that $w_i = w_j = 0$. This completes the proof of the claim.

\vspace{0.3cm}

Now we are ready to prove the Hausdorff measure estimate. First, let us denote 
\begin{align}
	DP(i,j) \equiv (D_i P, D_j P)|_{\mathcal{L}_{ij}}.
\end{align}
By the claim, the origin $\bm{0}^{\fm} \in \dC^{\fm}$ is also the only zero of the complex extension of $DP(i,j)$. Let $u\in C^{2d^2}(\cB_1(\bo_+))$ satisfy $|u - P|_{C^{2d^2}(\cB_1(\bo_+))} < \delta$ for some sufficiently small $\delta > 0$ such that the following holds: for any $p \in \cB_1(\bo_+)$, 
\begin{align}
|D_pu(i, j) - DP(i, j)|_{C^{2(d-1)^2}(\cB_1(\bo_+))} < \delta_0,\quad \forall\ 1\leq i,j\leq \fm,	
\end{align}
where 
$D_pu(i, j) \equiv (D_i u, D_ju)|_{\mathcal{L}_{ij}(p)}$ 	
 and $\mathcal{L}_{ij}(p) \equiv p + \mathcal{L}_{ij}$ is a translation of the plane $\mathcal{L}_{ij}$.
Then applying Theorem \ref{t:HHL-stability}, there exists some small $r > 0$ such that for all $p$,
\begin{align}
\#\left((\nabla u)^{-1}(\bo_+)\cap \mathcal{L}_{ij}(p) \cap \cB_r(\bo_+)\right) \leq \#\left((D_pu(i, j))^{-1}(\bm{0}^2)\cap \cB_r(\bo_+)\right)	\leq (d - 1)^2. 
\end{align}
To estimate the Hausdorff measure of $(\nabla u)^{-1}(\bo_+) \cap B_r(\bo_+)$, let us take the natural projection $\pi_{ij}:\dR^{\fm} \to \dR^{\fm - 2}$ by
\begin{align}
\pi_{ij}(\bx) \equiv (x_1,\ldots,\hat{x}_i,\ldots, \hat{x}_j, \ldots x_{\fm}) \in\dR^{n-2}, 	
\end{align}
where $x_i$ and $x_j$ are deleted from $\bx = (x_1,\ldots, x_{\fm}) \in \dR^n$.
Then for every $q\in B_r(\bm{0}^{\fm - 2}) \subset \dR^{\fm - 2}$, we have that 
\begin{align}
	\#\left((\nabla u)^{-1}(\bo_+)\cap (\pi_{ij})^{-1}(q) \cap \cB_r(\bo_+)\right) \leq (d - 1)^2. 
\end{align}
Therefore, applying the area formula (see \cite[3.2.22]{Federer} or \cite[theorem 1.2.10]{Han-Lin}), 
\begin{align}
&\ \mathcal{H}^{\fm - 2}\left((\nabla u)^{-1}(\bo_+) \cap \cB_r(\bo_+)\right)	
\\
\leq &\ \sum\limits_{1\leq i < j  \leq \fm}
\int_{B_r(\bm{0}^{\fm - 2})}
\#\left((\nabla u)^{-1}(\bo_+)\cap (\pi_{ij})^{-1}(q) \cap \cB_r(\bo_+)\right) d\mathcal{H}^{\fm - 2}(q)
\nonumber\\
\leq & \ C(\fm) (d - 1)^2 r^{\fm -2},  
\end{align}
which completes the proof of the proposition.
\end{proof}

\begin{theorem}
[$\epsilon$-regularity] \label{t:eps-reg-Euclidean}
Let $U$ solve \eqref{e:symmetric-extension} and satisfy $\cN_U(\bo_+ , 1) \leq \Lambda$.
Then there exist constants $\epsilon(\fm,\fa,\Lambda) > 0$
and $\bar{r}(\fm,\fa,\Lambda) > 0$
such that if there exists a  
homogeneous polynomial $P$ with $(\fm - 2)$-symmetry such that 
\begin{align}
\fint_{\p\cB_1(\bo_+)} |y|^\fa P^2 d\sigma = 1\quad \text{and} \quad 
\|\cT_{\bo_+,1}U - P\|_{L^2(\p\cB_1(\bo_+))} \leq \epsilon,	
\end{align}
then for all $r\leq \bar{r}$, 
\begin{align}
\cH^{\fm - 2}(\cC(U) \cap \cB_r(\bo_+))
\leq C(n,\Lambda) r^{\fm - 2},	
\end{align}
	where $\cC(U)\equiv \{\bx\in \dR^n: |\nabla U|(\bx) = 0\} \subset \dR^n$ is the restriction of the critical set of $U$ on the boundary $\dR^n$.
\end{theorem}

\begin{proof}
It follows from Theorem \ref{t:eps-regularity-smooth-approximation}.	
\end{proof}

\subsection{Hausdorff measure estimates for the critical sets}
\label{ss:Hausdorff-measure}

In this subsection, we will prove the following uniform Hausdorff measure estimate for the critical set of the solutions of \eqref{e:symmetric-extension}.

\begin{theorem}\label{t:Hausdorff-measure-estimate-Euclidean}
Let $U$ solve \eqref{e:symmetric-extension} with $\cN_U(\bo_+,1)\leq \Lambda$. There exists a constant $C=C(\fm,\Lambda,\fa)>0$ such that 
\begin{align}
\cH^{\fm - 2}(\cC(U) \cap \cB_{1/2}(\bo_+)) \leq C(\fm,\Lambda,\fa),	
\end{align}
		where $\cC(U)\equiv \{\bx\in \dR^n: |\nabla U|(\bx) = 0\} \subset \dR^n$ is the restriction of the critical set of $U$ on the boundary $\dR^n$.
\end{theorem}

\begin{proof}
First, Lemma \ref{l:almost-(m-1)-symmetric} implies that there exists some $\epsilon>0$ such that
\begin{align}\cC(U)\subset \cS_{\epsilon, 1}^{\fm-2}(U).\end{align}
Then let us write
\begin{align}
	\cC(U) \cap \cB_{1/2}(\bo_+) = \cC(U) \cap \Big( \left(\cS_{\epsilon, 1}^{\fm-2}(U) \setminus  \cS_{\epsilon, 1}^{\fm-3}(U)\right) \cup \cS_{\epsilon, 1}^{\fm-3}(U)\Big) \cap \cB_{1/2}(\bo_+). \label{e:CU-symmetry}
\end{align}
By definition, one can obtain the following dyadic decomposition of $\cS_{\epsilon, 1}^{\fm-3}(U)$. 
\begin{align}
\cS_{\epsilon, 1}^{\fm-3}(U) = \left( \bigcup\limits_{j=1}^{\infty} \left(\left( \cS_{\epsilon, \mu^j}^{\fm-2}(U) \setminus \cS_{\epsilon, \mu^j}^{\fm-3}(U) \right) \bigcap \cS_{\epsilon, \mu^{j-1}}^{\fm-3}(U) \right)\right) \bigcup  \left(\bigcap\limits_{j=1}^{\infty}	\cS_{\epsilon, \mu^j}^{\fm-3}(U) \right). \label{e:(m-3)-singular-stratum}
\end{align}

Now let us define 
\begin{align}
&\cC^{(0)}(U) = \cC(U) \cap \left(\cS_{\epsilon, 1}^{\fm-2}(U) \setminus \cS_{\epsilon, 1}^{\fm-3}(U)\right) \cap \cB_{1/2}(\bo_+),
\\
&\cC^{(j)}(U) = \cC(U) \cap \left(\cS_{\epsilon, \mu^j}^{\fm-2}(U) \setminus \cS_{\epsilon, \mu^j}^{\fm-3}(U)\right) \cap \cS_{\epsilon, \mu^{j-1}}^{\fm-3}\cap \cB_{1/2}(\bo_+), \quad j\in\dZ_+.
\end{align}
Combining \eqref{e:CU-symmetry} and \eqref{e:(m-3)-singular-stratum}, $\cC(U)$ has a further decomposition, 
\begin{align}
\cC(U) \cap \cB_{1/2}(\bo_+) 
= \left( \bigcup\limits_{j=0}^{\infty}
	\cC^{(j)}(U)\right) \bigcup \left(\cC(U)\bigcap\limits_{j=1}^{\infty}\cS_{\epsilon,\mu^j}^{\fm-3}(U) \right).
\end{align}
If follows from Theorem \ref{t:volume-estimate-model} that 
\begin{align}
\cH^{\fm-2}\left( \cC(U)\bigcap\limits_{j=1}^{\infty}\cS_{\epsilon,\mu^j}^{\fm-3}(U) \right) = 0.	
\end{align}
Indeed, for any $j$, the singular stratum $\cS_{\epsilon,\mu^j}^{\fm-3}(U) $ is covered by $C(\fm,\Lambda,\epsilon)(\mu^j)^{3-\fm-\epsilon}$ balls of radius $\mu^j$ and centered at $x_{\alpha}\in \cS_{\epsilon,\mu^j}^{\fm-3}(U)$. Then 
\begin{align}
\cH^{\fm-2}\left( \cC(U)\bigcap\limits_{j=1}^{\infty}\cS_{\epsilon,\mu^j}^{\fm-3}(U) \right) \leq &\ \cH^{\fm-2}(\cS_{\epsilon,\mu^j}^{\fm-3}(U)) \nonumber\\
\leq &\  C(\fm,\Lambda,\epsilon) (\mu^j)^{\fm-2} \cdot (\mu^j)^{3-\fm-\epsilon}	
\nonumber\\
= &\  C(\fm,\Lambda,\epsilon) (\mu^j)^{1-\epsilon}.
\end{align}
Therefore, 
\begin{align}\cH^{\fm-2}\left( \cC(U)\bigcap\limits_{j=1}^{\infty}\cS_{\epsilon,\mu^j}^{\fm-3}(U) \right)  = 0.	
\end{align}

Then we will prove that for any $k\geq 0$,
\begin{align}
	\cH^{\fm-2} \left( \bigcup\limits_{j=0}^k
	\cC^{(j)}(U)\right) \leq C(\Lambda, \fm ,\epsilon) \sum\limits_{j=0}^k \mu^{(1-\epsilon)j}. \label{e:measure-union-of-Cj}
\end{align}

By Theorem \ref{t:eps-reg-Euclidean}, we find that the estimate holds for $k = 0$. 
Next, we will prove \eqref{e:measure-union-of-Cj} in the general case. Let us choose a covering consisting of finite balls, \begin{align}\mathcal{D}(\mu^j\cdot \bar{r}) \equiv \left\{\cB_{\mu^j \cdot \bar{r}}(x_{\alpha}) | \text{ for some } x_{\alpha}\in \cC^{(j)}(U)\right\}	
\end{align} of 
$\cC^{(k)}(U)$ such that for any $
\alpha \neq \beta $, 
\begin{align}\cB_{\frac{\mu^j \cdot \bar{r}}{10}}(x_{\alpha}) \cap \cB_{\frac{\mu^j \cdot \bar{r}}{10}}(x_{\beta}) = \emptyset.\end{align} 	
	Then by Theorem \ref{t:volume-estimate-model}, we have that $\#\mathcal{D}(\mu^k\cdot \bar{r}) \leq C(\fm, \Lambda, \epsilon) \cdot (\mu^j \cdot \bar{r})^{(3-\fm-\epsilon)}$. By the definition of $\cC^{(j)}$, for each $x_{\alpha}$, there exists a scale $r\in [\mu^j, \mu^{j-1}]$
	such that for some normalized homogeneous polynomial $P$ of two variables, we have that 
	\begin{align}
	\int_{\p \cB_1(\bo_+)} |y|^{\fa} (\cT_{\bx_{\alpha}, r}(U) - P)^2 \dvol_{\p \cB_1(\bo_+)} < \epsilon.	
	\end{align}
Since $U$ is a solution of the boundary value problem \eqref{e:symmetric-extension}, using a simple compactness \& contradiction 
	argument,  one can replace $P$ with a homogeneous  polynomial that solves \eqref{e:symmetric-extension}. 
Applying Theorem \ref{t:eps-reg-Euclidean}, we have that 
\begin{align}
\cH^{\fm-2}(\cC(U) \cap B_{\mu^j\cdot \bar{r}}(x_{\alpha})) \leq 	C(\fm,\Lambda) \cdot (\mu^j\cdot \bar{r})^{\fm-2}.
\end{align}
	Therefore, for any $j\in\dZ_+$,
	\begin{align}
\cH^{\fm-2}(\cC^{(j)}((U))) \leq C(\Lambda, \fm, \epsilon) \cdot \mu^{(1-\epsilon)j},	
	\end{align}
which completes the proof. 
\end{proof}

\subsection{Singular set on the boundary}
\label{ss:boundary-estimate}

This subsection considers the structure of the singular set of a function $f$ that is $\gamma$-harmonic on $B_2(\bo)\subset \dR^n$. As stated in Theorem \ref{t:volume-estimate-boundary} and Theorem \ref{t:Hausdorff-measure-estimate-PE}, the full singular set $\cS(f)$ does not satisfy a codimension-$2$ estimate on $\dR^n$ due to the nonlocal nature of the fractional GJMS operator $P_{2\gamma}$. In the example presented in Example~\ref{ex-special-soln-U}, we have
\begin{align}
\mathcal{S}(f) = \{x \in \mathbb{R}^n : x_1 = 0\},
\end{align}
 and the dimension of $\mathcal{S}(f)$ is $\dim(\mathcal{S}(f)) = n-1$.

In the Euclidean model case, this phenomenon
has been observed in \cite{STT}. As in \cite[Section 8]{STT}, the points in $\cS(f)$ are divided into two classes depending upon the ``local nature" of the tangent map of $f$ at $x\in \cS(f)$. Let us recall that a codimension-$2$ estimate for $\cS(U)$ in $\overline{\dR^{\fm}}$ with $\fm = n + 1$ is based on a simple fact: if a single variable function $U$ is smooth on $\overline{\dR^{\fm}}$  and satisfies $\Div(y^{\fa}\nabla U) = 0$ in $\dR^{\fm}$, then $U$ must be a linear function on $\dR^n$ and independent of $y$.
As a consequence, $U$ does not have any critical point unless it is constantly vanishing. Motivated by this, we introduce the following notions.
\begin{definition}[Horizontal singular strata] Given a smooth nondegenerate function $f:B_1(\bo)\to \dR$ we define the $k^{th}$-horizontal singular stratum of $f$ by 
\begin{align}
	\uS^k (f) \equiv 
 \{x\in B_1(\bo): T_x (f) \ \text{satisfies} \ \Delta_{\dR^n} T_x (f) = 0\ \text{and is not}\ (k+1)\text{-symmetric}\},
\end{align}
 for any $0 \leq k \leq  n - 1$.
We also define the horizontal part of the critical set $\cC(f)$,
\begin{align}
		\uC(f) \equiv 
 \{x\in \cC(f)\cap B_1(\bo): T_x (f) \ \text{satisfies} \ \Delta_{\dR^n} T_x (f) = 0\}.
\end{align}
 Let us denote the complement $\fC(f) \equiv \cC(f) \setminus \uC(f)$ and $\fC^k(f) \equiv \cS^k(f) \setminus \uS^k(f)$. For the singular $\cS(f)$ of $f$, we define 
 \begin{align}
 	\uS(f) \equiv \cS(f) \cap \uC(f) \quad \text{and} \quad \fS(f) \equiv \cS(f) \setminus \uS(f).
 \end{align}
 For fixed parameters $r, \epsilon\in (0,1)$ and $0 \leq k \leq  n - 1$, one can define the  quantitative singular strata $\uS_{r,\epsilon}^k(f)$ and $\fS_{r,\epsilon}^k(f)$ of the critical/singular set of $f$ in a similar way.
\end{definition}

Based on the above definition, we have a simple observation.  If a tangent map $T_x(f)$ is $(n - 1)$-symmetric and satisfies $\Delta_{\dR^n}T_x (f) = 0$, then $T_x(f)$ must be linear, which implies $x \not\in \cC(f)$. This tells us $\uC(f) = \uS^{n - 2}(f)$ but in general  $\fC(f) \setminus \fC^{n - 2}(f) \neq \emptyset$. Therefore, one should establish codimension-$2$ estimates for the {\it horizontal singular set} instead of the whole singular set.

We give a quick lemma relating the quantitative symmetry of function on $\dR^{n+1}$ and its restriction on $\dR^n$ which enables us to understand the structure of the singular set of $f$ from $U$.
\begin{lemma}Given a smooth function $f$ that is $\gamma$-harmonic on $B_2(\bo)$, let $U$ be its Caffarelli-Silvestre type extension that solves \eqref{e:U-extension}. 
For every $\epsilon>0$, there exists $\delta=\delta(n,\epsilon,\Lambda, \fa)>0$
 such that if $U$ is $(k,\delta,s,\fa)$-symmetric at $x_+\equiv (x, 0)\in \cB_1(\bo_+)$, then $f$ is $(k - 1, \epsilon, s)$-symmetric at $x\in \dR^n$.	
\end{lemma}

\begin{proof}
We will	prove it by contradiction. Suppose there exists a constant $\epsilon_0>0$, a sequence $\delta_j\to 0$, and a sequence of solutions $U_j$ such that $U_j$ is $(k,\delta_j,s,\fa)$-symmetric at $\bx_j = (x_j, 0) \in \dR^{n+1}$ but $f_j$ is not $(k - 1,\epsilon_0, s)$-symmetric at $x_j$.
Without loss of generality, we can assume that $U_j$ satisfies 
\begin{align}
U_j(\bx_j) = 0 \quad \text{and} \quad \frac{1}{s^{\fa}}\fint_{\p \cB_1(\bo_+)}|y(\bx_j + s\xi)|^{\fa} \cdot U_j(\bx_j + s \xi)^2d\sigma(\xi) = 1.
\end{align}
By the definition of $(k,\delta_j,s,\fa)$-symmetry, we have that 
\begin{align}
\frac{1}{s^{\fa}}\fint_{\cB_1(\bx_j)}|y|^{\fa} \cdot |\cT_{\bx_j,s} U_j - P_j|^2 \dvol_{\dR^{n+1}} < \delta_j \to 0 	
\end{align}
for a sequence of $k$-symmetric normalized homogeneous polynomial $P_j$ with degree uniformly bounded by a constant depending upon $\Lambda$. 
Letting $j\to \infty$, we have that $\cT_{\bx_j, s} U_j$ converges to $\cT_{\bx_{\infty}, s} U_{\infty}\equiv P_{\infty}$, which is a $k$-symmetric polynomial. 

By the normalization condition, we have $U_{\infty} \equiv P_{\infty}$. Let $\underline{P}_{\infty}\equiv P_{\infty}|_{\dR^n}$ be the restriction of $P_{\infty}$ on $\dR^n$, which is apparently $(k-1)$-symmetric. Then 
the above implies that 
\begin{align} \frac{1}{s^{\fa}}\fint_{B_1(\bo)}|T_{x_j, s}f_j - \underline{P}_{\infty}|^2 \dvol_{\dR^n} \to 0.
\end{align}
The desired contradiction arises. \end{proof}

Then one can develop the same quantitative symmetry and quantitative cone splittings as in Section \ref{ss:quantitative-cone-splitting} for the horizontal singular strata. These technical preliminaries lead to the following result.  

\begin{theorem}\label{t:volume-estimate-model-boundary}
For every $n \geq 2$, $\gamma \in (0, 1)$, $\Lambda > 0$, and $\epsilon>0$, there exist $0<\mu(n,\epsilon, \gamma, \Lambda)<1$ and $C(n,\Lambda, \gamma, \epsilon)>0$ such that the following property holds. If $f \in C^{\infty}(\dR^n)$ is $\gamma$-harmonic on $B_1(\bo)\subset \dR^n$ and satisfies 
$\cN_f(\bo, 1)\leq \Lambda$, then for any $j\in \dZ_+$, 
\begin{align}
\Vol(T_{\mu^j}(\uS_{\epsilon,\mu^j}^k(f))\cap B_{1/2}(\bo)) \leq C \cdot (\mu^j)^{n - k -\epsilon}, \quad  0\leq k \leq n -2,
\\
	\Vol(T_{\mu^j}(\fS_{\epsilon,\mu^j}^k(f))\cap B_{1/2}(\bo)) \leq C \cdot (\mu^j)^{n - k -\epsilon}, \quad  0\leq k \leq n - 1,
\end{align}
where $T_r(A)$ is the $r$-tubular neighborhood of $A\subset \dR^n$. In particular, \begin{align}
\dim_{\Min}(\uC(f)\cap B_{1/2}(\bo)) \leq n - 2\quad \text{and} \quad \dim_{\Min}(\fC(f)\cap B_{1/2}(\bo)) \leq n - 1.	
\end{align}
\end{theorem}

\begin{remark}
In fact, one can obtain the same estimates for the critical set of $f$. 	
\end{remark}

The proof of the theorem follows along the the same lines as the covering arguments in Section \ref{ss:volume-estimate}, so we omit it.

We end this section by introducing the Hausdorff measure estimate for the critical set $\cC(f)$ for a $\gamma$-harmonic function $f$.

\begin{theorem}   \label{t:Hausdorff-measure-estimate-Euclidean-boundary}
 Given $n\geq 2$ and $\gamma \in (0,1)$, let $f\in C^{\infty}(\dR^n)$ be $\gamma$-harmonic on $B_1(\bo)$. 
 For any $\Lambda > 0$, there exists $C = C(\Lambda, n, \gamma) > 0$ such that if $\cN_f(\bo, 1) \leq \Lambda$, then 
$\mathcal{H}^{n - 2}(\uC(f) \cap B_{1/2}(\bo)) \leq C$ and
$\mathcal{H}^{n - 1}(\fC(f) \cap B_{1/2}(\bo)) \leq C$.
 \end{theorem}

\section{Singular set in general Poincar\'e-Einstein manifolds}

\label{s:results-on-PE}

This section is devoted to the proof of our main results in the general Poincar\'e-Einstein case. Given $n\geq 2$, denote $\fm \equiv n + 1$ and let $(X^{\fm}, g_+)$ be a complete Poincar\'e-Einstein manifold with conformal infinity $(M^n, h)$. We choose the Fefferman-Graham compactification $\bg = \vr^2 g_+ = e^{2w} g_+$ of $(X^{\fm}, g_+)$ as in Lemma \ref{l:E} such that $-\Delta_{g_+} w = n$ and $\bg|_{M^n} = h$. Throughout this section, we make some fundamental assumptions on the regularity of the Poincar\'e-Einstein metric $g_+$, that we precisely state below. 

\begin{enumerate}[{\bf {A}ssumpt{i}on (1).}]
\item $(X^{\fm}, g_+)$ admits a $C^{\infty}$ conformal compactification up to the boundary $M^n$. 
\item $(M^n,[h])$ has nonnegative Yamabe invariant, i.e., $\mathcal{Y}(M^n,[h]) \geq 0$. This particularly implies  $\lambda_1(-\Delta_{g_+}) = \frac{n^2}{4}$ so that one can define fractional GJMS operators $P_{2\gamma}$ on $ M^n$ for all $\gamma\in(0,1)$; see Remark \ref{r:bottom-of-spectrum}.
	\item $(M^n, h)$ is obstruction flat when $n$ is even.
\end{enumerate}	   
Under these assumptions, by Propositions \ref{p:general-regularity-FG} and \ref{p:regularity-obstruction-free}, the Fefferman-Graham compactified metric $\bg$ is $C^{\infty}$ up to the boundary $M^n$. 
To avoid unnecessary technical complications, there is no harm to make the following assumption on the compactified space $(\overline{X^{\fm}}, \bg)$.
\begin{enumerate}[{\bf {A}ssumpt{i}on}]
\item {\bf(4).} There exists $\iota_0>0$ such that $\Inj_{\bg}(x_+)\geq \iota_0 > 0$ for any $x_+\in \overline{X^{\fm}}$.

\item {\bf(5).} {\it The boundary injectivity radius} $\Inj_{\p}(M^n) \geq 2$, namely the {\it normal exponential map} is a diffeomorphism within the tubular neighborhood  \begin{align}T_2(M^n)\equiv \{p\in \overline{X^{\fm}}: d_{\bg}(p, M^n) \leq 2\}.\end{align} In particular, $T_2(M^n)\equiv \{p\in \overline{X^{\fm}}: d_{\bg}(p, M^n) \leq 2\}$ of $M^n$ is always diffeomorphic to $M^n \times [0,2]$. 
\end{enumerate}
Indeed, one can always achieve this by a finite rescaling. 

For our purpose, we need to double the compactified manifold $(\overline{X^{\fm}}, \bg)$ along the totally geodesic boundary $(M^n, h)$, which  gives a closed manifold $\fX^{\fm} \equiv \overline{X^{\fm}}\bigcup\limits_{M^n}\overline{X^{\fm}}$ equipped with a $C^{\fm-2,1}$-Riemannian metric (still denoted as $\bg$): $\bg$ fails to be smooth only when crossing $M^n$.
Now let us take a smooth function $f\in C^{\infty}(M^n)$ that satisfies 
$P_{2\gamma}(f) = 0$ for some $\gamma \in (0,1)$. Let $U\in H^{1,\fa}(X^{\fm})$ be the {\it even extension} of $f$ that solves 
\begin{align}\label{e:CS-extension-on-PE}
\begin{cases}
		- \Div_{\bg}(\vr^{\fa}\nabla_{\bg} U) + \vr^{\fa} \cJ_{\bg} U  = 0 &  \text{in}\ \fX^{\fm},
	\\
	U = f, & \text{on}\  M^n,
	\\ 
	P_{2\gamma} f = 0 & \text{on} \ B_1(p)\subset M^n,
	\end{cases}
\end{align}
where $\fa \equiv 1 - 2 \gamma$ and $\cJ_{\bg}\equiv C_{n,\gamma}R_{\bg}$. Lemma \ref{eq-C1a-U} implies that the solution $U$ yields $U\in C^{\fm-2,\alpha}(\overline{X^{\fm}})$ for some  $\alpha\in(0,1)$.

\subsection{Frequency and almost monotonicity}
\label{ss:almost-monotonicity}

This subsection is to develop an almost monotonicity formula for the generalized Almgren's frequency associated to the solution $U$ in \eqref{e:CS-extension-on-PE}. To begin with, for any $x_+\in M^n$ and $r\in (0,1)$, let us define 
 \begin{align}
H_U(x_+, r) & \equiv \int_{\partial \cB_r(x_+)} \vr^{\fa} U^2 d\sigma_{\bg},\label{eq-H}\\
I_U(x_+, r) & \equiv \int_{\cB_r(x_+)} \vr^{\fa}  (|\nabla_g U|^2 + \mathcal{J}_{\bg}   U^2) \dvol_{\bg}.\label{eq-I}
\end{align}
Define the {\it generalized Almgren's frequency} of $U$ by
\begin{align*}
\mathcal{N}_U(x_+, r) = \frac{rI_U(x_+, r)}{H_U(x_+, r)},
\end{align*}
if $H_U(x_+,r) \neq 0$. For the convenience of our later computations, we also define the weighted Dirichlet energy
\begin{align}
D_U(x_+, r) & \equiv \int_{\cB_r(x_+)} \vr^{\fa}   |\nabla_g U|^2 \dvol_{\bg}.\end{align}

For our later computations, we introduce some notations on the geodesic polar coordinates.  
Fix a point $x_+\in M^n$ and let $r_0 \in(0, \Inj_h(x_+)/10)$ be a small number.
Consider the geodesic polar coordinate system $(r,\Theta)$ on $\cB_{r_0}(x_+)$ defined by the exponential map at $x_+$,
\begin{align}
	\exp_{x_+}: B_{r_0}(0^{n + 1}) \to \cB_{r_0}(x_+),
\end{align}
where $B_{r_0}(0^{n + 1})\subset \dR^{n+1}$. As $M^n\subset (\overline{X^{\fm}}, \bar{g})$ is totally geodesic, one can also choose normal coordinates $\{x^1, \cdots, x^{n}, x^{n+1}\}$ at $x_+$ such that
\begin{align}\label{eq-norm-n+1}
\{x^{n+1}=0\}\subseteq M,
\quad r^2 = (x^1)^2+\cdots + (x^{n+1})^2.
\end{align}
We can define $\frac{\partial}{\partial r}$ using the polar coordinates. We have
\begin{align}
\frac{\partial}{\partial r} =\frac{x^i}{r}\frac{\partial }{\partial x^i} = \nabla_{\bar g}r.
\end{align}
Under polar coordinates,
\begin{align*}
\bar g = dr^2 + r^2 b_{ij}(r,\Theta) d\theta^i  d\theta^j,
\end{align*}
where $b_{ij} =\delta_{ij}$ at $x_+.$

\begin{lemma}\label{lem-rho-a}
Let $\vr$ be the function defined in \eqref{e:vr-expansion}, and let $r$ be the geodesic radial coordinate centered at $x_+ \in M^n$. Then there exists a small positive constant $r_0 = r_0(\fm, \iota_0) > 0$ such that the following holds, for $r\in (0, r_0)$,
\begin{align}\label{eq-rho-a-r}
\frac{\partial \vr^{\fa}}{\partial r}  = \frac{\fa \vr^{\fa}}{r}(1 + O(r)).
\end{align}
  In particular, for $r\in (0, r_0)$,
\begin{align}
\frac{\partial \vr^{\fa}}{\partial r} \geq 0.
\end{align} 
\end{lemma}

\begin{proof}
First, we need to estimate $\frac{\partial y}{\partial r}$. For $\alpha = 1,\cdots, n$, we have $\frac{\partial y}{\partial x^\alpha} =0$ on $\{y=0\}$, which implies
\begin{align}
\frac{\partial y}{\partial x^\alpha} = O(y).
\end{align}
Then, we have
\begin{align}
\frac{\partial y}{\partial r} = \frac{x^{n+1}}{r} \frac{\partial y}{\partial x^{n+1}}  + \frac{x^{\alpha}}{r} \frac{\partial y}{\partial x^{\alpha}} = \frac{x^{n+1}}{r} \frac{\partial y}{\partial x^{n+1}} +O(y).
\end{align}

Notice that $dy = dx^{n+1}$ at $x_+$. By Proposition \ref{p:general-regularity-FG}, $\exp_{x_+}$ is a local diffeomorphism, which gives \begin{align*}
dy = (1+O(r))dx^{n+1} + \sum_{\alpha =1}^n O_\alpha(y) dx^\alpha,
\end{align*}
which implies
\begin{align}\label{eq-dydxn+1}
\frac{\p y}{\p x^{n+1}} = dy \left(\frac{\partial}{\partial x^{n+1}}\right) = 1 + O(r).
\end{align}
So we obtain
\begin{align}\label{eq-DyDr}
\frac{\partial y}{\partial r}  = \frac{x^{n+1}}{r} +O(x^{n+1}) = \frac{y}{r} +O(y).
\end{align}
Here we used the fact that $\frac{x^{n+1}}{y}=1+O(r)$ near $x_+$.

Now we compute $\frac{\partial \vr^{\fa}}{\partial r}$. For $\vr =\vr$, applying \eqref{e:vr-expansion}, \eqref{eq-dydxn+1} and \eqref{eq-DyDr}, we have
\begin{align}
\vr &= y + O(y^{3}),\label{eq-rho-y-relation}\\
\vr_{, n+1} &= 1 + O(r),\label{eq-rho-y-1}\\
\vr_{,\alpha} &= O(y),\label{eq-rho-y-2}
\end{align}
where $\vr_{,\alpha}$ denotes the derivative of $\vr$ with respect to $x^\alpha$ for $\alpha = 1,\cdots, n$. Thus, we find
\begin{align}\label{eq-rho-a-r-1}
\frac{\partial \vr^{\fa}}{\partial r} = \fa\vr^{\fa-1} \left(\vr_{, n+1} (x^{n+1})_{,r} + \vr_{, \alpha} (x^{\alpha})_{,r}\right) = \fa\vr^{\fa-1} \left(\frac{y}{r} + O(y)\right) = \frac{\fa \vr^{\fa}}{r}(1 + O(r)).
\end{align}
Then, combining \eqref{eq-rho-y-relation} and \eqref{eq-rho-a-r-1}, we conclude the lemma.
\end{proof}

To establish the almost monotonicity formula for the generalized frequency $\mathcal{N}_U(x_+, r)$, one needs the following radial dilation in a given geodesic polar coordinate system $(\cO_{x_+}, (s,\Theta))$ around a point $x_+\in M^n$, where $s \in (0, r_0)$ for some constant $r_0 = r_0(\iota_0, \fm) > 0$.

 For any sufficiently small $t \in (0, r_0)$, we define the quantitative tangent map (radial blow-up function) $U_{x_+, t} = \mathcal{T}_{x_+, t}(U)  : \cB_1(x_+, \bg_t) \to \dR$ as follows
 \begin{align} U_{x_+, t}(x) = U_{x_+, t}(s,\Theta) \equiv \frac{\displaystyle{ U(ts, \Theta) } }{\displaystyle{\left(\frac{1}{t^{\fa}}\fint_{\p \cB_1(x_+, \bar{g}_t)}\vr^{\fa} U^2 \dst\right)^{\frac{1}{2}}}},\label{e:blow-up-function-U}
\end{align} 
where $\bg_t(s,\Theta) = \bg(ts, \Theta)$ and $\dst \equiv d\sigma_{\bg_t}$. 
The following lemma will be used in proving the almost monotonicity formula, and the proof of the lemma 
follows from straightforward computations. 
  \begin{lemma}\label{l:frequency-rescaling}
  For any given $t\in(0,1)$, let $U_{x_+, t}$ be the quantitative tangent map of $U$ at $x_+$ defined by \eqref{e:blow-up-function-U}. We define 
  \begin{align}
  &  D(U_{x_+, t}, x_+, r, \bg_t) \equiv \int_{\cB_r^t(x_+)} \vr^{\fa}  |\nabla_{\bg_t} U_{x_+, t}|^2  \dvol_{\bg_t},
\\
&  I(U_{x_+, t}, x_+, r, \bg_t) \equiv \int_{\cB_r^t(x_+)} \vr^{\fa}\left( |\nabla_{\bg_t} U_{x_+, t}|^2 +  t^2 \cdot \cJ_{\bg} U_{x_+, t}^2 \right) \dvol_{\bg_t},
  	\\
&  	H(U_{x_+, t}, x_+,  r, \bg_t) \equiv \int_{\p\cB_r^t(x_+)} \vr^{\fa} U_{x_+, t}^2 \dst,
  	\\
&  	\mathcal{N}(U_{x_+, t}, x_+,  r, \bg_t) \equiv \frac{r I(U_{x_+, t}, x_+, r, \bg_t)}{H(U_{x_+, t}, x_+,  r, \bg_t)},
  \end{align} 
 where $\cB_r^t(x_+) \equiv \cB_r(x_+, \bg_t)$. 
  Then the following holds:
  \begin{align}
  \begin{split}
  & D(U_{x_+, t}, x_+, t^{-1} r, \bg_t) = t^{2 - \fm} \omega^{-2} D(U, x_+, r, \bg) = t^{2 - \fm} \omega^{-2} D_U(x_+, r), 
  \\
& I(U_{x_+, t}, x_+, t^{-1} r, \bg_t) = t^{2 - \fm} \omega^{-2} I(U, x_+, r, \bg) = t^{2 - \fm} \omega^{-2} I_U(x_+, r), 
  \\
&  H(U_{x_+, t}, x_+, t^{-1} r, \bg_t) = t^{1 - \fm}\omega^{-2} H(U, x_+, r, \bg) = t^{1 - \fm} \omega^{-2} H_U(x_+, r),
  \\
 & 	  	\mathcal{N}(U_{x_+, t}, x_+, t^{-1}r, \bg_t) = \mathcal{N}(U, x_+, r, \bg) = \mathcal{N}_U(x_+, r),
  	  	  \end{split}
  \end{align}	
  where $\omega \equiv \left(\frac{1}{t^{\fa}}\fint_{\p \cB_1^t(x_+)}\vr^{\fa} U^2 \dst\right)^{\frac{1}{2}}$ and $\bg_t(s,\Theta) = \bg(t\cdot s, \Theta)$.
  \end{lemma}

\begin{remark}
In the above notation, actually it holds that $\cJ_{\bg_t} = t^2 \cJ_{\bg}$. 
\end{remark}

The main result in this subsection is the following almost monotonicity theorem for generalized Almgren's frequency.

\begin{theorem}[Almost monotonicity] \label{t:almost-monotonicity} Let $U$
be an even solution of the boundary value problem \eqref{e:CS-extension-on-PE}.
There exist   $C=C(\fm, \fa, \bg) >0$, $r_0 = r_0(\fm, \fa, \bg) >0$, $\epsilon_0 = \epsilon_0(\fm, \fa, \bg) > 0$, and $\tau_0 = \tau_0 (\fm, \fa, \bg) > 0$ such that either one of the following holds for any $x_+\in M^n$: 
\begin{enumerate}
	\item $|U_{x_+, t}| \geq \tau_0 > 0$ on $\cB_{1/2}(x_+, \bg_t)$ for some $t\in(0, r_0]$;
	\item $\cN_U(x_+,t) > \epsilon_0$ for any $t\in(0,r_0]$, and $
e^{C t}\mathcal{N}_U(x_+, t)$ is non-decreasing in $t \in (0, r_0]$.
\end{enumerate}
 \end{theorem}
 
 \begin{remark}\label{r:alternative}
 By Proposition \ref{p:eps-non-vanishing}, one can find that item (1) happens when $\cN_U(x_+,t) \leq \epsilon_0$ for any $t\in(0,r_0]$. In this case, $x_+$ is a good point and stays away from any quantitative singular stratum. In the proof, we will only consider the case of item (2). 	
 \end{remark}

To prove this theorem, we need a series of preliminary results. 

\begin{lemma}\label{l:L2-estimate-large-scale} Let $r\in (0,2)$ and $x_+\in M^n$ such that $\cB_r(x_+) \subset \fX^{\fm}$ has a smooth boundary. 
There exists a constant $C>0$ such that
 for all $u\in H^{1,\fa}(\cB_r(x_+))$,
    \begin{align}\label{eq-ET}
        \int_{\cB_r(x_+)} \vr^{\fa} u^2 \dvol_{\bar g}  \leq C\left(r^2\int_{\cB_r(x_+)} \vr^\fa |\nabla_{\bar g} u|^2 \dvol_{\bar g} +\  r \int_{\partial \cB_r(x_+)} \vr^\fa U^2 \dsg\right).
    \end{align}
\end{lemma}

This inequality is standard in the literature; see \cite[Chapter 10]{Hebey} for the unweighted version on a compact manifold with boundary.
 
\begin{proof}The inequality is scale invariant so that one can assume $r = 1$. 
    We first prove that there exists some constant $\Lambda_0 >0$ such that for all $U \in H^{1,\fa}_0 (\cB_1(x_+))$, 
        \begin{align}
        \label{eq-Poincare}\int_{\cB_1(x_+)} \vr^{\fa} u^2 \dvol_{\bar g}  \leq \Lambda_0  \int_{\cB_1(x_+)} \vr^\fa |\nabla_{\bar g} u|^2 \dvol_{\bar g}  .
    \end{align}
If it does not hold, then for any $i\in \mathbb{N}$, there is  $u_i \in H_0^{1,\fa}(\cB_1(x_+))$ such that
\begin{align}
    \int_{\cB_1(x_+)}\vr^{\fa} |\nabla_{\bar g} u_i|^2 \dvol_{\bar g}  < i^{-1}, \,\,  \int_{\cB_1(x_+)}\vr^{\fa}  u_i^2 \dvol_{\bar g} =1.
\end{align}
Then $u_i \rightarrow u_{\infty}$ weakly in $H^{1, \fa}_0(\cB_1(x_+))$. By the compact embedding $H^{1, \fa}_0(\cB_1(x_+)) \hookrightarrow L^{2,\fa}(\cB_1(x_+))$, we have 
\begin{align*}
    \int_{\cB_1(x_+)}\vr^{\fa} |\nabla_{\bar g} u_{\infty}|^2 \dvol_{\bar g}  =0, \,\,  \int_{\cB_1(x_+)}\vr^{\fa}  u_{\infty}^2 \dvol_{\bar g} =1.
\end{align*}
Then $u_{\infty}$ must be a nonzero constant, which contradicts $u_{\infty} \in H^{1,\fa}_0(\cB_1(x_+))$.

Next, we prove \eqref{eq-ET}. 
Suppose it is not true,  i.e., there exists contradiction sequences $A_j \to \infty$ and $u_j\in H^{1,\fa}(\cB_1(x_+))$ such that
\begin{align}
    \int_{\cB_1(x_+)} \vr^{\fa} u_j^2 \dvol_{\bar g}  > A_j \left(\int_{\cB_1(x_+)} \vr^\fa |\nabla_{\bar g} u_j|^2 \dvol_{\bar g} +  \int_{\partial \cB_1(x_+)} \vr^\fa u_j^2 \dsg \right).
\end{align}
 Without loss of generality, we assume that
\begin{align}\label{eq-Ua}
    \int_{\partial \cB_1(x_+)} \vr^\fa u_j^2 \dsg + \int_{\cB_1(x_+)} \vr^{\fa} u_j^2 \dvol_{\bar g}=1.
\end{align}
 Then
 \begin{align}\label{eq-DUa}
     \int_{\cB_1(x_+)} \vr^\fa |\nabla_{\bar g} u_j|^2 \dvol_{\bar g}
    < A_j^{-1},
 \end{align}
which implies that $\|u_j\|_{H^{1,\fa}(\cB_1(x_+))}$ is uniformly bounded. Now letting $j\rightarrow \infty$, $u_j \rightarrow u_{\infty}$ weakly in $H^{1,\fa}(\cB_1(x_+))$ and strongly in $L^{2, \fa}(\cB_1(x_+))$.  By \eqref{eq-Ua} and \eqref{eq-DUa},
 \begin{align}\label{eq-Ubar}
\int_{\cB_1(x_+)}\vr^{\fa} u_{\infty}^2\dvol_{\bar g}=1, \,\, \int_{\cB_1(x_+)} \vr^\fa |\nabla_{\bar g}  u_{\infty}|^2 \dvol_{\bar g} = 0.
 \end{align}
 In addition, $\int_{\partial \cB_1(x_+)} \vr^\fa u_j^2 \dvol_{\bar g} \rightarrow 0$, which implies that $u_{\infty} \in H^{1,\fa}_0(\cB_1(x_+))$.
Then \eqref{eq-Ubar} contradicts 
 \eqref{eq-Poincare} which completes the proof of the proposition.
\end{proof}

The above inequality has a quick corollary on the non-vanishing of $H_U(x_+, r)$ which implies that the generalized frequency $\mathcal{N}_U(x_+, r)$ is well-defined for a non-trivial solution of \eqref{e:CS-extension-on-PE}.

\begin{corollary}\label{c:H-not-vanishing} Let $U \in H^{1, \fa}(\cB_1(x_+))$ be a non-trivial solution of \eqref{e:CS-extension-on-PE}. There exists $r_0 = r_0 (\fm, \fa, \bg) > 0$ such that $H_U(x_+, r) \neq 0$
 for any $r\in (0, r_0)$.
\end{corollary}
\begin{proof}
Assume that $H_U(x_+, r) =0$ for some $r \in (0, r_0)$, so $U \equiv 0$ on $\p\cB_r(x_+)$. By the divergence theorem, $I_U(r,x_+) =0$. So it follows that \begin{align} 
D_U(x_+,r) \leq C_0(\fm, \fa, \bg)  \int_{\cB_r(x_+)}\vr^{\fa} U^2 \dvol_{g}\leq   C_0'(\fm, \fa, \bg) \left(r^2 D_U(x_+,r)+  r H_U(x_+,r)\right), 
\end{align}
where we used Lemma \ref{l:L2-estimate-large-scale} for the last inequality. If $r$ is chosen such that $r^2 < \frac{1}{2C_0'(\fm ,\fa, \bg)}$, then $D_U(x_+, r) \equiv 0$.
Therefore, $U \equiv 0$ on $\cB_r(x_+)$. 
Applying standard unique continuation, we obtain a contradiction to the nontriviality assumption on $U$.
\end{proof}

\begin{lemma}\label{l:doubling-fixed-scale} For any $\Lambda > 0$, there exist $r_0 = r_0 (\Lambda, \fm, \fa, \bg) > 0$ and $D_0 = D_0(\Lambda, \fm, \fa, \bg) > 0$ such that if a solution $U$ of \eqref{e:CS-extension-on-PE} satisfies $\cN_U(x,2r)\leq \Lambda$ for some $x\in M^n$ and $r\in (0,r_0/10)$, then the following holds:
\begin{enumerate}
	\item $\cN_U(x, s) \leq 2\Lambda$ for any $s \in (r/2,2r)$; 
	\item $H_U(x, 2r) \leq D_0 \cdot H_U(x, r)$.
\end{enumerate} 

\end{lemma}

\begin{proof}
We will prove item (1) by contradiction. Suppose there does not exist such a constant $r_0 > 0$. That is, one can find a sequence of solutions $U_j$ of \eqref{e:CS-extension-on-PE} and sequences of numbers $r_j \to 0$ and $s_j \in (r_j/2, 2r_j)$ such that 
\begin{align}\cN_U(x, 2r_j) \leq \Lambda \quad \text{and} \quad \cN_U(x, s_j) >  2\Lambda.\end{align}
 Let us take rescaled sequences $\tU_j(s,\Theta)\equiv U_j(r_j s, \Theta)$ and $g_j(s,\Theta) \equiv g(r_js,\Theta)$
for any sufficiently large $j$. Then by scale invariance,
\begin{align}
	\cN(\tU_j,x, 2, g_j) \leq \Lambda \quad \text{and} \quad \cN(\tU_j,x, s_j/r_j, g_j) >  2\Lambda.
\end{align}
We also assume that $U_j$ is normalized by $H(\tU_j,x,2, g_j) = 1$.
Letting $j\to \infty$, passing to a subsequence, we have  $g_j$ converges to the Euclidean metric $g_0$ on $\dR^{\fm}$, $\lim\limits_{j\to\infty} s_j/r_j = t_0 \in [1/2, 2]$ and 
\begin{align}
	\tU_j \xrightarrow{H^{1,\fa}(\cB_2(x))} \tU_{\infty}\in H^{1,\fa}(\cB_2(0^{\fm})) 
\end{align}
such that 
	\begin{align}\cN(\tU_{\infty}, 0^{\fm}, 2, g_0) \leq \Lambda \quad \text{and} \quad \cN(\tU_{\infty}, 0^{\fm},t_0, g_0) >  2\Lambda.
\end{align}
But this contradicts to Lemma \ref{l:Euclidean-monotonicity}, which completes the proof of item (1).

Now we are ready to prove the doubling property in item (2).
In geodesic polar coordinates, we rewrite $H_U(x, s)$ as
\begin{align*}
H_U(x, s) = s^{\fm-1} \int_{\partial \cB_1}\vr^{\fa}(s,\Theta)  U^2 (s, \Theta) \sqrt{b(s, \Theta)} d\Theta.
\end{align*}
For fixed center $x$, let us denote $H_U(s) \equiv H_U(s,x)$. 
Applying \eqref{eq-rho-a-r}, we obtain
\begin{align}
H_U'(s) = \left(\frac{\fm - 1 + \fa }{s}+O(1)\right)H_U(s) +\int_{\partial \cB_s} \vr^{\fa}(\partial_n  \log\sqrt{b})U^2 \dsg + 2\int_{\partial \cB_s} \vr^{\fa} U (\p_n U) \dsg,
\end{align}
which implies 
\begin{align}\label{eq-H'}
\frac{H_U'(s)}{H_U(s)} = \frac{\fm - 1 + \fa }{s}+O(1) + 2\ \cdot \frac{\displaystyle{\int_{\partial \cB_s} \vr^{\fa} U (\p_n U) \dsg}}{\displaystyle{\int_{\partial \cB_r} \vr^{\fa}  U^2 \dsg }} = \frac{\fm - 1 + \fa }{s}+O(1) + \frac{2\cN_U(s)}{s}.
\end{align}
Here $O(1)$ denotes a function of $s$ and $\Theta$, which is bounded in absolute value by a constant $C$.
Therefore, 
\begin{align}
	 \frac{d}{ds} \left( \log\frac{H_U(s)}{s^{\fm - 1 + \fa}}\right) = O(1) + \frac{2\cN_U(s)}{s}.\label{e:derivative-normalized-log-H}
\end{align}
By item (1), $\cN_U(s) \leq 2\Lambda$ for any $s\in [r,2r]$. Then integrating \eqref{e:derivative-normalized-log-H}, we have
\begin{align}
	H_U(2r) \leq D_0\cdot H_U(r)
\end{align}
for some constant $D_0 = D_0(\Lambda, \fm, \fa, \bg)$, which completes the proof of item (2).
 \end{proof}

Next, we prove an $\epsilon$-regularity result for small frequency.

\begin{proposition}\label{p:eps-non-vanishing} Let $U$ be a non-trivial even solution of \eqref{e:CS-extension-on-PE}.
For every $\eta > 0$, there exist small constants $r_0 = r_0(\eta, \fm, \fa, \bar{g}) >0$, $\epsilon_0  = \epsilon_0(\eta, \fm, \fa, \bar{g})> 0$, and $\tau_0 = \tau_0(\eta, \fm, \fa, \bar{g}) >0$ such that if \begin{align}\mathcal{N}_U(x_+, r) \leq \epsilon_0 \quad  \text{or} \quad \frac{r D_U(x_+, r)}{H_U(x_+, r)} \leq \epsilon_0\label{e:small-frequency}\end{align} for some $x_+\in \cB_{1/2}(p_+)$ and $r \in (0, r_0]$, then for any $0 \leq k \leq \fm - 1$,
\begin{align}
	|U_{x_+, r} - \tau_0|_{C^k(\cB_{1/2}(x_+, \bg_r))} \leq \eta.
\end{align} 
In particular, $U$ is nowhere vanishing in $\cB_{r/2}(x_+)$.

\end{proposition}

\begin{remark}\label{r:definite-lower-bound-of-frequency}
This proposition tells us that a point at which the frequency is sufficiently small must be away from any quantitative singular stratum. In other words, such points must be ``good points" which are not in our considerations. Therefore, in the following analysis, there is no harm to assume the definite lower bound $\cN_U(x_+, r) > \epsilon_0$ for any of our interested point $x_+$ in a quantitative singular stratum.
\end{remark}

\begin{proof}[Proof of Proposition \ref{p:eps-non-vanishing}] We will only give a proof in the first case of \eqref{e:small-frequency} since the proof in the other case is the same. We will prove the proposition by contradiction. Suppose no such constants $r_0$, $\epsilon_0$, and $\tau_0 > 0$ exist. That is, there exists a sequence of solutions $U_j$ of \eqref{e:CS-extension-on-PE} that serves as a contradiction sequence such that 
\begin{align}
	\cN_{U_j}(x_j, r_j)\leq \epsilon_j
\end{align} 
for $r_j \to 0$ and $\epsilon_j \to 0$, but 
$\sU_j\equiv U_{x_j, r_j}$ is not $C^k$-close to any constant for some $0\leq k\leq \fm - 2$.

By scale invariance, we have 
\begin{align}
	\cN_j(x_j, 1) \equiv \cN_{\sU_j}(x_j, 1, \bg_j) \leq \epsilon_j,
\end{align}
where $\bg_j (s,\Theta) \equiv \bg(r_j s, \Theta)$ in a local geodesic polar coordinate system $(s,\Theta)$ centered at $x_j$.
Since the quantitative tangent map $\sU_j$ satisfies the normalization 
\begin{align}\label{eq-H=1}
    H(\sU_j, x_j, 1, \bg_j) = 1 \ \text{and}\ D(\sU_j, x_j, 1, \bg_j) + r_j^2 \int_{\cB_1(x_j, \bg_j)}\cJ_{\bg}  \sU_j^2\dvol_{\bg_j}\leq \epsilon_j,
\end{align}
applying Lemma \ref{l:L2-estimate-large-scale}, we have
\begin{align}
	D(\sU_j, x_j, 1, \bg_j) \leq \epsilon_j + Cr_j^2(D(\sU_j, x_j, 1, \bg_j) + H(\sU_j, x_j, 1, \bg_j)).
\end{align} 
Therefore, 
$D(\sU_j, x_j, 1, \bg_j) \leq C(r_j^2 + \epsilon_j)$.
 Applying Lemma \ref{l:L2-estimate-large-scale}, we have
\begin{align}
	\int_{\cB_1(x_j, \bg_j)}\vr^{\fa}\sU_j^2\dsj \leq C(\fm, \fa, \bg).
\end{align}

Letting $j\to \infty$, it holds that $\bg_j$ converges to the Euclidean metric $g_0$, $x_j$ converges to some point in the Euclidean space, and $\sU_j$ converges to $\sU_{\infty}$ that satisfies 
\begin{align}
	\int_{\cB_1(0^{\fm})}|y|^{\fa} |\nabla_{g_0} \sU_{\infty}|^2 \dvol_{g_0} = 0.
\end{align}
By the doubling property in Lemma \ref{l:doubling-fixed-scale}, there exists $D_0 = D_0(\fm, \fa, \bg) > 0$ such that 
\begin{align}
\int_{\p \cB_{1/2}(x_+, \bg_j)}\vr^\fa  \sU_j ^2 \dsj \geq D_0^{-1}.
\end{align}
Then one can conclude that $\sU_{\infty} \equiv \tau_0$ on $\cB_1(0^{\fm})\subset \dR^{\fm}$ for some uniform constant $\tau_0 \equiv \tau_0(\fm, \fa, \bg) > 0$.   
 So the desired contradiction arises by using elliptic regularity, which completes the proof. \end{proof}

\begin{corollary}\label{c:approximate-Dirichlet-energy}
There exist $C=C(\fm, \fa, \bar{g}) >0$ and $r_0 = r_0(\fm, \fa, \bar{g}) >0$ such that if $U$ is not $(\fm, 10^{-6}, t, r_0)$-symmetric at $x_+$ for some $t\in (0,r_0)$, then 
\begin{align}
(1 - Ct^2)  \int_{\cB_1^t(x_+)}\vr^{\fa} |\nabla_{\bg_t} \sU|^2\dvol_{\bg_t}  	\leq  I(\sU,x_+, 1, \bg_t)   \leq (1 + Ct^2) \int_{\cB_1^t(x_+)}\vr^{\fa} |\nabla_{\bg_t} \sU|^2\dvol_{\bg_t},
\end{align}
where  $\cB_1^t(x_+)\equiv \cB_1(x_+, \bg_t)$ and $\sU \equiv U_{t,x+}$.\end{corollary}

\begin{proof}
It suffices to estimate the $0^{\text{th}}$-order term in 
\begin{align}I(\sU,x_+, 1, \bg_t)\equiv  \int_{\cB_1^t(x_+)}\vr^{\fa}(|\nabla_{\bg_t}\sU|^2 + t^2\cJ_{\bg}\sU^2)\dvol_{\bg_t}.\end{align}
 By Lemma \ref{l:L2-estimate-large-scale}, 
\begin{align}
	\left|\int_{\cB_1^t(x_+)} \vr^{\fa}\cJ_{\bg}\sU^2\dvol_{\bg_t}\right| \leq C(\fm, \fa, \bg) \left(\int_{\cB_1^t(x_+)}\vr^{\fa}|\nabla_{\bg_t}\sU|^2\dvol_{\bg_t} + \int_{\p\cB_1^t(x_+)}\vr^{\fa} \sU^2\dst\right).
\end{align}
Since $U$ is not $(\fm, 10^{-6}, t, r_0)$-symmetric at $x_+$, applying Proposition \ref{p:eps-non-vanishing}, we have $\cN(\sU, x_+, 1, \bg_t) > \epsilon_0$ with $\epsilon_0(\fm ,\fa, \bg) > 0$ the uniform constant in Proposition \ref{p:eps-non-vanishing}, which gives 
\begin{align}
	\left|\int_{\cB_1^t(x_+)} \vr^{\fa}\cJ_{\bg}\sU^2\dvol_{\bg_t}\right| \leq C(\fm, \fa, \bg) (1 + \epsilon_0^{-1}) \int_{\cB_1^t(x_+)}\vr^{\fa}|\nabla_{\bg_t}\sU|^2\dvol_{\bg_t}, 
\end{align}
which completes the proof.
\end{proof}

Now we are ready to complete the proof of Theorem \ref{t:almost-monotonicity}. 
\begin{proof}
[Proof of Theorem \ref{t:almost-monotonicity}] 
Recalling Remark \ref{r:alternative} and Proposition \ref{p:eps-non-vanishing}, we will only consider the case of item (2) of the Theorem \ref{t:almost-monotonicity}.

To prove the theorem, we will first carry out some reduction by rescaling. Theorem \ref{t:almost-monotonicity} is equivalent to \begin{align}\frac{\mathcal{N}_U'(x_+, t)}{\mathcal{N}_U(x_+, t)}	 \geq - C, \ \quad \forall t\in(0,r_0), \label{e:derivative-of-log-N-bounded-below}\end{align}where $\mathcal{N}_U'(x_+, t)=\frac{d}{dt}\mathcal{N}_U(x_+, t)$.  For fixed $t > 0$ and $x_+\in M^n$, let $U(x_+, r)$ be the blow-up function of $U$. Let us also denote   \begin{align}\cN_t(r) \equiv \mathcal{N}(U(x_+, r), x_+, r, \bg_t).\end{align} By Lemma \ref{l:frequency-rescaling}, the derivative bound \eqref{e:derivative-of-log-N-bounded-below} is equivalent to \begin{align}\frac{\cN_t'(1)}{\cN_t(1)} \geq - C t,	\end{align}where $\cN_t'(1) = \left.\frac{d}{dr}\right|_{r = 1}\mathcal{N}(U(x_+, r), x_+, r, \bg_t)$. Throughout the proof, we will use the simplified notation $\sU \equiv  U(x_+, r)$.

The first step is to estimate the term $H_t(1)$.
\begin{align}
	H_t(r) = \int_{\p\cB_r^t(x_+)}\vr^{\fa}\sU^2 d\sigma_t = r^{\fm - 1}\int_{\p \wt{\cB}_1^t(x_+)}\vr^{\fa}\sU^2 d\sigma_{t,r},
\end{align}
where $\cB_r^t(x_+)\equiv \cB_r(x_+, \bg_t)$, $\wt{\cB}_1^t(x_+) \equiv \cB_1(x_+, r^{-2}\bg_t)$, and $d\sigma_{t,r}\equiv d\sigma_{r^{-2}\bg_t}$.
So it follows that 
\begin{align}
\begin{split}
		H_t'(1) = & \ (\fm  - 1) H_t(1) + 2 \int_{\p \wt{\cB}_1^t(x_+)}\vr^{\fa}\sU \cdot  (\p_r\sU) d\tilde{\sigma}_{t, r} + \int_{\p \wt{\cB}_1^t(x_+)}  (\p_r\vr^{\fa}) \cdot \sU^2  d\tilde{\sigma}_{t, r} \\
	& \ + \int_{\p \wt{\cB}_1^t(x_+)}\vr^{\fa}\sU^2 \left(\p_r\log\sqrt{\wt{G}_{t,r}}\right) d\tilde{\sigma}_{t, r}.
	\end{split}
\end{align}
First we estimate the third and the last term in $H_t'(1)$.
For the third term, applying Lemma \ref{lem-rho-a}, we have that
\begin{align}
\int_{\p \wt{\cB}_1^t(x_+)} \frac{\p\vr^{\fa}}{\p r}\sU^2  d\tilde{\sigma}_{t, r} = (\fa + O(t))\cdot H_t(1).  	
\end{align}
For the last term, since $\sqrt{\wt{G}_{t,r}} = t^{\fm - 1}\sqrt{B(tr,\Theta)}$ with $B(tr,\Theta)\equiv\det(b_{ij}(tr,\Theta))$, taking the derivative,
\begin{align}
\left.\frac{\p}{\p r}\right|_{r = 1} \log \sqrt{\wt{G}_{t,r}} = \frac{t}{2}\cdot \left.\frac{\p}{\p r}\right|_{r = 1} \log B.
\end{align}
Then there exists a constant $C > 0$ such that 
 \begin{align}
	\left| \int_{\p \wt{\cB}_1^t(x_+)}\vr^{\fa}\sU^2 \left(\p_r\log\sqrt{\wt{G}_{t,r}}\right) d\tilde{\sigma}_{t, r} \right| \leq C \cdot t \cdot H_t(1).  
\end{align}
Combining the above, 
\begin{align}
	\left|H_t'(1) - (\fm - 1 + \fa) H_t(1)
	-2 \int_{\p\cB_1^t(x_+)}  \vr^{\fa}\sU \cdot  \sU_n d\sigma_t \right| \leq O(t) \cdot H_t(1),
\end{align}
which implies 
\begin{align}\label{e:derivative-log-H}
\left|\frac{H_t'(1)}{H_t(1)} - (\fm - 1 + \fa)  
	-2 \frac{\displaystyle{\int_{\p\cB_1^t(x_+)}  \vr^{\fa}\sU \cdot  \sU_n d\sigma_t}}{\displaystyle{ \int_{\p\cB_1^t(x_+)}  \vr^{\fa}\sU^2   d\sigma_t }} \right| \leq O(t).	
\end{align}

In the next step, we estimate the derivative of $I_t$. Let us split the integral 
\begin{align}\label{e:derivative-I_t}
	I_t'(r) = \int_{\p\cB_r^t}\vr^{\fa}|\nabla_{\bg_t}\sU|^2 \dst + t^2 \int_{\p\cB_r^t} \vr^{\fa} \cdot \mathcal{J} \cdot\sU^2 \dst \equiv \cI_D  +  \cI_Q.
\end{align} 
We pick a normal vector field $X \equiv s\nabla_{\bg_t} s$, where $s(x) \equiv d_{\bg_t}(x,x_+)$ is the distance to $x_+$ in terms of $\bg_t$. Then we have
\begin{align}
\begin{split}
		\cI_D(r) = &\ \frac{1}{r}\int_{\p\cB_r^t}\vr^{\fa}|\nabla_{\bg_t}\sU|^2 \langle X, r^{-1}X\rangle\dst 
	\nonumber\\
	= & \ \frac{1}{r}\int_{\cB_r^t}\Div\left(\vr^{\fa}|\nabla_{\bg_t}\sU|^2 X \right)\dvol_{\bg_t}
	\nonumber\\
	= & \ \frac{1}{r}\int_{\cB_r^t}\left(\vr^{\fa}|\nabla_{\bg_t}\sU|^2 \Div(X) \right) \dvol_{\bg_t} + \frac{1}{r}\int_{\cB_r^t}\langle\nabla_{\bg_t}\vr^{\fa}, X\rangle|\nabla_{\bg_t}\sU|^2\dvol_{\bg_t} \nonumber\\
	& \ + \frac{1}{r} \int_{\cB_r^t}\vr^{\fa} \langle\nabla_{\bg_t}|\nabla_{\bg_t}\sU|^2, X\rangle \dvol_{\bg_t}.
	\end{split}
\end{align}
By Lemma \ref{lem-rho-a}, the second term in $\cI_D(r)$ yields the bound 
\begin{align}\begin{split}
	\frac{1}{r}\int_{\cB_r^t}\langle\nabla_{\bg_t}\vr^{\fa}, X \rangle|\nabla_{\bg_t}\sU|^2\dvol_{\bg_t}	= & \ \int_{\cB_r^t} (\p_r\vr^{\fa})|\nabla_{\bg_t}\sU|^2 \dvol_{\bg_t} \nonumber\\
 = & \ \left(\frac{\fa}{r} + O(rt)\right)\int_{\cB_r^t} \vr^{\fa}|\nabla_{\bg_t}\sU|^2 \dvol_{\bg_t}.\end{split}
\end{align}
Now let us analyze the last term.
By local computations, 
\begin{align}
\begin{split}
		\langle\nabla_{\bg_t}\langle\nabla_{\bg_t} \sU, X\rangle,\nabla_{\bg_t} \sU\rangle  = &\ (\nabla_{\bg_t} \sU)\langle\nabla_{\bg_t} \sU, X\rangle \\
		= & \ \frac{1}{2}\langle\nabla_{\bg_t} |\nabla_{\bg_t} \sU|^2, X\rangle + \langle \nabla_{\bg_t} \sU, (\nabla_{\bg_t})_{\nabla_{\bg_t} \sU}X\rangle
		\\
		= &\ \frac{1}{2}\langle\nabla_{\bg_t}|\nabla_{\bg_t} \sU|^2, X\rangle + (\sU_r)^2 + r(\nabla_{\bg_t}^2 r)(\nabla_{\bg_t} \sU,\nabla_{\bg_t} \sU)
		\\
		= &\ \frac{1}{2}\langle\nabla_{\bg_t}|\nabla_{\bg_t} \sU|^2, X\rangle + (1 + O(rt)) |\nabla_{\bg_t} \sU|^2,
	\end{split}
\end{align}  
where we used the local estimate  \begin{align}
	|(X^i)_j - \delta_{ij}| \leq Crt. 
\end{align}

Therefore, 
\begin{align}\label{e:crossing}
\begin{split}
	\frac{1}{r} \int_{\cB_r^t}\vr^{\fa} \langle\nabla_{\bg_t}|\nabla_{\bg_t} \sU|^2, X\rangle \dvol_{\bg_t} = \frac{2}{r} \int_{\cB_r^t}\vr^{\fa}\left(\langle\nabla_{\bg_t}\langle\nabla_{\bg_t} \sU, X\rangle,\nabla_{\bg_t} \sU\rangle- (1 + O(rt)) |\nabla_{\bg_t} \sU|^2\right)\dvol_{\bg_t}.
	\end{split}	
\end{align}
 Since 
\begin{align}
	\Div(\vr^{\fa} \langle\nabla_{\bg_t} \sU, X\rangle \nabla_{\bg_t} \sU) = \vr^{\fa} \langle\nabla_{\bg_t}\langle\nabla_{\bg_t} \sU, X\rangle,\nabla_{\bg_t} \sU\rangle +  t^2\cdot \vr^{\fa} \cdot \cJ   \sU \langle\nabla_{\bg_t} \sU, X\rangle.
\end{align}
Applying divergence theorem to the first term in \eqref{e:crossing}, we have 
\begin{align}
\frac{2}{r} \int_{\cB_r^t}\vr^{\fa} \langle\nabla_{\bg_t}\langle\nabla_{\bg_t} \sU, X\rangle,\nabla_{\bg_t} \sU\rangle  \dvol_{\bg_t} = \frac{2}{r}\int_{\p \cB_r^t} \vr^{\fa} \sU_n^2 \dst  -  \frac{2\cdot t^2}{r}\int_{\cB_r^t}\vr^{\fa} \cdot \cJ   \sU \langle\nabla_{\bg_t} \sU, X\rangle \dvol_{\bg_t},
\end{align}
where $\sU_n \equiv \p \sU/\p n$ and $\p/\p n$ is the outward unit normal.

 Combining the above computations, 
\begin{align}
\left|\cI_D(1) - (\fm - 2 + \fa) \int_{\cB_1^t}\vr^{\fa}|\nabla_{\bg_t} \sU|^2 \dvol_{\bg_t} - 2\int_{\p\cB_1^t}\vr^{\fa} \sU_n^2  \dst \right|\leq Ct\int_{\cB_1^t}\vr^{\fa}|\nabla_{\bg_t} \sU|^2 \dvol_{\bg_t} .\end{align}
Finally, Corollary \ref{c:approximate-Dirichlet-energy} leads to 
\begin{align}
	\left|\cI_D(1) - (\fm - 2 + \fa) I_t(1) - 2\int_{\p\cB_1^t}\vr^{\fa} \sU_n^2\dst\right|\leq Ct\cdot I_t(1). \label{e:I_D-bound}
\end{align}

Now we estimate the term $\cI_Q$.  By Proposition \ref{p:eps-non-vanishing},  $H_t(r)\leq \frac{r I_t(r)}{\epsilon_0}$,
 which implies that 
\begin{align}\label{e:I_Q-bound}
	\cI_Q(1) \leq Ct^2 H_t(1) \leq  \frac{Ct^2}{\epsilon_0} I_t(1)  .
\end{align}
Combining \eqref{e:derivative-I_t}, \eqref{e:I_D-bound}, \eqref{e:I_Q-bound}, we have
\begin{align}
		\left|I_t'(1) - (\fm - 2 + \fa) I_t(1) - 2\int_{\p\cB_1^t}\vr^{\fa} \sU_n^2\dst\right|\leq Ct\cdot I_t(1). \end{align}
Notice that, by the divergence theorem, 
\begin{align}
 I_t(1) = \int_{\p\cB_1^t} \vr^{\fa} \sU \sU_n \dst,
\end{align}
 which implies 
 \begin{align}\label{e:derivative-log-I}
 \left|\frac{I_t'(1)}{I_t(1)} - (\fm - 2 + \fa)  - 2\frac{\displaystyle{\int_{\p\cB_1^t}\vr^{\fa} \sU_n^2\dst}}{\displaystyle{\int_{\p\cB_1^t} \vr^{\fa} \sU \sU_n \dst}}\right|\leq Ct. 	
 \end{align}

Finally, it follows from \eqref{e:derivative-log-I} and \eqref{e:derivative-log-H} that 
\begin{align}
\begin{split}
	\frac{\cN_t'(1)}{\cN_t(1)}
= 1 + \frac{I_t'(1)}{I_t(1)} - \frac{H_t'(1)}{H_t(1)}
\geq & \ 2\left(\frac{\displaystyle{\int_{\p\cB_1^t}\vr^{\fa}\sU_n^2\dst}}{\displaystyle{\int_{\p\cB_1^t}\vr^{\fa}\sU \sU_n\dst}} -\frac{\displaystyle{\int_{\p\cB_1^t}\vr^{\fa}\sU \sU_n\dst}}{\displaystyle{\int_{\p\cB_1^t}\vr^{\fa} \sU^2\dst}}\right) 	- C t
\\ 
\geq & \  -Ct.
\end{split}
\end{align}
This completes the proof. 
\end{proof}

\subsection{Quantitative symmetry and quantitative cone-splitting}

\label{ss:quantitative-symmetry-PE}

In this subsection, we will look at the behaviors of the solutions of \eqref{e:CS-extension-on-PE} on small scales in a quantitative way. Mainly, for a solution $U$ of the boundary value problem \eqref{e:CS-extension-on-PE}, we will summarize several quantitative rigidity results which are used to prove the uniform volume estimates for the quantitative strata of the singular set in Section \ref{ss:main-results-PE}.

Let us begin with the following proposition which shows how the energy boundedness property \eqref{e:N(0,1)-bounded-by-Lambda} on a large scale can pass to any smaller scales. This is quite important in the implementation of the quantitative analysis of the singular set.

\begin{proposition}[Uniform boundedness] \label{p:frequency-bound-U-PE} Let $p\in M^n$ and assume that $\cB_1(p_+) \subset \overline{X^{\fm}}$ has a smooth boundary. 
  Let $U: \fX^{\fm}\to \dR$ be an even solution of \eqref{e:CS-extension-on-PE}. 
Given any $\Lambda > 0$, there exist constants $N_0 = N_0 (\Lambda, \fm, \fa, \bg) > 0$ and $r_0 = r_0 (\Lambda, \fm, \fa, \bg) > 0$ 
such that if 
\begin{align}
\frac{\displaystyle{\int_{\cB_1(p_+)}\vr^{\fa}|\nabla_{\bar{g}}U|^2 \dvol_{\bg}}}{\displaystyle{\int_{\p\cB_1(p_+)}\vr^{\fa}U^2 \dsg}} \leq \Lambda, 	 	\label{e:N(0,1)-bounded-by-Lambda}
\end{align}
then for every $x \in \cB_{1/2}(p_+)$ and $r \in (0, r_0)$,
\begin{align}\cN_U(x, r) \leq N_0.\end{align}
\end{proposition}

\begin{proof} 
Since the condition and the conclusion of proposition are  scale-invariant in $U$, we assume that
\begin{align}
 \int_{\cB_1(p_+)}\vr^{\fa}|\nabla_{\bar{g}}U|^2 \dvol_{\bg} \leq \Lambda, \quad \int_{\p\cB_1(p_+)}\vr^{\fa}U^2 \dsg = 1.\label{e:unit-ball-U-normalization}	
\end{align}
Then Lemma \ref{l:L2-estimate-large-scale}
implies that 
$\|U\|_{H^{1,\fa}(\cB_1(p_+))} \leq C(\Lambda,\fm ,\fa , \bg)$. 
  It suffices to show that there exists $r_0 = r_0(\Lambda, \fm, \fa, \bg) > 0$ satisfying the following property: for any $r\in(0, r_0]$, there exists $\delta = \delta (r, \Lambda, \fm, \fa, \bg) > 0$ such that for every $x\in \cB_{1/2}(p_+)$, 
\begin{align}
\int_{\p\cB_r(x)} \vr^{\fa} U^2 \dvol_{\bg} > \delta. \label{e:positive-lower-bound-of-boundary-L2}
\end{align} 
Indeed, if \eqref{e:positive-lower-bound-of-boundary-L2} is proved, then we can choose $\delta_0 \equiv \delta(r_0,\Lambda, \fm, \fa, \bg))$ so that
\begin{align}
\begin{split}
	\mathcal{N}_U(x, r_0) 
= &\  \frac{\displaystyle{r_0\int_{\cB_{r_0}(x)}\vr^{\fa}(|\nabla_{\bg}U|^2 + \cJ_{\bg} U^2) \dvol_{\bg}}}{\displaystyle{\int_{\p\cB_{r_0}(x)} \vr^{\fa} U^2 \dsg}}	
 \\
\leq  \ &\delta_0^{-1}r_0 \cdot \left(1 + C_{\fm, \gamma}\sup\limits_{\overline{X^{\fm}}}|R_{\bg}|\right)\cdot  \left(\int_{\cB_1(p_+)}\vr^{\fa} (|\nabla_{\bg}U|^2  +   U^2) \dvol_{\bg}\right).
\end{split}
\end{align}
Applying Lemma \ref{l:L2-estimate-large-scale},
\begin{align}
\begin{split}
\mathcal{N}_U(x, r_0) 
\leq  &\ C_1(\Lambda, \fm ,  \fa, \bg) \cdot  \left( \int_{\cB_1(p_+)}\vr^{\fa}|\nabla_{\bg}U|^2 \dvol_{\bg} + \int_{\p\cB_1(p_+)} \vr^{\fa} U^2 \dsg \right) \\
\leq & \ C_1(\Lambda, \fm ,  \fa, \bg)    \cdot (\Lambda+1).
\end{split}
\end{align}
By Theorem \ref{t:almost-monotonicity}, $\mathcal{N}_U(x,r) \leq e^{C_2(\Lambda, \fm, \fa, \bg)r_0}(\Lambda + 1)$ for any $r\in(0,r_0)$, which proves the uniform boundedness.

Now let us prove \eqref{e:positive-lower-bound-of-boundary-L2}. The proof follows from a standard compactness/contradiction argument.	
Suppose that for some sufficiently small $r > 0$, there exist contradiction sequences: a sequence of solutions $U_j$ of \eqref{e:CS-extension-on-PE}, a  sequence of points $x_j \in \cB_{1/2}(p_+)$, and a sequence $\delta_j \to 0$ such that  $\|U_j\|_{H^{1,\fa}(\cB_1(p_+))} \leq C$ and 
$H_U(x_j,r)\leq \delta_j \to 0$. 
This implies that, passing to a subsequence, $U_j$ converges to $U_{\infty}$ weakly in the $H^{1,\fa}(\cB_1(p_+))$-topology and 
\begin{align}\label{eq-U-2-P0}
	\int_{\p\cB_{r_0}(x_{\infty})}\vr^{\fa} U_{\infty}^2 \dsg 
 = 0,\end{align}
where $x_j\to x_{\infty} \in \cB_{1/2}(p_+)$. Therefore, $U_{\infty} \equiv 0 $ on $\p\cB_r(x_{\infty})$. 
Since $U_{\infty}$ solves \eqref{e:CS-extension-on-PE} and $U_{\infty}\in C^{n,\alpha}(\cB_1(p_+))$, 
 it follows that \begin{align}
 \int_{\cB_r(x_{\infty})}  \vr^{\fa}(|\nabla_{\bg} U_{\infty}|^2 + \cJ_{\bg}	U_{\infty}^2 ) \dvol_{\bg} = \int_{\p \cB_r(x_{\infty})} \vr^{\fa} \cdot U_{\infty} \cdot \frac{\p U_{\infty}}{\p n}\dsg = 0,
 \end{align}
 which leads to 
 \begin{align}
D_U(x_{\infty}, r)	 
 \leq C(\fm,\fa,\bg)\int_{\cB_r(x_{\infty})}  \vr^{\fa} U^2 \dvol_{\bg}
= C(\fm,\fa,\bg)\left(r H_U(x_{\infty}, r) + r^2D_U(x_{\infty}, r)\right).
 \end{align}
 If $r > 0$ is chosen such that $C(\fm, \fa, \bg) r^2 < \frac{1}{2}$, then $D_U(x_{\infty}, r) \equiv 0$, and hence $U_{\infty} \equiv 0$ on $\cB_r(x_{\infty})$.

Notice that $u_{\infty} = \vr^{n-s} U_{\infty}$ satisfies 
\begin{align}
	\Delta_{g_+} u_{\infty} + s(n-s) u_{\infty} = 0, \quad s = \frac{n}{2} + \gamma,
\end{align}
in an open subset $\mathcal{O}\subset X^{\fm}$. Here $\mathcal{O}$ is the conformal image of $\cB_1(p_+)$ in terms of $g_+  =  \vr^{-2} \bg$. 
By the standard unique continuation, $u_{\infty} \equiv 0$ on $\mathcal{O}$, which  implies $U_{\infty}\equiv 0$ on $ \cB_1(p_+)$. But this contradicts the normalization condition \eqref{e:unit-ball-U-normalization}, which completes the proof of \eqref{e:positive-lower-bound-of-boundary-L2}.
\end{proof}

Next, we will list a series of quantitative rigidity results which replace the corresponding results in Sections \ref{ss:quantitative-rigidity-Euclidean} and \ref{ss:quantitative-cone-splitting}.

The quantitative rigidity below can be proved by standard contradiction and compactness arguments, which is a replacement of Proposition \ref{p:quantitative-symmetry} in the general Poincar\'e-Einstein manifold. 
\begin{proposition}
[Quantitative rigidity] \label{p:quantitative-symmetry-U-PE} For fixed $\fa \in (-1,1)$, let $U$ be the unique solution to \eqref{e:CS-extension-on-PE}. 
 Given $\epsilon > 0$, $\gamma \in (0,1)$, and $\Lambda > 0$, there exist $\delta = \delta(\epsilon,\Lambda, \gamma, n, \fa,  g_+, h) > 0$, $r_c = r_c(\epsilon,\Lambda,\gamma, n, \fa,  g_+, h) > 0$, and $C = C(\epsilon,\Lambda, \gamma, n,\fa,  g_+, h) > 0$ such that if 
\begin{align}
\frac{\displaystyle{\int_{\cB_1(p_+)}\vr^{\fa}|\nabla_{\bar{g}}U|^2 \dvol_{\bg}}}{\displaystyle{\int_{\p\cB_1(p_+)}\vr^{\fa}U^2 \dvol_{\bg} }} \leq \Lambda, 	\end{align} and for $s \in (0, r_c)$,
	\begin{align}
		\mathcal{N}(x, s) -  		\mathcal{N}(x, s\gamma)  < \delta,
	\end{align}
	then $U$ is $(0, \epsilon , s , \fa)$-symmetric at $x \in B_{1/2}(p)$.
\end{proposition}

Applying the almost monotonicity in Theorem \ref{t:almost-monotonicity} and contradiction arguments, we have the following {\it inductive cone-splitting principle} which replaces Corollary  \ref{c:inductive-splitting}.

\begin{proposition}
	[Inductive cone-splitting] \label{p:inductive-cone-splitting-PE}
		For any fixed $n\geq 2$, $\epsilon>0$, $\gamma \in (0,1)$, $\rho>0$, $r\in (0,1)$, $k\in\{0,1,\ldots, n-2\}$, $\Lambda>0$,  there exist $\delta=\delta(\epsilon, n, \rho, \Lambda, \fa) > 0$
and $r_{\#} = r_{\#}(\epsilon, n, \rho, \Lambda, \fa)  > 0$		
		 such that the following holds. Let $U$ be the unique solution to \eqref{e:CS-extension-on-PE} with 
\begin{align}
\frac{\displaystyle{\int_{\cB_1(p_+)}\vr^{\fa}|\nabla_{\bar{g}}U|^2 \dvol_{\bg}}}{\displaystyle{\int_{\p\cB_1(p_+)}\vr^{\fa}U^2 \dvol_{\bg} }} \leq \Lambda, \end{align}
If for some $s\in (0,r_{\#})$ and $x \in B_{1/2}(p)$,
\begin{enumerate}
\item $U$ is $(0,\delta, s\gamma ,\fa)$-symmetric at $x$,
\item for any vector space $V$ of dimension $\leq k$, there exists some point $z\in B_r(x)\setminus B_{\rho}(V)$ such that $U$ is  $(0,\delta,s\gamma,\fa)$-symmetric at $z$, 
\end{enumerate}
then $U$ is also $(k+1,\epsilon,s,\fa)$-symmetric at $x$.

\end{proposition}

The following lemma is used in estimating the dimension and measure of the singular set which is a replacement of Lemma \ref{l:almost-(m-1)-symmetric}. 
\begin{lemma}
	Let  $U$   be the unique even solution of \eqref{e:CS-extension-on-PE} such that \begin{align}
\frac{\displaystyle{\int_{\cB_1(p_+)}\vr^{\fa}|\nabla_{\bar{g}}U|^2 \dvol_{\bg}}}{\displaystyle{\int_{\p\cB_1(p_+)}\vr^{\fa}U^2 \dvol_{\bg} }} \leq \Lambda. \end{align}
  Then for every $\epsilon > 0$, $k\in\dN$, and $\alpha\in(0,1)$, there exists 
	$\eta = \eta (\epsilon, n, k, \Lambda, \fa) > 0$ such that if 
	$U$ is $(\fm - 1, \eta, r, \fa)$-symmetric at $p_+$, then 
	\begin{align}
		\|\cT_{p_+,r} U - L \|_{C^{1,\alpha}(\cB_{1/2}(\bo_+))} < \epsilon,
	\end{align}
	where $L$ is a linear polynomial with $\frac{1}{r^{\fa}}\fint_{\p\cB_1(\bo_+)}|L|^2 d\sigma_0 = 1$.
	In particular, choosing any positive number $\epsilon < \epsilon_0 \equiv |\nabla L|/3$, we have $r_{p_+} \geq r$ if $U$ is $(\fm - 1, \epsilon, r_{p_+} , \fa)$ at $p_+$. \end{lemma}

\subsection{Volume and measure estimates for the singular set}\label{ss:main-results-PE}

In this subsection, we prove the main results of this paper. 
To start with, we recall the equations together with basic settings of the underlying Poincar\'e-Einstein manifolds.

For any given $n\geq 2$, let $\fm \equiv n + 1$ and let $(X^{\fm}, g_+)$ be a complete Poincar\'e-Einstein manifold with $\Ric_{g_+} \equiv - n g_+$. Assume that a closed manifold $M^n$ equipped with a conformal class $[h]$ is 
the conformal infinity of $(X^{\fm},g_+)$.
  We choose Fefferman-Graham's compactification $\bg = \vr^2 g_+ = e^{2w} g_+$ of $(X^{\fm}, g_+)$ as described previously in Section \ref{ss:PE-metrics} such that $-\Delta_{g_+} w = n$ on $X^{\fm}$ and $\bg|_{M^n} = h$. In addition, we also assume that $\bg$ satisfies Assumptions (1) - (4) on $\overline{X^{\fm}}$ as introduced at the beginning of Section \ref{s:results-on-PE}. We also denote by $\fX^{\fm} \equiv \overline{X^{\fm}}\bigcup\limits_{M^n} \overline{X^{\fm}}$ the doubling of $(\overline{X^{\fm}},\bg)$ along the totally geodesic boundary $(M^n, h)$.
  
In the following, we will collect our main results in the above setting. For any fixed $\gamma\in(0,1)$, let $f \in C^{\infty}(M^n)$ satisfy $P_{2\gamma}(f) = 0$ on $B_1(p)  \subset M^n$, where $P_{2\gamma}$ is the fractional GJMS-operator. Let $\fa \equiv 1 - 2\gamma$ and let $U\in H^{1,\fa}(\fX^{\fm})$ be an even solution of the following boundary value problem:
 \begin{align}\label{e:PE-boundary-value-problem}
 \begin{cases}
  	-\Div_{\bg}(\vr^{\fa} \nabla_{\bg} U) + \vr^{\fa} \cJ_{\bg} U = 0 & \text{in} \ \fX^{\fm}
	 	\\
U = f & \text{on} \ M^n,\\
P_{2\gamma} f = 0 &\text{on}\ B_1(p) \subset M^n,
 \end{cases}
 \end{align}
where $\cJ_{\bg}\equiv C_{n,\gamma}R_{\bg}$.

\begin{theorem}\label{t:volume-estimate-general}
For any $j\in\dN$,  $\epsilon>0$, $\Lambda > 0$, $k\leq \fm-2$, and $\fa \in (-1,1)$,  there exist  positive constants  $\mu = \mu(\epsilon,\fm,\Lambda, \fa, \bg)\in(0,1)$ and $C = C(\epsilon,\fm,\Lambda,  \fa, \bg)>0$ such that the following property holds. Let $U$ be the unique even solution of \eqref{e:PE-boundary-value-problem} with 
\begin{align}
\frac{\displaystyle{\int_{\cB_1(p_+)}\vr^{\fa}|\nabla_{\bar{g}}U|^2 \dvol_{\bg}}}{\displaystyle{\int_{\p\cB_1(p_+)}\vr^{\fa}U^2 \dvol_{\bg} }} \leq \Lambda. \end{align}
Then 
\begin{align}
\Vol_{\bg}(B_{\mu^j}(\cS_{\epsilon,\mu^j}^k(U))\cap B_{1/2}(p)) \leq C \cdot (\mu^j)^{\fm - k -\epsilon}.	
\end{align}

\end{theorem}

\begin{proof}
The strategy of proof is the same as that of Theorem  \ref{p:effective-covering}. Here we only mention the technical difference.	

In the proof of Theorem  \ref{p:effective-covering}, the uniform 
boundedness estimate in Lemma \ref{l:frequency-control-large-to-small} is replaced by Proposition 	\ref{p:frequency-bound-U-PE}, the quantitative rigidity in Proposition \ref{p:quantitative-symmetry} is replaced by Proposition \ref{p:quantitative-symmetry-U-PE}, and 
the inductive splitting property in Corollary \ref{c:inductive-splitting} is replaced by 
	Proposition \ref{p:inductive-cone-splitting-PE}.
\end{proof}

Next, we will exhibit two results on the Hausdorff measure estimates. 
Regarding the $(\fm - 1)$-dimensional Hausdorff measure estimate on the zero set of the solution $U$. We have the following.
 \begin{theorem}[Hausdorff measure of the zero set]\label{t:zero-set-PE}
For $p\in M^n$, let $U$ be the unique even solution of \eqref{e:PE-boundary-value-problem} that satisfies 
\begin{align}
\frac{\displaystyle{\int_{\cB_1(p_+)}\vr^{\fa}|\nabla_{\bar{g}}U|^2 \dvol_{\bg}}}{\displaystyle{\int_{\p\cB_1(p_+)}\vr^{\fa}U^2 \dvol_{\bg} }} \leq \Lambda. \label{e:bounded-energy-large-scale} \end{align}
There exists a constant $C=C(\fm,\Lambda,\fa,\bg)>0$ such that 
\begin{align}
\cH^{\fm - 1}(\cZ(U) \cap B_{1/2}(p)) \leq C.	
\end{align}

 \end{theorem}
The uniformly elliptic version of Theorem \ref{t:zero-set-PE} was proved in \cite{HL-geometric-measure, Lin}. The main technical tools employed there include the uniform doubling property (guaranteed by the almost monotonicity of frequency), compactness of solutions with bounded energy, and harmonic approximation. In our setting, all these tools are available. Since the main strategy and arguments are the same, we omit the proof of Theorem \ref{t:zero-set-PE}.

The last result is the $(\fm - 2)$-dimensional Hausdorff measure estimate for the singular set of the solution $U$ of \eqref{e:PE-boundary-value-problem}.
\begin{theorem}
	[Hausdorff measure estimate] \label{t:Hausdorff-measure-estimate-PE}
For $p\in M^n$, let $U$ be the unique solution of \eqref{e:PE-boundary-value-problem} with 
\begin{align}
\frac{\displaystyle{\int_{\cB_1(p_+)}\vr^{\fa}|\nabla_{\bar{g}}U|^2 \dvol_{\bg}}}{\displaystyle{\int_{\p\cB_1(p_+)}\vr^{\fa}U^2 \dvol_{\bg} }} \leq \Lambda.  \end{align}
There exists a constant $C=C(\fm,\Lambda,\fa,\bg)>0$ such that 
\begin{align}
\cH^{\fm - 2}(\cS(U) \cap B_{1/2}(p)) \leq C.	
\end{align}
	\end{theorem}
	
The proof of Theorem \ref{t:Hausdorff-measure-estimate-PE} is the same as Theorem \ref{t:Hausdorff-measure-estimate-Euclidean}. The main step is to establish the following $\epsilon$-regularity theorem which replaces Theorem \ref{t:eps-reg-Euclidean}.
		
	\begin{theorem}
[$\epsilon$-regularity] \label{t:eps-reg-PE} 
For $p\in M^n$, let $U$ be the unique solution of \eqref{e:PE-boundary-value-problem} that satisfies \eqref{e:bounded-energy-large-scale}. Then there exist constants $\epsilon(\fm,\fa,\Lambda,\bg) > 0$
and $\bar{r}(\fm,\fa,\Lambda,\bg) > 0$
such that if there exists a normalized 
homogeneous polynomial $P$ with $(\fm - 2)$-symmetries such that 
\begin{align}
\fint_{\p\cB_1(\bo_+)} \vr^{\fa} P^2 \dsg = 1\quad \text{and} \quad 
\int_{\p\cB_1(\bo_+)}\vr^{\fa}(U_{x_+, r} - P)^2 \dsg \leq \epsilon,	\label{e:L2-approximation-PE}
\end{align}
for some $x_+ \in B_{1/2}(p)$, then for all $r\leq \bar{r}$, 
\begin{align}
\cH^{\fm - 2}(\cS(U) \cap B_r(x_+))
\leq C(\fm,\fa,\Lambda,\bg) r^{\fm - 2}.	
\end{align}
\end{theorem}
To prove Theorem \ref{t:eps-reg-PE}, it is the same as Theorem \ref{t:eps-reg-Euclidean} that one needs to upgrade the approximation \eqref{e:L2-approximation-PE} and establish a higher order estimate. That is,	
\begin{proposition}
[Higher order approximation] \label{p:higher-order-approximation-PE} Let $U$ be the unique solution of \eqref{e:PE-boundary-value-problem} that satisfies \eqref{e:bounded-energy-large-scale}. For every $\epsilon > 0$, there exists $\delta (\epsilon, \fm,\fa,\Lambda,\bg) > 0$ and $r_0 (\epsilon, \fm,\fa,\Lambda,\bg) > 0$ such that if there exists a normalized 
homogeneous polynomial $P$ with $(\fm - 2)$-symmetries that satisfies 
\begin{align}
\fint_{\p\cB_1(\bo_+)} \vr^{\fa} P^2 \dsg = 1\quad \text{and}    
\quad 
\int_{\p\cB_1(\bo_+)}\vr^{\fa}(U_{x_+, r} - P)^2 \dsg \leq \epsilon,	
\end{align}
for some $x_+ \in \cZ(U) \cap \cB_{1/2}(p_+)$ and $r\in(0,r_0)$,
then there exists a function $\tU\in C^{2\fm^2}(\cB_{2r/3}(x_+))$ such that $\tU \equiv U$ in $\overline{\cB_{2r/3}^+(x_+)}$ and 
\begin{align}
\|\tU_{x_+, r} - P\|_{C^{2\fm^2}(\cB_{2/3}(\bo_+))} < \epsilon. \label{e:extension-estimate}	
\end{align}  
\end{proposition}

\begin{remark}
Note that the doubling metric $\bg$ on $\fX^{\fm}$ is only $C^{\fm - 2,1}$ in general. Especially, when $n$ is odd and the {\it global term} in the normal expansion of $\bg$ is non-vanishing, the $C^{\fm - 2,1}$-smoothness of $\bg$ is optimal.  Since a key tool, theorem \ref{t:HHL-stability}, in proving Theorem \ref{t:eps-reg-PE} requires a very high order approximation of $U$ (up to $2\fm^2$), \eqref{e:L2-approximation-PE} cannot be upgraded to this by directly using elliptic estimates. Proposition \ref{p:higher-order-approximation-PE} constructs a sufficiently regular replacement of $U$ which overcomes insufficient regularity of the metric $\bg$.	
\end{remark}

Once this proposition is proved, the rest of the proof of Theorem \ref{t:eps-reg-PE} is identical to Theorem \ref{t:eps-reg-Euclidean}.

\begin{proof}[Proof of Proposition \ref{p:higher-order-approximation-PE}]
	
To construct the extension $\tU$, let us first analyze the regularity of $U$ in the upper half space. In this proposition, the estimate is required to pass to a small scale $r_0 > 0$ for carrying out elliptic estimates. Indeed, one needs to pass to a small scale on which the elliptic operator is sufficiently close to the model operator in the Euclidean case.     
In order to avoid unnecessary technical complications, we just set $r_0 = 1$ and assume that the distortion between our interested elliptic operator and the (degenerate) Euclidean Laplacian is sufficiently small.  

First, applying elliptic regularity, for any $k\in\dZ_+$, there exists $C_k = C_k (\fm , \fa, \bg) > 0$ such that
\begin{align}
\|U - P\|_{C^k \left( \overline{\cB_{2/3}^+(\bo_+)} \right)} < C_k\delta.	 \label{e:derivative-estimate-order-k}
\end{align}
Notice that, by the standard compactness argument, the homogeneous polynomial $P$ can be chosen as a solution in the model case of \eqref{e:PE-boundary-value-problem},  
$P$ is even in $y$.	 Moreover, there exists a uniform constant $N_0 = N_0(\Lambda, \fm, \fa, \bg) > 0$ such that $d_0 \leq N_0$ since the frequency of $P$ is uniformly bounded due to the compactness.
	
Let us pick a sufficiently large integer $k_0\in\dZ_+$ (to be fixed later) which is determined by the structure of the Taylor expansions of $U$ and $P$. For any $y < \frac{2}{3}$, applying Taylor's theorem in the $y$-direction,
we have that
\begin{align}
	U  =  U_{k_0}  + R_{k_0} \equiv \sum\limits_{\ell = 0}^{k_0} c_{\ell} y^{\ell} + R_{k_0},\quad  R_{k_0} = \frac{U^{(k_0+1)}(\xi y, \cdot)}{(k_0 + 1)!}\cdot y^{k_0 + 1}, \label{e:U-expansion-in-y}\end{align}
where each $c_{\ell} = c_{\ell}(\cdot)$ is independent of $y$, $0< \xi < 1$, and $U^{(k_0+1)}$ denotes the $(k_0+1)^{\text{th}}$-derivative of $U$ in the $y$ direction. 
For the homogeneous polynomial $P$, we write 
$P = \sum\limits_{\substack{\ell = 0, \\ \ell \ \text{is even}}}^{m_0} p_{\ell} y^{\ell}$,
where $m_0\leq d_0$ is some even integer and  every $p_{\ell} = p_{\ell}(\cdot)$ is a polynomial in the directions transversal to the $y$-direction.

Next,
we write   
\begin{align}U - P = (U_{k_0} - P) + R_{k_0} = \sum\limits_{0\leq \ell\leq k_0}\hat{c}_{\ell} y^{\ell} + \widehat{R}_{k_0},\end{align}
where 
$\hat{c}_{\ell} \equiv c_{\ell} - p_{\ell} = \frac{(U-P)^{(\ell)}(0, \cdot)}{\ell!}$ and $\widehat{R}_{k_0}$ is the remainder in Taylor's theorem. 
Now  we are ready to choose the truncation parameter $k_0$.  
For any fixed $\epsilon > 0$, let us choose 
\begin{align}k_0 \equiv \max\left\{2\fm^2, N_0 + 1, \frac{\log\epsilon}{\log(2/3)}\right\}.\end{align} Immediately, $y^{k_0 + 1} \leq (2/3)^{k_0 + 1} < \epsilon$. For such an integer $k_0$, one can further choose a sufficiently small $\delta = \delta(\epsilon, \fm, \fa, \bg) > 0$ such that if \eqref{e:L2-approximation-PE} holds, then 
\begin{align}\|\nabla^{\nu} (U - P)\|_{C^0(\cB_{2/3}^+(\bo_+))} < \frac{\epsilon}{2 (k_0 + 1)}, \quad \forall \ 0\leq \nu\leq 2\fm^2 + k_0 + 1. \label{e:fixed-derivative-error}
\end{align}
We also make some simple observations. 
\begin{enumerate}
	\item Since $\deg(P)$ is even, it follows that $\hat{c}_{\ell}  = c_{\ell}$ for any odd number $0\leq \ell \leq k_0$.
	\item By our choice, $k_0 > N_0 \geq  d_0$, which gives $\widehat{R}_{k_0} \equiv R_{k_0}$. 
\end{enumerate}
By observation (2) and \eqref{e:fixed-derivative-error},
\begin{align} \|R_{k_0}\|_{C^{2\fm^2}(\cB_{2/3}^+(\bo_+))} = \|\widehat{R}_{k_0}\|_{C^{2\fm^2}(\cB_{2/3}^+(\bo_+))}\leq \frac{\|\nabla^{2\fm^2 + k_0 + 1} (U - P)\|_{C^0(\cB_{2/3}^+(\bo_+))}}{(k_0 + 1)!}\cdot y^{k_0 + 1}. \label{e:remainder-bound}
\end{align}
So the remainder yields
\begin{align}
 \|R_{k_0}\|_{C^{2\fm^2}(\cB_{2/3}^+(\bo_+))}  < \frac{\epsilon^2}{10}.
\end{align}

Finally, we define
 \begin{align}
 \tU \equiv	\begin{cases}
 		U, & y \geq 0, 
 		\\
 		&
 		\\
 		\sum\limits_{\substack{0\leq \ell\leq k_0, \\ \ell \ \text{is odd}}}(-c_{\ell}) y^{\ell} +\sum\limits_{\substack{0\leq \ell\leq k_0, \\ \ell \ \text{is even}}} c_{\ell} y^{\ell} + R_{k_0}, & y < 0.
 	\end{cases}
 \end{align}
 	Clearly, $\tU \in C^{2\fm^2}(\cB_1(\bo_+))$. The rest of the proof is to verify the approximation \eqref{e:extension-estimate} in the lower half space.
 	In fact, 	
 	one can write \begin{align}
\tU - P = \sum\limits_{\substack{0\leq \ell\leq k_0, \\ \ell \ \text{is odd}}}(-\hat{c}_{\ell}) y^{\ell} +\sum\limits_{\substack{0\leq \ell\leq k_0, \\ \ell \ \text{is even}}} \hat{c}_{\ell} y^{\ell} + \widehat{R}_{k_0}  \quad y < 0.  	\end{align}
 Then 
 \begin{align}
 \begin{split}
 	\|\tU - P\|_{C^{2\fm^2}(\cB_{2/3}^-(\bo_+))}
 	\leq &\ ( k_0 + 1 ) \cdot \max\limits_{0\leq \nu k_0}\|\nabla^{\nu} (U - P)\|_{C^0(\cB_{2/3}^+(\bo_+))}  +  \|R_{k_0}\|_{C^{2\fm^2}(\cB_{2/3}^+(\bo_+))} 
 	\\
 	 \leq  &\  \frac{\epsilon}{2} + \frac{\epsilon^2}{10} < \epsilon. 
 	\end{split} 
 \end{align} 	
 	Here we used \eqref{e:fixed-derivative-error} and \eqref{e:remainder-bound}.
 	The proof is complete. 
 	\end{proof}

For the Minkowski type estimates and Hausdorff  measure estimates for $\cS(f)$ on the boundary, one needs to split $\cS(f)$ in a horizontal part and a mixed part as in Section \ref{ss:boundary-estimate}
The main results have been stated in Section \ref{ss:main-results}, while the technical refinements can be stated in a similar way as the relevant results in Section \ref{ss:boundary-estimate}.

\begin{remark}
The techniques used to prove the previous theorems allow actually to prove a rectifiability result for the singular set $\mathcal S(U)$ and its stratification $\mathcal S^k(U)$. More precisely, one can prove that the following decomposition holds 
\begin{align}\mathcal S^k(U)= \bigcup_{j=1}^\infty{E_j},
\end{align}
for   some sets $E_j$. Those sets can be shown to be closed, using Monneau's type monotonicity formulae for instance;  see \cite{STT}. The previous characterization of the each singular stratum comes from a growth lemma together with a non-degeneracy lemma. The idea is that singular  strata can be described also as the blow-ups having homogeneity $k$. This part follows from the Almgren's frequency function. The classification and behaviour of blow-ups is then done by the introduction of a suitable Monneau monotonicity formula originally introduced for the obstacle problem; see \cite{monneau}. 
Therefore as in \cite[Theorem 1.3.8]{GarofaloPetrosyan}, the rectifiability of each stratum follows from Whitney's extension theorem and the implicit function theorem; see also \cite[Theorem 7.7]{STT}. This approach describes just another way of stratifying the singular set. 
\end{remark}

\bibliographystyle{amsalpha}

\bibliography{PE}

 \end{document}